\documentclass[a4paper,11pt]{article}
\usepackage[latin1]{inputenc}
\usepackage[T1]{fontenc}
\usepackage[british,UKenglish,USenglish,american]{babel}
\usepackage{amsmath}
\usepackage{mathtools}
\usepackage{amsfonts}
\usepackage{hyperref}
\usepackage[T1]{fontenc}
\usepackage[latin1]{inputenc}
\usepackage{ntheorem}
\usepackage{theorem}
\usepackage[normalem]{ulem}
\usepackage{graphics}
\usepackage{makeidx}
\usepackage{fancyhdr}
\usepackage{hyperref}
\usepackage{tikz}
\usepackage{todonotes}
\usepackage{verbatim}
\usepackage{bbm}
\usepackage{bm}
\usepackage{amssymb}
\usepackage{color}
\usepackage{relsize}
\numberwithin{equation}{section}
\usepackage{amssymb}
\newcommand{\be}{\begin{equation}}
	\newcommand{\ee}{\end{equation}}
\newcommand{\benn}{\begin{equation*}}
	\newcommand{\eenn}{\end{equation*}}
\newcommand{\bea}{\begin{eqnarray}}
	\newcommand{\eea}{\end{eqnarray}}

\newcommand{\beann}{\begin{eqnarray*}}
	\newcommand{\eeann}{\end{eqnarray*}}

\newtheorem{theorem}{Theorem}[section]
\newtheorem{proposition}[theorem]{Proposition}
\newtheorem{corollary}[theorem]{Corollary}
\newtheorem{lemma}[theorem]{Lemma}

\newtheorem{definition}[theorem]{Definition}
\theorembodyfont{\rm}
\newtheorem{remark}[theorem]{Remark}
\newtheorem{notation}[theorem]{Notation}
\newcommand{\myquad}[1][1]{\hspace*{#1em}\ignorespaces}
\newtheorem{example}[theorem]{Example}
\newtheorem{assumptions}[theorem]{Assumption}

\newcommand{\proof}{\noindent{\bf Proof.\, }}
\newcommand{\qed}{\hfill $\Box$\smallskip\newline}
\newcommand{\E}{\noindent{$\mathbb{E}$ \ }}

\def\R{\mathbb{R}}

\def\N{\mathbb{N}}

\def\P{\mathbb{P}}
\def\E{\mathbb{E}}
\def\T{\mathbb{T}}

\def\P{\mathbb{P}}
\def\cB{\mathcal{B}}

\def\cF{\mathcal{F}}

\def\cP{\mathcal{P}}

\def\cR{\mathcal{R}}

\def\cX{\mathcal{X}}

\def\txtd{{\textnormal{d}}}

\def\txtD{{\textnormal{D}}}

\allowdisplaybreaks

\title{On the negativity of the top Lyapunov exponent for stochastic differential equations driven by fractional Brownian motion}
\author{Alexandra Blessing Neam\c tu~\thanks{Alexandra Blessing Neam\c tu.  Department of Mathematics and Statistics, University of Konstanz,
		Universit\"atsstra\ss{}e~10, 78464 Konstanz, Germany. E-Mail: alexandra.blessing@uni-konstanz.de}~~~~and~~~~Mazyar Ghani Varzaneh \thanks{Mazyar Ghani Varzaneh. Department of Mathematics and Statistics, University of Konstanz,
		Universit\"atsstra\ss{}e~10, 78464 Konstanz, Germany. E-Mail: mazyar.ghanivarzaneh@uni-konstanz.de } }
\begin{document}
	\maketitle
	\begin{abstract}
		We provide sign information for the top Lyapunov exponent for a stochastic differential equation driven by fractional Brownian motion.~To this aim we analyze the stochastic dynamical system generated by such an equation, obtain a random dynamical system and construct an appropriate invariant measure.~Suitable estimates for its  density together with Birkhoff's ergodic theorem imply the negativity of the top Lyapunov exponent by increasing the noise intensity. 
		
	\end{abstract}
	\tableofcontents

\section{Introduction}
The main goal of this work is to develop a machinery that allows one to investigate the asymptotic behavior of stochastic differential equations perturbed by non-Markovian noise. To this aim, we consider  the SDE
\begin{align}\label{sde:intro}
	\begin{cases}
		& \txtd Y_ t = F(Y_t)~\txtd t + \sigma\txtd B^H_t\\
		&Y_0=x\in \R^d,
	\end{cases}
\end{align}
where $d\geq 1$ and $(B^H_t)_{t\geq 0}$ is a $d$-dimensional fractional Brownian motion (fbm) with Hurst index $H\in(0,1)$.~In particular, under suitable assumptions on the coefficients we prove that~\eqref{sde:intro} has a negative Lyapunov exponent.~Our approach involves the construction of a random dynamical system (RDS) associated to~\eqref{sde:intro} given a stochastic dynamical system (SDS) as introduced in~\cite{Hai05}.~The RDS framework is useful to define concepts as (in)stability, attractors, invariant manifolds, synchronization (by noise) or chaos.~One key indicator for synchronization / chaos is  given by the top Lyapunov exponent of the underlying system.~However, deriving sign information for the top Lyapunov exponent of SDEs/SPDEs is a challenging task which recently received lots of attention~\cite{GT24,HR,BBPS2}.~In particular, a positive top Lyapunov exponent was established for different models arising from fluid dynamics such as  Lagrangian flows~\cite{BBPS1} or Galerkin truncations for 2D Navier-Stokes~\cite{BBPS2} indicating their chaotic behavior.\\%It is worth mentioning that Lyapunov exponents are also key indicators for measuring chaos, as they are closely related to the system's metric entropy through Pesin's theorem; see~\cite{FL85,BY17}.\\
% \todo{the first reference has typos and the style is not consistent,\textcolor{red}{I made it ok now, I did not detect Typo rather the journal name was not written like the rest} } \\ % \textcolor{red}{The aim of this sentence is to make the paper more accessible to the physics community, which, as you know, places great importance on entropy.}

Here we are interested in the opposite phenomenon, i.e.~when the top Lyapunov exponent is negative.~To the best of our knowledge, all results on Lyapunov exponents and RDS are formulated in a Markovian setting heavily relying on the invariant measure of~\eqref{sde:intro} which solves a Fokker-Planck equation or on ergodic properties of the projective process using Furstenberg's theorem.~All these tools break down for fbm.~Therefore the main challenge of this work is to obtain an invariant measure for~\eqref{sde:intro} based on the approach in~\cite{Hai05} together with a representation for its density.~One of our main results, Theorem~\ref{MAIN} establishes that the top Lyapunov exponent of~\eqref{sde:intro} can become negative by taking the noise intensity $\sigma$ sufficiently large.~To the best of our knowledge, this is the first work that provides sign information for Lyapunov exponents in a non-Markovian setting.~The statement of Theorem~\ref{MAIN} aligns with the results of~\cite{FGS16a} obtained for a similar type of SDE driven by a Brownian motion showing that its invariant measure  flattens for large values of $\sigma$.~This fact can be inferred from the corresponding Fokker-Planck equation satisfied by the invariant measure of the SDE.~As already stated, such arguments are not available for SDEs driven by fbm.\\

Therefore we follow the following strategy: based on the results in~\cite{Hai05} one can construct a SDS associated to~\eqref{sde:intro}.~Given this we show how to obtain a RDS and construct an invariant measure using disintegration.~In this framework, the multiplicative ergodic theorem~\cite[Theorem 3.4.1]{Arn98} entails the existence of Lyapunov exponents.~To the best of our knowledge this is the first work that constructs a RDS from a SDS yielding the applicability of the multiplicative ergodic theorem in an non-Markovian setting.~Furthermore, using Birkhoff's ergodic theorem, we obtain a representation for the top Lyapunov exponent in terms of the stationary density of a rescaled version of \eqref{sde:intro}. Following \cite{LPS23}, we derive Gaussian upper and lower bounds for this stationary density.~In comparison to the results in~\cite{LPS23}, due to the rescaling, both drift and diffusion coefficients depend on the intensity of the noise $\sigma$. This has to be increased in order to make the top Lyapunov exponent of~\eqref{sde:intro} negative. Therefore the precise dependence of the stationary density on $\sigma$ is essential for our aims. For Gaussian bounds for the density of singular SDEs with additive fractional Brownian motion we refer to~\cite{density1} and~\cite{density2} and emphasize that for our aims the density of the invariant measure of the corresponding SDE is relevant.\\

At an informal level, our main result reads as follows. 

\begin{theorem}[Main result, see Theorem \ref{thm:main}] Let $F$ be eventually strictly monotone, see Assumption~\ref{Drift2}.~Then the top Lyapunov exponent of the SDE~\eqref{sde:intro} becomes negative choosing the noise intensity $\sigma$ large enough.
\end{theorem}
In conclusion, the connection between SDS and RDS developed in this work is of independent interest and can be applied to gain further dynamical insights for SDEs perturbed by fbm such as chaos, (singleton) attractors or invariant manifolds.~We leave such aspects for future works and refer to~\cite{BG26} for a weak form of synchronization by noise for~\eqref{sde:intro}.~Moreover, the sign of the top Lyapunov exponent is related to several fundamental problems in dynamical %\todo{reformulate this sentence since it's not clear for the reader\textcolor{red}{done}} 
systems.~For example, having an ergodic invariant measure for~\eqref{sde:intro}, one can disintegrate it into probability measures to obtain a family of random measures on the state space as constructed in Theorem~\ref{disintegration_2}. These random measures turn out to be closely connected to the sign of the first Lyapunov exponent.~When the top Lyapunov exponent is negative, they are typically expected to become atomic as showed in~\cite{BG26}.~In contrast, when the first Lyapunov exponent is positive, the measures are expected to admit absolutely continuous conditional distributions on unstable manifolds, leading to the existence of Sinai-Ruelle-Bowen (SRB) measures.~Naturally, the non-Markovian nature of the system leads to additional challenges.~Nevertheless, we believe that the tools established in this manuscript can serve as a basis for further questions, such as obtaining a Furstenberg-type representation theorem for the top Lyapunov exponent. Moreover, we plan to extend our results to gradient type SDEs, i.e. $F=-\nabla V $, for a potential $V$ and multiplicative noise, where the setting of~\cite{Nualart} could be helpful.
\paragraph*{Literature}  
There has been a growing interest in the dynamics of stochastic systems with non-Markovian noise.~Several results as averaging~\cite{averaging,LiS}, multi-scale dynamics~\cite{multiscale1,Katharina,Nils}, finite-time Lyapunov exponents~\cite{DB}, early warning signs for bifurcations~\cite{EWS}, invariant manifolds and stability~\cite{GVR25Z,GVR25,GVR26,KN21,KN23} or traveling waves~\cite{Stannat} have been established.~Remarkably, it was observed in~\cite{Stannat} that large values of the Hurst parameter $H$ imply higher stability of traveling waves.~In~\cite{DB,Nils} and here this is not the case, since we can always obtain stability, i.e.~a negative top Lyapunov exponent by increasing the noise intensity.~However, the proof of this statement is more technical for $H\in(1/2,1)$ due to certain singularities that appear in zero, see Proposition~\ref{YHHNA7as}.~Moreover, the results in~\cite{Stannat,DB,Nils} are valid on a finite-time horizon, whereas here we are interested in asymptotic statements and therefore need ergodicity which is a challenging task in a non-Markovian setting.~On the other hand, the negativity of the top Lyapunov exponent and synchronization by noise for SDEs with additive noise were investigated in \cite{FGS16a,FGS16b}.~For one-dimensional SDEs driven by fbm, weak synchronization was established in~\cite[Section 4.1]{FGS16b} using that the 
corresponding random dynamical system is order-preserving and strongly mixing whereas for higher-dimensional SDEs we refer to~\cite{BG26}.  
\paragraph*{Plan of the paper} This work is structured as follows. In Section~\ref{sec:p} we collect important concepts from the theory of stochastic and random dynamical systems, the Mandelbrot van Ness representation of fractional Brownian motion together with a stationary noise process associated to it.~Since this encodes the past of the noise we construct in Subsection~\ref{sec:fbm} a two-sided fbm as required for the RDS approach.~Based on these foundations we construct in Section~\ref{sec:cons:sds} a RDS given a SDS.~Since the SDS incorporates the past of the noise and the RDS its future, for the construction of the RDS given the SDS in Theorem~\ref{RDS_SDS}, we first define a space encoding the randomness which contains two components (the past and future of the noise) and an appropriate shift.~In this setting we can verify the cocycle property.
Moreover, the SDS is defined on the product space of two Polish space $\cX$ and $\cB$, where $\cX$ stands for the phase space and $\cB$ is associated to the noise.~In this setting one can construct a probability measure on $\cX\times \cB$ which is invariant for the SDS.~Given this, we construct an invariant measure on the phase space $\cX$ for the RDS by disintegration together with a perfection-type argument in Theorem~\ref{disintegration_2} to prove its invariance.~This result is of independent interest and essential for the forthcoming work~\cite{BG26} on synchronization by fractional noise for~\eqref{sde:intro}.\\

In Section~\ref{sec:sde} we consider the SDE~\eqref{sde:intro} for which one can obtain a SDS due~\cite{Hai05}.~We apply the results stated in Section~\ref{sec:cons:sds} to generate a RDS and compare this approach to the standard one, where one subtracts the noise from~\eqref{sde:intro} obtaining an ODE with random coefficients~\cite{MS04}.~Here the model of the noise is described by the metric dynamical system $(C_0(\R;\R^d); \cB(C_0(\R;\R^d)); \mathbb{P}; (\theta_t)_{t\in\R})$, where $C_0(\R;\R^d)$ denotes the space of continuous functions which are zero in zero, $\mathbb{P}$ is the two-sided Wiener measure and $\theta$ is the Wiener shift.~In our setting, the space associated to the noise contains two components, both chosen from the space specified in Definition~\ref{BB_HH} whose choice is motivated by the Mandelbrot van Ness representation of fbm.~Before analyzing the negativity of the top Lyapunov exponent, as a first step we investigate in Subsection~\ref{Linearization} the linearized dynamics of~\eqref{sde:intro}.~Provided that the top Lyapunov exponent is negative, we obtain the existence of a local stable manifold.~A similar statement has been derived in the Markovian setting in~\cite[Section 3.1]{FGS16a}.~We employ novel tools based on the results in~\cite{GVR23A} to obtain this kind of statement in the non-Markovian case. \\

In Section~\ref{sec:LP} we first rescale the SDE~\eqref{sde:intro} and derive bounds for the density of the invariant measure of the corresponding rescaled SDE depending on $\sigma$.~To this aim we follow the approach of~\cite{LPS23}, where in our situation we have to keep track of the dependence of $\sigma$ on both drift and diffusion coefficient. By the rescaling argument, the diffusion term lies on the unit sphere and the parametric dependence can be handled.~For the convenience of the reader, we provide these computations in Appendix \ref{AAZAZA}. Furthermore, based on Birkhoff's ergodic theorem we show that increasing $\sigma$ we obtain a negative top Lyapunov exponent for~\eqref{sde:intro}.~We collect some properties of fractional Brownian motion, auxiliary results and tools on fractional calculus required in Section~\ref{sec:LP} in three appendices.
	\paragraph{Acknowledgements}
The authors acknowledge funding by the Deutsche Forschungsgemeinschaft (DFG, German Research Foundation) - CRC/TRR 388 "Rough Analysis, Stochastic Dynamics and Related Fields" - Project ID 516748464.~A Blessing acknowledges support from DFG CRC 1432 " Fluctuations and Nonlinearities in Classical and Quantum Matter beyond Equilibrium" - Project ID 425217212.
\section{Preliminaries}\label{sec:p}
\subsection{Notations} 
In this section, we introduce some notations, definitions and auxiliary results which will frequently by used throughout the manuscript.

\begin{lemma}\label{pushforward}
	Let \((\mathcal{X}, \sigma(\mathcal{X}))\) and \((\mathcal{Y}, \sigma(\mathcal{Y}))\) be two measure spaces, \(T: \mathcal{X} \to \mathcal{Y}\) be a measurable function and let \({\nu}\) be a Borel measure on \(X\). Then the pushforward measure \(T_{\star} {\nu}\) on \(\mathcal{Y}\) is defined by  
	\[
	T_{\star} \nu(E) := \nu(T^{-1}(E)), \quad \forall E \in \sigma(\mathcal{Y}).
	\]
	In addition, for a measurable function \(f: \mathcal{Y} \to \mathbb{R}^+\), we have  
	\[
	\int_{\mathcal{X}} f(T(x)) \, \nu(\mathrm{d}x) = \int_{\mathcal{Y}} f(y) \, T_{\star} \nu(\mathrm{d} y).
	\]
\end{lemma}
The following result is widely used in this manuscript.  
\begin{lemma}{\em (Fernique's theorem)}\label{FerniqueAA}
	Let $(\mathcal{X}, \|\cdot\|)$ be a real separable Banach space and let $\nu$ be a centered Gaussian measure on $\mathcal{X}$. Then there exists a constant $c > 0$ such that
	\[
	\int_{\mathcal{X}} \exp \bigl(c \|x\|^2\bigr)\, \nu(\mathrm{d}x) < \infty.
	\]
\end{lemma}

\proof
	See \cite[Theorem~4.1]{Led96} and \cite[Theorem~15.33]{FV10}.
\qed
\begin{definition}
	Let $(\Omega,\mathcal{F},\mathbb{P})$ be a probability space.~We define 
	\[
	\mathrm{Pr}_{\mathcal{B}}(\mathbb{R}^d) := \left\{ \nu \text{ is a probability measure on } \mathbb{R}^d \times \mathcal{B} \,\middle|\, \left( \Pi_{\mathcal{B}} \right)_{\ast} \nu = \mathbf{P} \right\},
	\]
	where $\Pi_{\mathcal{B}}$ denotes the projection onto $\mathcal{B}$.
\end{definition}
%We also recall the following definition.
\begin{definition}
	For a path \(f : [0, T] \to \mathbb{R}^d\) and a parameter \(\alpha \in (0,1]\), we define the H\"older seminorm
	\[
	\|f\|_{\alpha; [0,T]} := \sup_{\substack{s,t \in [0,T]\\s \neq t}}
	\frac{\lvert f(t) - f(s)\rvert}{\lvert t - s\rvert^{\alpha}},
	\]
	and the supremum norm
	\[
	\|f\|_{\infty; [0,T]} := \sup_{s \in [0,T]} \lvert f(s)\rvert.
	\]
	For \(m \in \mathbb{N}_0\), we say that \(f \in C^{m,\alpha}([0,T],\mathbb{R}^d)\) if
	\begin{itemize}
		\item \(f\) is \(m\)-times continuously differentiable on \([0,T]\), and
		\item its \(m\)-th derivative \(f^{(m)}\) has finite \(\alpha\)-H\"older seminorm.
	\end{itemize}
	We equip \(C^{m,\alpha}([0,T],\mathbb{R}^d)\) with the norm
	\[
	\|f\|_{C^{m,\gamma};[0,T]}
	:= \sum_{k=0}^{m} \|f^{(k)}\|_{\infty;[0,T]}
	\;+\; \|f^{(m)}\|_{\alpha;[0,T]}.
	\]
\end{definition}
\subsection{Stochastic and random dynamical systems}
In this section, we provide a concise overview of the theory of {stochastic dynamical systems} (SDS) introduced in the seminal work~\cite{Hai05}. Furthermore, we outline key definitions and concepts from the theory of {random dynamical systems} (RDS) which goes back to~\cite{Arn98}.~Even if these two theories are not equivalent, we show how to obtain a RDS given a SDS and how to apply well-established results from  RDS theory in order to investigate non-Markovian dynamics.
%Now, let us provide some basic concepts of the theory of stochastic dynamical systems based on \cite{Hai05}.
%\textcolor{red}{In the next section, we apply this machinery to stochastic differential equations (SDEs) with fractional Brownian motion.} \todo{we should cite the definition as~\cite[Definition  ]{Hai05}}
\begin{definition}{(\cite[Definition 2.1]{Hai05})}\label{SDS}
	Let \(\mathcal{B}\) be a Polish space equipped with the Borel \(\sigma\)-algebra \(\sigma(\mathcal{B})\).~Then \(\left(\mathcal{B}, \{P_t\}_{t \geq 0}, \mathbf{P}, \{\vartheta_t\}_{t \geq 0}\right)\) is called a {stationary noise process} if it satisfies the following conditions.
	%\todo{check alignment\textcolor{red}{what you mean?}} 
	\item \textbf{(I)}\label{I} We assume that $\mathbf{P}$ is a probability measure on $(\mathcal{B},\sigma(\mathcal{B}))$.\\
	\textbf{(II)}\label{II} For every $t\geq 0$, we assume that \( P_{t}: \mathcal{B} \times \mathcal{B} \to \mathcal{B} \) and \( \vartheta_t: \mathcal{B} \to \mathcal{B} \) are two measurable functions such that for every \( \omega^- \in \mathcal{B} \), the random variables \( P_{t}(\omega^-, \cdot) \) and \( \vartheta_t \) are independent. This means that for every \( A, B \in \sigma(\mathcal{B}) \), we have %\todo{\textcolor{red}{In most cases two arguments are required. 
			%	It is preferable to distinguish them using ``$-$'' and ``$+$''; 
			%	in the single-argument case, we simply write $\omega$. Are you agree?} yes, but this has to be consistent and before it was not, i am making it now, if the first component is $\omega^-$ it stays like that everywhere and it is not switched to $\omega$ and the other way around}
	\begin{align}\label{indepenet}
		\begin{split}
			&\mathbf{P}\left( \omega^+ \in \mathcal{B} : P_{t}(\omega^-, \omega^+) \in A, \vartheta_{t}(\omega^+) \in B \right)\\& = \mathbf{P}\left( \omega^+ \in \mathcal{B} : P_{t}(\omega^-, \omega^+) \in A \right) \mathbf{P}\left( \omega^+ \in \mathcal{B} : \vartheta_{t}(\omega^+) \in B \right).
		\end{split}
	\end{align}
	\item \textbf{(III)}\label{III} We assume that \(\{\vartheta_t\}_{t \geq 0}\) is a semiflow on \(\mathcal{B}\), i.e. \(\vartheta_{t+s} = \vartheta_t \circ \vartheta_s\) for every \(t, s \geq 0\). In addition, this flow is invariant with respect to $\mathbf{P}$ meaning that \((\vartheta_t)_{\star}(\mathbf{P}) = \mathbf{P}\) for every \(t \geq 0\).
	\item \textbf{(IV)}\label{IV} For every \( \omega^- \in \mathcal{B} \) and $t\geq 0$, we define
	\[
	\mathcal{P}_t(\omega^-; \cdot) := (P_t(\omega^-, \cdot))_{\star}(\mathbf{P}),
	\]
	and assume that this family defines a Feller transition semigroup on \( \mathcal{B} \). This means that for the family of probability measures
	\[
	\mathcal{P}_t : \mathcal{B} \times \sigma(\mathcal{B}) \to [0, 1],
	\]
	and for all \( A \in \sigma(\mathcal{B}) \) and \( t, s > 0 \), the following properties hold:
	\begin{itemize}
		\item the map \(\omega^-\mapsto \mathcal{P}_t(\omega^-; A) \) is measurable;
		\item one has the semigroup property
		\begin{align}\label{SEMI_1}
			\mathcal{P}_{t+s}(\omega^-; A) = \int_{\mathcal{B}} \mathcal{P}_t(\omega^+; A) \mathcal{P}_s(\omega^-; \mathrm{d}\omega^+);
		\end{align}
		\item for every $A\in\sigma(\cB)$ and $\omega\in\mathcal{B}$  it holds that \( \mathcal{P}_0(\omega^-, A) = \mathlarger{\chi}_{A}(\omega^-) \), where $\chi_A$ denotes the indicator function of the set $A$;
		\item for every \( f \in C_b(\mathcal{B}) \), the map
		\begin{align*}
			f \mapsto \mathcal{P}_t f, \quad 
			\mathcal{P}_t f(\omega^-) := \int_{\mathcal{B}} f(P_t(\omega^-, \omega^+)) \, \mathbf{P}(\mathrm{d}\omega^+),
		\end{align*}
		belongs to the space of bounded and continuous functions on $\cB$ denoted by \( C_b(\mathcal{B}) \). %\todo{the notation in the first component is fixed, either $\omega$ or $\omega^-$ but we are not randomly switching them}
	\end{itemize}
	Furthermore, we assume that \( \mathbf{P} \) is the unique invariant measure associated with the transition semigroup \( \{ \mathcal{P}_t \}_{t \geq 0} \).
	%\todo{write the map, i.e. $f\mapsto \cP_t f$,\textcolor{red}{done} } %belongs to \( C_b(\mathcal{B}) \), and \( \mathcal{P}_0(\omega^-, A) = \mathlarger{\chi}_{A}(\omega^-) \). Moreover, \( \mathbf{P} \) is the unique invariant measure for \( \mathcal{P}_t \). %\todo{notation for the indicator function is not introduced.\textcolor{red}{done}}
	%\todo{replace also, by moreover, furthermore etc,\textcolor{red}{done}}
	\item \textbf{(V)}\label{V} For every \( \omega^- ,\omega^+ \in \mathcal{B} \) and \( t,s \geq 0 \), we assume that
	\begin{align*}
		&\vartheta_{t}\circ P_{t}(\omega^-,\omega^+)= P_{t}\left(\vartheta_{t}\omega^-,P_{t}(\omega^+,\omega^-)\right)=\omega^-,\\
		& P_{t+s}(\omega^-,\omega^+)=P_{t}(P_{s}(\omega^-,\omega^+),\vartheta_{s}\omega^+).
	\end{align*} 
\end{definition}
%\todo{\textcolor{red}{In the definition, the symbol $\omega$ denotes a single element of $\mathcal{B}$, whereas the notation $(\omega^-, \omega^+)$ is used when a pair of elements is required.}
	%	}
%\todo{why $P_t(\omega^+,\omega^-)$ in the first statement of V?,\textcolor{red}{Here is the assumption, so we can simply assume that. But I think you are referring to proving it when applied to our model in the FBM?In that case, we prove it there, where these maps are defined.
		%}}

\begin{remark}
	We note that properties \hyperref[I]{\textbf{(I)}}--\hyperref[IV]{\textbf{(IV)}} are stated in~\cite[Definition~2.1]{Hai05}. In addition, we require property \hyperref[V]{\textbf{(V)}}, which can easily be verified for the stationary noise process associated to fbm in Section~\ref{sec:fbm}.
	%Similarly, the property \hyperref[V]{\textbf{(V)}} is not explicitly stated in~\cite{Hai05}, but it can be easily verified in the examples considered in this manuscript.
\end{remark}

Now we define a stochastic dynamical system.
\begin{definition}\label{SDSSS}
	Let \( \left(\mathcal{B}, \{P_t\}_{t \geq 0}, \mathbf{P}, \{\vartheta_t\}_{t \geq 0}\right) \) be a stationary noise process and let \( \mathcal{X} \) be a Polish space. We call 
	\begin{align}
		\begin{split}
			&\varphi: \mathbb{R}^+ \times  \mathcal{X}\times \mathcal{B}  \to \mathcal{X},\\
			&  (t,x,\omega)\rightarrow \varphi^{t}_{\omega}(x)
		\end{split}
	\end{align}
	a continuous stochastic dynamical system (SDS) if the following properties hold.
	\begin{itemize}
		\item  For every \( T > 0 \), \( \omega \in \mathcal{B} \), and \( x \in \mathcal{X} \), for the map \( \tilde{\varphi}_T( x,\omega): [0, T] \to \mathcal{X} \) defined as
		\[
		\tilde{\varphi}_T(x,\omega)(t) := \varphi^{t}_{\vartheta_{T - t} \omega}(x),
		\]
		we have \( \tilde{\varphi}_T(x,\omega) \in C\left([0, T], \mathcal{X}\right) \). Moreover,  the map \( (x,\omega) \mapsto \tilde{\varphi}_T(x,\omega) \) is continuous from \(  \mathcal{X}\times\mathcal{B} \) to \( C\left([0, T], \mathcal{X}\right) \).
		\item For every $t,s\geq 0$, $\omega\in \mathcal{B}$ and $x\in\mathcal{X}$, we have
		\begin{align}\label{cocycle_SDS}
			\begin{split}
				&\varphi^{0}_{\omega}(x)=x,\\
				&\varphi^{t+s}_{\omega}(x)=\varphi^{s}_{\omega}\circ\varphi^{t}_{\vartheta_{s}\omega}(x).
			\end{split}
		\end{align}
		\item If $\cX$ is a separable Banach space, we call the stochastic dynamical system   \( C^k \) if for every \( k \geq 1 \), \( t \geq 0 \) and \( \omega \in \mathcal{B} \), the map
		\[
		x \mapsto \varphi^t_{\omega}(x)
		\]
		is \( C^k \)-Fr\'{e}chet differentiable.
	\end{itemize}
	
\end{definition}
When referring to a continuous SDS, we omit specifying the stationary noise process and the underlying
Polish space \( \mathcal{X} \) whenever these are clear from the context. 
%\begin{remark} 
%We may also assume that \( \mathcal{X} \) is a \( C^k \)-manifold and define relevant concepts, such as the differentiability of \( \varphi \), in a natural way.
%\end{remark}
The following result is taken from \cite[Lemma 2.12]{Hai05}.
\begin{definition}\label{DEAFRS}
	Let \( \varphi \) be a continuous SDS. For \( t \geq 0 \) and \( \Gamma \in \sigma(\mathcal{X}) \otimes \sigma(\mathcal{B}) \), we set
	\begin{align}\label{Chapman}
		\mathcal{Q}_{t}\left((x,\omega^-); \Gamma\right) := \int_{\mathcal{B}} \mathlarger{\chi}_{\Gamma}\left( \varphi^{t}_{P_{t}(\omega^-, \omega^+)}(x), P_{t}(\omega^-, \omega^+) \right) \mathbf{P}(\mathrm{d}\omega^+).
	\end{align}
	Then the family \( ( \mathcal{Q}_t )_{t \geq 0} \) constitutes a Feller transition semigroup on \( \mathcal{X} \times \mathcal{B} \). %In particular the Chapman--Kolmogorov equations hold, meaning that for every \( t, s > 0 \), \( (x_1, \omega^-) \in \mathcal{X} \times \mathcal{B} \) and \( \Gamma \in \sigma(\mathcal{X}) \otimes \sigma(\mathcal{B}) \) we have
	%\begin{align*}%\label{SEMI_2}
	%\mathcal{Q}_{t+s}\left((x_1,\omega^-);\Gamma\right) = \int_{\mathcal{B}} \mathcal{Q}_{t}\left((x_2,\omega^+); \Gamma\right) \mathcal{Q}_s\left((x_1,\omega^-); \mathrm{d}( x_2\times \omega^+)\right).
	%\end{align*}
\end{definition}
This leads to the following definition of an invariant measure for a SDS on the state space \( \mathcal{B} \times \mathcal{X} \).
\begin{definition}\label{measure}
	Let \( \varphi \) be a continuous SDS. A probability measure \( \mu \) on \( \mathcal{B} \times \mathcal{X} \) is called an {invariant measure} for $\varphi$ if it satisfies the following properties:
	\begin{itemize}
		\item we have  
		\begin{align}\label{UJas}
			(\Pi_\mathcal{B})_{\star} \mu = \mathbf{P},
		\end{align}
		%	where \( \Pi_{\mathcal{B}}: \mathcal{X} \times \mathcal{B} \to \mathcal{B} \) denotes the canonical projection onto \( \mathcal{B} \).
		%	Similarly, we  use $(\Pi_{\mathcal{X}})_{\star} \mu$ to denote the usual pushforward measure induced on $\mathcal{X}$.
		\item the measure \( \mu \) is invariant with respect to the transition semigroup \( (\mathcal{Q}_t)_{t\geq 0} \), i.e. for every \( t \geq 0 \) and $\Gamma\in\sigma(\mathcal{X})\otimes\sigma(\mathcal{B})$ it holds that
		\[
		(\mathcal{Q}_t)_{\star} \mu(\Gamma) :=\int_{\mathcal{X}\times\mathcal{B}}\mathcal{Q}_{t}\left((x,\omega);\Gamma\right)\mu\left(\mathrm{d}(x,\omega)\right)= \mu(\Gamma).
		\]
	\end{itemize}
	%This measure is called {ergodic} if it is ergodic in the sense of \cite[Section 2.2]{DPZ96}. \todo{state this ergodicity for completeness, \textcolor{red}{This is a canonical construction, and if you agree, simply refer to the address, as this is quite classical.} i know it's classical but i am glad that you added section 2.2 in~\cite{DPZ96}, we should be precise with the references}
\end{definition}
%\begin{remark}
%	Note that from Definition \ref{SDS}, we can conclude that the transition semigroup $\lbrace\mathcal{Q}_{t}\rbrace_{t\geq 0}$ is stochastically continuous. This means that for every $(x,\omega) \in \mathcal{X} \times \mathcal{B}$ and for every open ball of radius $\delta>0$ around $(x,\omega)$ denoted by $B((x,\omega),\delta)$, we have
%	\begin{align*}
	%		\lim_{t\rightarrow 0}\mathcal{Q}_{t}\left((x,\omega);B((x,\omega),\delta)\right)=0.
	%	\end{align*}
%	The same claim holds for $(\mathcal{P}_t)_{t\geq 0}$.
%\end{remark}
\begin{remark}\label{ASAsw65}
	From \eqref{Chapman} we know that for $A\times B\in\sigma(\mathcal{X})\otimes\sigma(\mathcal{B})$
	\begin{align}\label{PROJJE}
		(\Pi_{\mathcal{B}})_{\star}\mathcal{Q}_{t}\left((x,\omega);A\times B)\right)=\mathcal{P}_{t}(\omega;B). 
	\end{align}
	In particular, this implies that the projected measure above evolves independently of the state space \( \mathcal{X} \).
	%	Moreover from \hyperref[V]{\textbf{(V)}}, the transition semigroup \( \lbrace\mathcal{P}_{t}\rbrace_{t\geq 0} \) has the unique invariant measure $\mathbf{P}$. Consequently, if \( \mu \) is an invariant measure of \( \lbrace\mathcal{Q}_{t}\rbrace_{t\geq 0} \), then condition \eqref{UJas} holds.
	In order to prove the uniqueness of the measure \( \mu \), it suffices to consider the class of invariant measures of \( \lbrace\mathcal{Q}_{t}\rbrace_{t\geq 0}\) satisfying condition~\eqref{UJas}.
	We also recall that by the ergodicity of \( \mu \), we mean ergodicity with respect to the shift map on \( C(\mathbb{R}^+, \mathcal{X} \times \mathcal{B}) \), see \cite[Chapter 2]{DPZ96} for further details. %\todo{later, emphasize there is also ergodicity in rds sense, defined below}
\end{remark}
We now introduce some concepts from the theory of random dynamical systems.
\begin{definition}\label{SAGBAsd}
	Let \((\Omega, \mathcal{F})\) be measurable spaces and let \(\T\) %\todo{$\T$ everywhere in the def since $T$ is a time horizon, usually the MDS is defined on $\R$,\textcolor{red}{True. Let us also consider $\R^+$, since we used that}}
	%\todo{yes, but we have to write $\T$ everywhere not only here but also below. if a notation has to be changed in the todo we do it everywhere not only in one place, i changed it now}
	be either \(\mathbb{R}\) or \(\mathbb{R}^+\), each equipped with its corresponding Borel \(\sigma\)-algebra.
	A family $\theta = \lbrace\theta_t\rbrace_{t \in \T}$ of maps from $\Omega$ to itself is called a {measurable dynamical system} if the following conditions are satisfied.
	\begin{itemize}
		\item[(i)] The map $(\omega, t) \mapsto \theta_t(\omega)$ is $\mathcal{F} \otimes \sigma(\T) / \mathcal{F}$-measurable. %\todo{$\cG$ is $\cF$,\textcolor{red}{true}}
		\item[(ii)] It holds $\theta_0 = \mathrm{Id}$.
		\item[(iii)] The semigroup property holds, i.e. $\theta_{s+t} = \theta_s \circ \theta_t$, for all $s,t \in \T$.
	\end{itemize}
	If $\P$ is a probability measure on $(\Omega, \mathcal{F})$ which is invariant with respect to $(\theta_{t})$, i.e.
	$
	(\theta_{t})_\star\P = \P, \text{for every } t \in \T,
	$
	then the quadruple $(\Omega, \mathcal{F}, \P, \theta)$ is called a {measurable metric dynamical system}. We call \((\Omega, \mathcal{F}, \mathbb{P}, \theta)\) {ergodic} if the map $\theta_t : \Omega \to \Omega$
	is ergodic for every \( t \in \T \setminus \{0\} \). %This means that every $\theta_t$ if a set \( A \in \mathcal{F} \) satisfies $\mathbb{P}\big( A \triangle \theta_t^{-1}(A) \big) = 0,$
	%then
	%\[
	%	\mathbb{P}(A) = 0 \quad \text{or} \quad \mathbb{P}(A) = 1.
	%\]
\end{definition}
Having a measurable metric dynamical system, one can define a random dynamical system.
\begin{definition}\label{Olas5d}
	Let \(\mathcal{X}\) be  a Polish space equipped with its Borel \(\sigma\)-algebra \(\sigma(\mathcal{X})\). Then an {(ergodic) measurable random dynamical system} on \((\mathcal{X}, \sigma(\mathcal{X}))\) consists of an (ergodic) measurable metric dynamical system \((\Omega, \mathcal{F}, \mathbb{P}, \theta)\) together with a measurable map
	\[
	\Phi \colon \mathbb{R}^+ \times \mathcal{X}\times\Omega  \to \mathcal{X},
	\]
	which satisfies the {cocycle property}.~Setting \(\Phi^t_\omega(x) := \Phi(t,x,\omega)\) we require
	\[
	\Phi^0_\omega = \operatorname{Id}_{\mathcal{X}} \quad \text{for all } \omega \in \Omega,
	\]
	and
	\[
	\Phi^{t+s}_\omega(x) = \Phi^s_{\theta_t \omega} \big( \Phi^t_\omega(x) \big),
	\]
	for all \(s, t \in \mathbb{R}^+\), \(x \in \mathcal{X}\), and \(\omega \in \Omega\). The map \(\Phi\) is called a {cocycle}.
	If additionally the map \(\Phi(\cdot,\cdot,\omega) \colon \mathbb{R}^+ \times \mathcal{X} \to \mathcal{X}\) is continuous for every \(\omega \in \Omega\), then \(\Phi\) is called a {continuous random dynamical system}.
	%	More generally, we say that {\(\Phi\) has property \(P\)} if and only if the map \(\Phi(t, \omega, \cdot) \colon \mathcal{X} \to \mathcal{X}\) has property \(P\) for every \(t \in \mathbb{R}^+\) and \(\omega \in \Omega\).% whenever the latter statement makes sense. 
\end{definition}
%When referring to the RDS \(\Phi\) we omit mentioning the underlying quadruple \((\Omega, \mathcal{F}, \mathbb{P}, \theta)\) and the Polish space \(\mathcal{X}\), whenever these are clear from context.

%\begin{remark}
%In Definition~\ref{SAGBAsd}, \(\T\) can also be either \(\mathbb{Z}\) or \(\mathbb{Z}^+\). %\todo{this is a repetition, \textcolor{red}{you mean the remark is not necessary?}} 
%In this case, we equip \(\T\) with its power set \(P(\T)\) to form a measurable space. %\todo{clear and $\cP$ is used for something else,\textcolor{red}{true}} 
%	Similarly, in Definition~\ref{Olas5d}, one can replace \(\mathbb{R}^+\) with \(\mathbb{Z}^+\).
%		\todo{{red}{I incorporated your comment. We could skip this remark, but I think it's better to keep it included.}}
%\end{remark}

Let $\Phi$ be a measurable RDS. Then we can define the associated skew-product flow for $t\geq 0$ by
\begin{align}\label{OL963ass}
	\Theta_{t} : \mathcal{X} \times \Omega \to \mathcal{X} \times \Omega, \quad
	(x, \omega) \mapsto \big( \Phi^{t}_{\omega}(x), \theta_t \omega \big).
\end{align}
It is straightforward to verify that the semigroup property holds, meaning that for all \(s, t \geq 0\), we have
\[
\Theta_{t+s} = \Theta_s \circ \Theta_t.
\]
This leads to the following natural concept of an invariant measure for a random dynamical system.
\begin{definition}\label{ATTAAARFASA}
	Let $\Phi$ be a measurable RDS over the metric dynamical system $(\Omega,\cF,\P,\theta)$ and let $\Theta = \{\Theta_t\}_{t \geq 0}$ be the associated skew-product~\eqref{OL963ass}.~A probability measure $\nu$ on the product space $\left( \mathcal{X} \times \Omega, \sigma(\mathcal{X}) \otimes \mathcal{F} \right)$
	is called an {invariant measure} for the RDS $\Phi$ if the following conditions are satisfied:
	\begin{itemize}
		\item we have
		\[
		(\Pi_\Omega)_\star \nu = \mathbb{P};
		\]
		%	where \( \Pi_{\Omega}: \mathcal{X} \times \Omega \to \Omega \) is the canonical % projection onto \(\Omega\). 
		\item the measure \( \nu \) is invariant under the skew-product map $\Theta$, i.e., for every \( t \geq 0 \),
		\[
		(\Theta_t)_\star \nu = \nu.
		\]
	\end{itemize}
	Similar to Definition~\ref{SAGBAsd}, the measure \( \nu \) is called {ergodic} if it is ergodic with respect to the transformation \( \Theta_t \) for every \( t > 0 \).
\end{definition}
\subsection{Stationary noise process}\label{sec:fbm}

In order to show that fractional Brownian motion generates a stationary noise process in the sense of Definition~\ref{SDS}, we rely on the Mandelbrot van Ness representation.

\begin{definition}
	Let \( (W(t))_{t\in\R} \) be a two-sided Wiener process with values in \( \mathbb{R}^d \) and $H\in (0,1)$.  
	Then we have the following Mandelbrot van Ness representation of a two-sided fractional Brownian motion $(B^H_t)_{t\in \R}$ given by
	\begin{align}\label{ASAdd}
		B^{H}_t  = \frac{1}{\alpha_{H}} \int_{-\infty}^{0}  (-r)^{H-\frac{1}{2}} \left( \mathrm{d}W(t+r) - \mathrm{d}W(r) \right),
	\end{align}
	where the stochastic integral is understood in the It\^o sense and
	\begin{align*}
		\alpha_{H} := \left( \frac{1}{2H} + \int_{0}^{\infty} \left( (1 + s)^{H - \frac{1}{2}} - s^{H - \frac{1}{2}} \right)^2 \, \mathrm{d}s \right)^{\frac{1}{2}}.
	\end{align*}
\end{definition}
%We mention that there also exists a Volterra-type representation for fractional Brownian motion.~Since this does not lead to a stationary noise process, we rely on~\eqref{ASAdd}. 
Given this representation, $W$ is referred to as the Wiener process associated to $B^H$.%~\cite{Hai05}.

\begin{definition}\label{BB_HH}
	For \( 0 < H < 1 \), we define \( C_0^\infty(\mathbb{R}^-, \mathbb{R}^d) \) as the space of smooth, compactly supported functions \( \omega: \mathbb{R}^- \to \mathbb{R}^d \) such that \( \omega(0) = 0 \).  For $\omega\in C_0^\infty(\mathbb{R}^-, \mathbb{R}^d)$ we set
	\begin{align}\label{norm:BH}
		\Vert \omega\Vert_{\mathcal{B}_H}:=\sup_{s,t\in \R^- }\frac{\vert \omega(t)-\omega(s)\vert}{\sqrt{1+\vert t\vert+\vert s\vert }\vert t-s\vert^{\frac{1-H}{2}}}.
	\end{align}
	The space \( \mathcal{B}_H \) is defined as the closure of \( C_0^\infty(\mathbb{R}^-, \mathbb{R}^d) \) with respect to the norm given above and is a separable Banach space.% under the condition \( \omega(0) = 0 \). \todo{repetition, we have this already}
	%This is a separable Banach space.
\end{definition}

Motivated by the representation~\eqref{ASAdd} we state the following result. 
\begin{lemma}{\em (\cite[Lemma 3.6]{Hai05})}\label{AJSJAsd}
	For \( \omega \in C_0^\infty(\mathbb{R}^-, \mathbb{R}^d) \) the operator
	\begin{align}\label{AAS11}
		\mathcal{D}_H \omega(t) := \frac{1}{\alpha_H} \int_{-\infty}^{0} (-r)^{H - \frac{1}{2}} \left( \dot{\omega}(t + r) - \dot{\omega}(r) \right) \, \mathrm{d}r, \quad t \leq 0,
	\end{align}
	can be canonically extended to \( \mathcal{B}_H \). This means that for any sequence \( \{ \omega_n \}_{n \geq 1} \subset C_0^\infty(\mathbb{R}^-, \mathbb{R}^d) \) converging to \( \omega \in \mathcal{B}_H \), the  limit 
	$
	\mathcal{D}_H \omega := \lim_{n \to \infty} \mathcal{D}_H \omega_n
	$
	exists.~Moreover, \( \mathcal{D}_H \omega \) takes values in \( \mathcal{B}_{1-H} \) and the mapping
	%\begin{align}\label{D_HHF}
	$\mathcal{D}_H : \mathcal{B}_H \longrightarrow \mathcal{B}_{1-H}$
	is continuous and admits a bounded inverse. Furthermore, there exists a constant \( \gamma_H > 0 \), with \( \gamma_H = \gamma_{1-H} \), such that the inverse of \( \mathcal{D}_H \) is given by
	\[
	\mathcal{D}_H^{-1} = \gamma_H \mathcal{D}_{1-H} : \mathcal{B}_{1-H} \longrightarrow \mathcal{B}_H.
	\]
\end{lemma}

We point out that there exists a Gaussian measure \(\mathbf{P}\) on $\cB_H$ such that the
canonical process associated to it is a time-reversed Brownian motion.~Then, thanks to the properties of the operator \(\mathcal{D}_H\), the pushforward measure \((\mathcal{D}_H)_\star \mathbf{P}\) defines a time-reversed fractional Brownian motion on \(\mathcal{B}_{1-H}\) with Hurst parameter \(H\). Furthermore, using the operators \(\mathcal{D}_H\) and \(\mathcal{D}_{1-H}\), one can switch between Wiener processes and fractional Brownian motions.%~We summarize these facts in the following lemma.

%\begin{lemma}{\em(\cite[Lemma 3.8]{Hai05}}\label{OLalsf}
%	For  \( H \in (0, 1) \), there exists a unique Gaussian measure \( \mathbf{P} \) in \( \mathcal{B}_H \), for which the associated canonical process is a time-reversed Brownian motion. In addition, the  canonical process associated to the measure \((\mathcal{D}_H)_{\star} \mathbf{P}\) on \(\mathcal{B}_{1-H}\) is a time-reversed fractional Brownian motion with Hurst parameter \(H\).
%\end{lemma}

\begin{remark}\label{RGASsdd}
	As we already mentioned, for the {canonical Wiener process} \(\omega \in \mathcal{B}_{H}\), the process \(\{\mathcal{D}_H \omega(t)\}_{t \leq 0}\) defines a fractional Brownian motion for negative times. %\todo{whenever we talk about the adaptedness we must specify the filtration,\textcolor{red}{Thank you for the clarification. This remark is intended to provide additional (albeit informal) explanation. To address your point, I have added a few sentences.
			%}}
	%This mapping reflects the adaptedness structure of fractional Brownian motion with respect to the natural filtration of the Wiener process $\omega$. 
	However, in the context of random dynamical systems we need a two-sided fractional Brownian motion, i.e.~we have to extend the previous construction to positive times. This is in contrast to the theory of stochastic dynamical systems, where it suffices to characterize the driving noise by its law. 
	Therefore, we will extend in Lemma~\ref{FARASSa} the construction of the fractional Brownian motion to positive times.% preserving its definition for negative times according to Lemma~\ref{OLalsf}.% and ensuring that the resulting process is adapted.% Moreover, this construction will clearly exhibit the memory effects characteristic of fractional Brownian motion, thereby reflecting its non-Markovian nature.
\end{remark}

%\subsection{Stationary noise process associated to fractional Brownian motion}\label{sec:fbm}
We now construct a stationary noise process for a fractional Brownian motion.
\begin{lemma}\label{Olafads} The quadrupel  \(\left(\mathcal{B}_H, \{P_t\}_{t \geq 0}, \mathbf{P}, \{\vartheta_t\}_{t \geq 0}\right)\) is a stationary noise process in the sense of Definition \ref{SDS}, where:
	\begin{enumerate}
		\item the family of shift maps $\lbrace\vartheta_t\rbrace_{t\geq 0}$ on $\mathcal{B}_{H}$ is defined as 
		\begin{align*}
			\vartheta_{t}\omega(s):=\omega(s-t)-\omega(-t),  \ \ \  \forall s\leq  0,
		\end{align*}
		\item  for every $t\geq 0$,  \( P_{t}: \mathcal{B}_{H} \times \mathcal{B}_{H} \to \mathcal{B}_{H} \) is defined as %\todo{state why we don't take in the first def why not $\omega^+(t+s)-\omega^+(t)$ for $-t\leq s$}%\todo{$\tilde{w}=\tilde{\omega}$, check~\eqref{V} for this, \textcolor{red}{This is because of the Brownian motion and the independence of the filtrations. Do you mean Item \textbf{V}, which you have emphasized several times?  
				%		} yes, item V, both identities. last term in (3.4) $\omega$ instead of $\tilde{\omega}$, \textcolor{red}{I think what is written  in (3.4) is correct, regrading the ITEM V, lets discuss in person} it's ok, sorry i made a mistake when computing $P_t(\tilde{\omega},\omega)(-t)$}
		\begin{align}\label{PPOOAS}
			P_{t}(\omega^-,\omega^+)(s):=\begin{cases}
				\omega^+(-s-t)-\omega^+(-t), & \text{for } -t\leq s\leq 0,\\
				\omega^-(s+t)-\omega^+(-t), & \text{for }  s\leq -t . 
			\end{cases}
		\end{align}
	\end{enumerate}
	%\begin{comment}
	\begin{figure}[h]
		\centering
		\includegraphics[width=7.5cm]{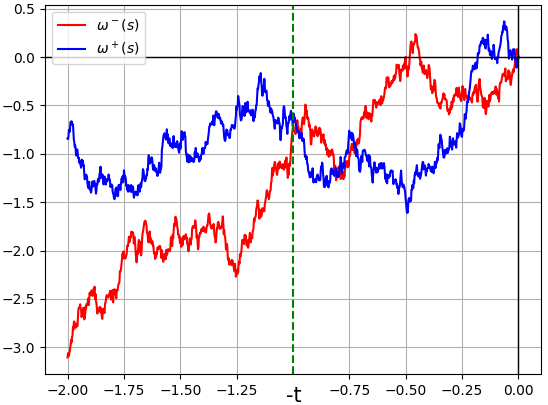}
		\includegraphics[width=7.5cm]{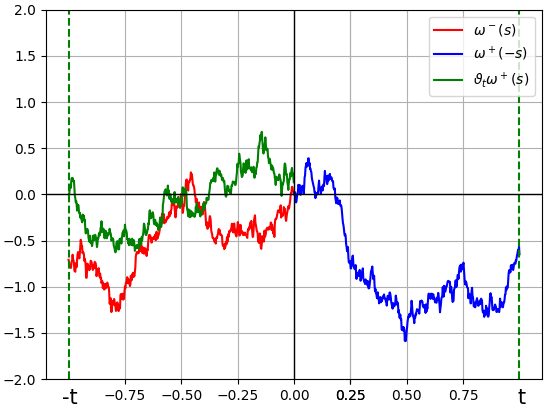}
		\includegraphics[width=7.5cm]{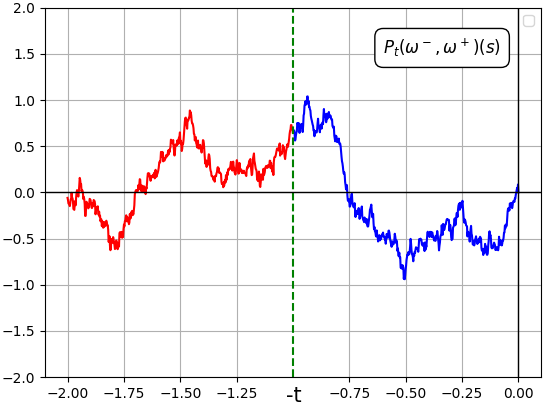}
		\caption{Visualization of \( \vartheta_{t}\omega^+\) and \( P_t(\omega^-, \omega^+) \) for two paths \( \omega^- \) and \( \omega^+ \) in \( \mathcal{B}_{H} \). }
	\end{figure}
	%\end{comment} 
\end{lemma}
\proof
	Since \( \mathcal{B}_H \) can be canonically equipped with the Wiener measure \( \mathbf{P} \), we can easily verify that \( \left( \mathcal{B}_H, \{P_t\}_{t \geq 0}, \mathbf{P}, \{\vartheta_t\}_{t \geq 0} \right) \) is a stationary noise process in the sense of Definition \ref{SDS}, see \cite[Lemma 3.10]{Hai05} for more details. We only prove here the identities~\eqref{V} which are straightforward. For example, using the definition of $\{\vartheta_t\}_{t\geq 0}$ and~\eqref{PPOOAS} we have 
	\begin{align*}
		P_{t}\left(\vartheta_{t}\omega^-, P_{t}(\omega^+, \omega^-)\right)(s)=\begin{cases}
			P_{t}(\omega^+, \omega^-)(-s-t)- P_{t}(\omega^+, \omega^-)(-t)=\omega^-(s), & \text{for } -t\leq s\leq 0,\\
			\vartheta_{t}\omega^-(s+t)-P_{t}(\omega^+,\omega^-)(-t)=\omega^-(s), & \text{for }  s\leq -t. 
		\end{cases}
	\end{align*}
	For $s\leq - t$ we have since $\omega^+(0)=0$ that 
	\[ \vartheta_{t}\omega^-(s+t)-P_{t}(\omega^+,\omega^-)(-t)=\omega^-(s) = \omega^-(s+t-t) -\omega^-(-t) - \omega^+(-t+t) +\omega^-(-t)=\omega^-(s) \]
	whereas for $-t\leq s \leq 0$
	\[ P_{t}(\omega^+, \omega^-)(-s-t)- P_{t}(\omega^+, \omega^-)(-t) = \omega^-(s+t-t) -\omega^-(-t) - \omega^-(t-t) +\omega^-(-t) =\omega^-(s). \]
	This verifies that $\mathcal{\vartheta}_t \circ P_t (\omega^-,\omega^+) = P_t(\mathcal{\vartheta}_t \omega^-, P_t(\omega^+,\omega^-))=\omega^-$.
\qed	
For our aims, we need the following result.	\begin{lemma}\label{RR_TT}
	For  every \( \omega^-, \omega^+ \in \mathcal{B}_H \) and every \( 0 \leq s \leq t \) and \( 0 \leq \tau \leq s \), it holds that  
	\begin{align}\label{RR_TT1}
		\mathcal{D}_{H} P_{t}(\omega^-, \omega^+)(\tau - t) - \mathcal{D}_{H} P_{t}(\omega^-, \omega^+)(-t)  
		= \mathcal{D}_{H} P_{s}(\omega^-, \omega^+)(\tau - s) - \mathcal{D}_{H} P_{s}(\omega^-, \omega^+)(-s).
	\end{align}
	Moreover, for all $r \leq 0$ and $\omega^-\in\mathcal{B}_{H}$ we have
	\begin{align}\label{RR_TT2}
		(\mathcal{D}_{H}\vartheta_{t}\omega^-)(r)=\mathcal{D}_{H}\omega^-(-t+r)-\mathcal{D}_{H}\omega^-(-t).
	\end{align}
\end{lemma}
\proof
	We can assume that $\omega, \tilde{\omega} \in C_{0}^\infty(\mathbb{R}^-,\mathbb{R}^d)$. Otherwise, we can take two sequences $\lbrace\omega_n^{-}\rbrace_{n\geq 1}$ and $\lbrace\omega_n^{+}\rbrace_{n\geq 1}$ in $C_{0}^\infty(\mathbb{R}^-,\mathbb{R}^d)$ that converge to $\omega^-$ and $\omega^+$ in $\mathcal{B}_H$ and conclude the claim using the continuity of $P_t$, $P_s$ and $\mathcal{D}_H$. From \eqref{AAS11} it is sufficient to prove that for almost all $r \leq 0$ we have
	\begin{align*}
		&\frac{\mathrm{d}}{\mathrm{d}r}\left[\left(P_{t}(\omega^-,\omega^+)(\tau-t+r)-P_{t}(\omega^-,\omega^+)(-t+r)\right)\right]=\\&\frac{\mathrm{d}}{\mathrm{d}r}\left[\left(P_{s}(\omega^-,\omega^+)(\tau-s+r)-P_{s}(\omega^-,\omega^+)(-s+r)\right)\right],
	\end{align*}	which easily follows from \eqref{PPOOAS}. A similar argument can be used to prove the second claim.
\qed
Now we extend the fractional Brownian motion to positive times. 
\begin{lemma}\label{FARASSa}
	For \( H \in (0,1) \), let \( \Omega_H := \mathcal{B}_H \times \mathcal{B}_H \) and define the probability space \( (\Omega_H, \mathcal{F}, \mathbb{P}) \) by setting
	\[
	\mathcal{F} := \sigma(\mathcal{B}_H) \otimes \sigma(\mathcal{B}_H), \quad \mathbb{P} := \mathbf{P} \times \mathbf{P}.
	\]
	%\todo{natural filtration of the Brownian motion,\textcolor{red}{If you think it is necessary to mention it, then I would suggest citing a reference rather than explaining it in detail.} we just write that it's the natural filtration of a Brownian motion, we must specify natural filtration of that process since the natural filtration can be defined for every process,\textcolor{red}{done}} 
	Furthermore, we define 
	\begin{align}\label{FBANHYD}
		B_t^H := 
		\begin{cases}
			B_t^H\!\left(P_t(\omega^-, \omega^+)\right) 
			= - \bigl( \mathcal{D}_H P_t(\omega^-, \omega^+) \bigr)(-t), & t \geq 0, \\[6pt]
			\ \mathcal{D}_H \omega^-(t), & t \leq 0.
		\end{cases}
	\end{align}
	Then \( (B_t^H)_{t \in \mathbb{R}} \) is a two-sided fractional Brownian motion with Hurst parameter \( H \) adapted to the canonical filtration \( (\mathcal{F}_{-\infty}^t)_{t \in \mathbb{R}} \)  of the Brownian motion, where
	\[
	\mathcal{F}_{-\infty}^t 
	:= \sigma\!\left( \bigcup_{s \leq t} \mathcal{F}_s^t \right) 
	= \bigcup_{s \leq t} \mathcal{F}_s^t.
	\]
\end{lemma}
\proof
	%    \todo{a bit more details, \textcolor{red}{is now ok?}}
	From Definition~\ref{SDS}, for every $t \geq 0$ we have $(P_{t})_{\star}(\mathbb{P}) = \mathbf{P}$.  
	Moreover, from \eqref{RR_TT1} it follows for $s,t \geq 0$ 
	\[
	B_t^H - B_s^H = -\mathcal{D}_{H} P_{t}(\omega^-, \omega^+)(s - t).
	\]
	Therefore we conclude that the process $\{ B^{H}_t \}_{t \in \mathbb{R}}$ is a fractional Brownian motion with Hurst parameter $H$ adapted to the filtration $\{ \mathcal{F}^t_{-\infty} \}_{t \geq 0}$.   %In conclusion, this construction canonically yields a two-sided fractional Brownian motion %\todo{in the statement recall the construction of fbm $D_H\omega(t)$ for $t<0$,\textcolor{red}{Following this comment, numerous modifications have been implemented.}} 
	% with respect to the natural two-parameter filtration \( \{ \mathcal{F}_s^t \}_{s \leq t} \) of the Brownian motion. %\todo{changed: a process is not equipped with a filtration,\textcolor{red}{thanks}}
\qed
%We further state the following result, which will be required in Section~\ref{sec:LP}.
We finally state the following result regarding the pathwise decomposition of fractional Brownian motion into a history and innovation process.
\begin{lemma}\label{SDASASASAa}
	Let \( \omega^-, \omega^+ \in \mathcal{B}_H \). For \( t \geq 0 \), recall that
	\[
	B^{H}_{t}\!\left(P_{t}(\omega^-, \omega^+)\right) = -\bigl(\mathcal{D}_H P_t(\omega^-, \omega^+)\bigr)(-t).
	\]
	Then
	\begin{equation}\label{SALOSWa}
		B^{H}_{t}\!\left(P_{t}(\omega^-, \omega^+)\right) 
		= \mathcal{P}(\omega^{-})(t) + \tilde{B}^{H}_{t}(\omega^+),
	\end{equation}
	where  
	\begin{align}\label{FF14}
		\begin{split}
\mathcal{P}(\omega)(t) 
&:= \frac{1}{\alpha_H} \int_{-\infty}^{0} 
\left[ (t - r)^{H - \frac{1}{2}} - (-r)^{H - \frac{1}{2}} \right] 
\dot{\omega}(r)\, \mathrm{d}r, \\
\tilde{B}^{H}(\omega)(t) 
&:= \tilde{B}^{H}_t(\omega) 
:= -\frac{1}{\alpha_{H}} \int_{-t}^{0} 
(r + t)^{H - \frac{1}{2}} \dot{\omega}(r)\, \mathrm{d}r
		\end{split}
	\end{align}
	extend continuously from \( C_0^\infty(\mathbb{R}^-, \mathbb{R}^d) \) to \( \mathcal{B}_H \).
\end{lemma}
\proof
	We first assume that
	\( \omega^-, \omega^+ \in C_0^\infty(\mathbb{R}^-, \mathbb{R}^d) \). 
	Then the statement follows from \eqref{AAS11} and \eqref{FF14}. 	For \( \omega^-, \omega^+ \in \mathcal{B}_H \), we choose sequences 
	\( (\omega_n^-)_{n \geq 1} \) and \( (\omega_n^+)_{n \geq 1} \) in 
	\( C_0^\infty(\mathbb{R}^-, \mathbb{R}^d) \) and conclude by a density argument. For more details about the properties of $\cP$ and $\tilde{B}^H$ we refer to Lemma \ref{Liouville} and Lemma \ref{Liouville1}.
	%The result then follows by a density argument.
	% \todo{follows from what and what are the basic simplifications,\textcolor{red}{Now must be clear}}
	%	\todo{i would simply state it follows from Lemma 2.24 and 2.27,\textcolor{red}{your comment is addressed}}
\qed
%\begin{remark}
%The previous lemma provides a decomposition of the fractional Brownian motion into two components: the history represented by $\mathcal{P}(\omega^{-})(t)$ and the innovation process given by $\tilde{B}^{H}_{t}(\omega^{+})$.
%\end{remark}
%\textcolor{red}{This definition is supposed to be changed to lemma}:
\begin{comment}  
	\begin{definition}\label{SDASASASAa}
		Recall that
		
		Now, from Lemma~\ref{Liouville}, we have the decomposition
		\[
		B^{H}_{t}\left(P_{t}(\omega^-, \omega^+)\right) = \mathcal{P}(\omega^{-})(t) + \tilde{B}_{t}(\omega^+),
		\]
		where the random variable \( \tilde{B}_{t}(\omega^+) \) is defined by
		\[
		\tilde{B}_{t}^{H}(\omega^+) := - \int_{-t}^{0} (r + t)^{H - \frac{1}{2}} \, \mathrm{d}\omega^{+}(r),
		\]
		for \( \omega^+ \in \mathcal{B}_H \).
		This integral is defined in the It\^o sense%, i.e., almost surely for every \( \omega^+ \in \mathcal{B}_H \). 
		and the process \( (\tilde{B}_{t})_{t \geq 0} \) is called {Liouville fractional Brownian motion}.
	\end{definition}
\end{comment}

\section{Construction of a random dynamical system given a stochastic dynamical system}\label{sec:cons:sds}
Now that we can construct a stationary noise process for a fractional Brownian motion, let us return to the connection between stochastic and random dynamical systems. 
\begin{remark} %\todo{remark that $\Omega=\cB\times\cB$ in thm 3.2 always incorporates the past of the noise as opposite to the RDS framework. Compare rds with sds here emphasizing the different noise spaces and drop section 4.2, \textcolor{red}{Thanks for your careful observation. Regarding this point and Section 4.2, let's discuss it in person. From my perspective, the comparison with the traditional construction is important, as you know this differs from comparing the SDS and RDS. I guess the previous title led you to this decision, which should now be reconsidered (I have updated the title, or you can propose another title). Of course, based on your comments, we can also try to reformulate and summarize some sentences. The point is that the section also includes your observation, and based on your subtle comment a few months ago, we included it because we wanted to apply it to a concrete example rather than something abstract.} }
	% \todo{why canonical? as already wrote if we write probability space we give all components $\Omega,\cF,\P$ and here $\P$ is missing again,\textcolor{red}{true}}
	Before generating a RDS from SDS we first point out some fundamental differences between the two theories. First of all, the abstract space $\Omega$ in Definition~\ref{Olas5d} incorporates the future of the noise, in contrast to $\cB$ which incorporates the past of the noise.~Moreover, in the RDS framework, the noise is frozen once an element $\omega\in\Omega$ is fixed which is in contrast to the construction of the SDS. 
\end{remark}

The next results entails a random dynamical system given a stochastic dynamical system. 
\begin{theorem}\label{RDS_SDS}
	Let \( \varphi \) be a continuous SDS. We consider the probability space $(\Omega,\cF,\P)$ where  $\Omega:=\mathcal{B} \times \mathcal{B}$,\ \(\mathcal{F} := \sigma(\mathcal{B}) \otimes \sigma(\mathcal{B})\) and \(\mathbb{P} := \mathbf{P} \times \mathbf{P}\) and define a family of measurable functions 
	%\todo{if we write probability space we need to write the probability measure as well,\textcolor{red}{true}}
	\begin{align}\label{COCO}
		\begin{split}
			&\Phi: \R^+\times  \mathcal{X}\times\Omega\rightarrow \mathcal{X},\\
			&\left(t,x,(\omega^-,\omega^+)\right)\longrightarrow \Phi^{t}_{(\omega^-,\omega^+)}(x):=\varphi^{t}_{P_{t}(\omega^-,\omega^+)}(x).
		\end{split}
	\end{align}
	Furthermore, we define the shift \( \theta_t : \Omega \to \Omega \)  by
	\begin{align}\label{semigr}
		\begin{split}   
			\theta_t(\omega^-,\omega^+):=\begin{cases}
				\left(P_{t}(\omega^-,\omega^+),\vartheta_{t}\omega^+\right), & \text{for } t\geq 0,\\
				\left(\vartheta_{-t}\omega^-,P_{-t}(\omega^+,\omega^-)\right), & \text{for }  t<0 . 
			\end{cases}
		\end{split}
	\end{align}
	For every \( t, s \in \mathbb{R} \) we have \( \theta_{t+s} = \theta_t \circ \theta_s \). Additionally, this family preserves the probability measure, i.e.
	\begin{align}\label{preserve}
		(\theta_t)_{\star} (\mathbf{P}\times\mathbf{P}) =\mathbf{P}\times\mathbf{P}, \quad \text{ for all }t \in \mathbb{R}.
	\end{align}
	In conclusion, \(\Phi\) is a continuous random dynamical system satisfying the cocycle property, i.e. for \( t, s \geq 0 \) and \( x \in \mathcal{X} \) we have
	\begin{align}\label{cocycle}
		\Phi^{t+s}_{(\omega^-,\omega^+)}(x)=\Phi^{s}_{\theta_t(\omega^-,\omega^+)}\circ\Phi^{t}_{(\omega^-,\omega^+)}(x).
	\end{align}
\end{theorem}
%In particular, this yields that $\Phi$ is a continuous RDS.

\proof
	From \hyperref[V]{\textbf{(V)}} in Definition \ref{SDS} and the semiflow property of \(\vartheta_t\) we have for \(t, s \geq 0\)
	\begin{align}\label{GG1}
		\begin{split}
			\theta_{t}\circ\theta_{s}(\omega^{-},\omega^{+})&=\theta_{t}\left(P_{s}(\omega^-,\omega^+),\vartheta_{s}\omega^+\right)=\left(P_{t}\left(P_{s}(\omega^-,\omega^+),\vartheta_{s}\omega^+\right),\vartheta_{t}\circ\vartheta_{s}\omega^+\right)\\
			&=\left(P_{t+s}(\omega^-,\omega^+),\vartheta_{t+s}\omega^+\right)=\theta_{t+s}(\omega^{-},\omega^{+}).
		\end{split}
	\end{align}
	We now compute the inverse of $\theta_t$ for $t\geq 0$. We observe 
	\begin{align}\label{inver}
		\theta_{t}\circ\theta_{-t}(\omega^{-},\omega^{+})=\theta_{t}\left(\vartheta_{t}\omega^-,P_{t}(\omega^+,\omega^-)\right)=\left(P_{t}(\vartheta_{t}\omega^-,P_{t}(\omega^+,\omega^-)),\vartheta_{t}P_{t}(\omega^+,\omega^-)\right)=(\omega^{-},\omega^{+}),
	\end{align}
	where we used the first property in \hyperref[V]{\textbf{(V)}} in the last step.
	Similarly, we can prove that $\theta_{-t} \circ \theta_{-s} = \theta_{-t-s}$. Now, let \(0 \leq s \leq t\). Then, from \eqref{GG1} and \eqref{inver} we have
	\begin{align*}
		\theta_{t}\circ\theta_{-s}(\omega^{-},\omega^{+})=\theta_{t-s}\circ\theta_{s}\circ\theta_{-s}(\omega^{-},\omega^{+})=\theta_{t-s}(\omega^{-},\omega^{+}).
	\end{align*}
	The remaining cases can be addressed in a similar way. To prove \eqref{preserve}, for simplicity, let \( t \geq 0 \) and \( A, C \in \sigma(\mathcal{B}) \). From Definition \ref{SDS} and \eqref{semigr}
	\begin{comment}     
		\begin{align*}
			&\theta_t_{\star} (\mathbf{P}\times\mathbf{P})(A\times C)=\mathbf{P}\times\mathbf{P}(\theta_{-t}(A\times C))=\int_{\mathcal{B}}\int_{\mathcal{B}}\mathlarger{\chi}_{A}(P_{t}(\omega^-,\omega^+))\mathlarger{\chi}_{C}(\vartheta_{t}\omega^+)\mathbf{P}(\mathrm{d}\omega^+)\mathbf{P}(\mathrm{d}\omega^-)=\\&\mathbf{P}(C)\int_{\mathcal{B}}\int_{\mathcal{B}}\mathlarger{\chi}_{A}(P_{t}(\omega^-,\omega^+))\mathbf{P}(\mathrm{d}\omega^+)\mathbf{P}(\mathrm{d}\omega^-)=\mathbf{P}(C)\int_{\mathcal{B}}\mathcal{P}_t(\omega^-; A)\mathbf{P}(\mathrm{d}\omega^-)=\mathbf{P}\times\mathbf{P}(A\times C),
		\end{align*}
	\end{comment}
	
	\begin{align*}
		(\theta_t)_{\star} (\mathbf{P}\times\mathbf{P})(A\times C)&
		{\mathrel{\mathop{=}}} \int_{\mathcal{B}}\int_{\mathcal{B}}\mathlarger{\chi}_{A}(P_{t}(\omega^-,\omega^+))\mathlarger{\chi}_{C}(\vartheta_{t}\omega^+)\mathbf{P}(\mathrm{d}\omega^+)\mathbf{P}(\mathrm{d}\omega^-)\\
		&\stackrel{ \hyperref[II]{\textbf{(II)}}}{\mathrel{\mathop{=}}}\int_{\mathcal{B}}\left(\int_{\mathcal{B}}\mathlarger{\chi}_{A}(P_{t}(\omega^-,\omega^+))\mathbf{P}(\mathrm{d}\omega^+)\int_{\mathcal{B}}\mathlarger{\chi}_{C}(\vartheta_{t}\omega^+)\mathbf{P}(\mathrm{d}\omega^+)\right)\mathbf{P}(\mathrm{d}\omega^-)\\&\stackrel{ \hyperref[III]{\textbf{(III)}}}{\mathrel{\mathop{=}}}\mathbf{P}(C)\int_{\mathcal{B}}\mathcal{P}_t(\omega^-; A)\mathbf{P}(\mathrm{d}\omega^-)\\&\stackrel{ \hyperref[IV]{\textbf{(IV)}}}{\mathrel{\mathop{=}}}\mathbf{P}(C)\mathbf{P}(A)=\mathbf{P}\times\mathbf{P}(A\times C).
	\end{align*}
	Now, it remains to show the cocycle property. This can be verified as follows.
	\begin{align*}
		\Phi^{t+s}_{(\omega^-,\omega^+)}(x)& {\mathrel{\mathop{=}}} \varphi^{t+s}_{P_{t+s}(\omega^-,\omega^+)}(x)\\&\stackrel{ \eqref{cocycle_SDS}}{\mathrel{\mathop{=}}}\phi^{s}_{P_{t+s}(\omega^-,\omega^+)}\circ\phi^{t}_{\vartheta_{s}\circ P_{t+s}(\omega^-,\omega^+)}(x)\\&\stackrel{ \hyperref[V]{\textbf{(V)}}}{\mathrel{\mathop{=}}} \phi^{s}_{P_{s}\left(P_{t}(\omega^-,\omega^+),\vartheta_{t}\omega^+\right)}\circ\phi^{t}_{\vartheta_{s}\circ P_{s}\left(P_{t}(\omega^-,\omega^+),\vartheta_{t}\omega^+\right)}(x)\\&\stackrel{ \hyperref[V]{\textbf{(V)}}}{\mathrel{\mathop{=}}}\phi^{s}_{P_{s}\left(P_{t}(\omega^-,\omega^+),\vartheta_{t}\omega^+\right)}\circ\phi^{t}_{P_{t}(\omega^-,\omega^+)}(x)\\&\stackrel{ \eqref{COCO}-\eqref{semigr}}{\mathrel{\mathop{=}}}\Phi^{s}_{\theta_{t}(\omega^-,\omega^+)}\circ\Phi^{t}_{(\omega^-,\omega^+)}(x).
	\end{align*}
\qed
\begin{remark}\label{MJNFSA}
	From the definition, it is clear that if \(\varphi\) is \(C^{k}\) for some \(k \geq 1\), then \(\Phi\) is also \(C^{k}\).
\end{remark}
Assuming that \(\varphi\) admits an invariant measure \(\mu\) on \(\mathcal{B} \times \mathcal{X}\), it is possible to define a flow as in~\eqref{OL963ass} on \(\mathcal{X} \times \mathcal{B} \times \mathcal{B}\) that preserves the probability measure \(\mu \times \mathbf{P}\).
%\todo{$\mu$ is defined on $\mathcal{B}\times \mathcal{X}$  so the phase space should be $\mathcal{X}\times\mathcal{B}\times\mathcal{B}$,\textcolor{red}{done} }
\begin{lemma}\label{THe} %\todo{replace throughout the paper $\phi$ by $\varphi$}\todo{\textcolor{red}{done}}
	Let \(\varphi\) be a continuous SDS and suppose it admits an invariant measure \(\mu\) in the sense of Definition~\ref{measure}. For every $t\geq 0$ we set %\todo{same notation as for the skew product.\textcolor{red}{Sorry, i don't get}}
	\begin{align*}
		&\Theta_{t}:\mathcal{X}\times\mathcal{B}\times\mathcal{B} \rightarrow \mathcal{X}\times\mathcal{B}\times\mathcal{B} ,\\
		&\left(x,\omega^-,\omega^+\right)\longrightarrow \left(\Phi^{t}_{(\omega^{-},\omega^+)}(x),P_{t}(\omega^{-},\omega^+),\vartheta_{t}\omega^{+}\right).
	\end{align*}
	Then we have
	\begin{enumerate}
		\item For all \( t, s \geq 0 \), it holds that 
		\begin{align*}
			\Theta_{t+s}=\Theta_{t}\circ\Theta_{s}.
		\end{align*}
		\item This family of maps preserves the probability measure $ \mu\times\mathbf{P}$, i.e.
		\begin{align}\label{preserve_2}
			(\Theta_{t})_{\star} ( \mu\times\mathbf{P}) = \mu\times\mathbf{P}, \quad \forall t\in [0,\infty).
		\end{align} 
	\end{enumerate}
\end{lemma}
\proof
	The flow property is a simple consequence of Theorem \ref{RDS_SDS}. It remains to prove \eqref{preserve_2}. By definition for $\Gamma\times A\in \left(\sigma(\mathcal{X})\otimes\sigma(\mathcal{B})\right)\times \sigma(\mathcal{B})$, we have
	\begin{align*}
		(\Theta_{t})_{\star}(\Gamma\times A)&=\int_{\mathcal{B}}\int_{\mathcal{X}\times\mathcal{B}}\mathlarger{\chi}_{\Gamma}\left(\varphi^{t}_{P_{t}(\omega^-,\omega^+)}(x),P_{t}(\omega^{-},\omega^+)\right)\mathlarger{\chi}_{A}\left(\vartheta_{t}\omega^{+}\right)\mu(\mathrm{d}(x,\omega^-))\mathbf{P}(\mathrm{d}\omega^+)\\&=\int_{\mathcal{X}\times\mathcal{B}}\int_{\mathcal{B}}\mathlarger{\chi}_{\Gamma}\left(\varphi^{t}_{P_{t}(\omega^-,\omega^+)}(x),P_{t}(\omega^{-},\omega^+)\right)\mathlarger{\chi}_{A}\left(\vartheta_{t}\omega^{+}\right)\mathbf{P}(\mathrm{d}\omega^+)\mu(\mathrm{d}(x,\omega^-)).
	\end{align*}
	Now, from \hyperref[II]{\textbf{(II)}} and \hyperref[III]{\textbf{(III)}} and the invariance of $\mu$, we conclude that
	\begin{align*}
		(\Theta_{t})_{\star}(\Gamma\times A)&=\mathbf{P}(A)\int_{\mathcal{X}\times\mathcal{B}}\int_{\mathcal{B}}\mathlarger{\chi}_{\Gamma}\left(\varphi^{t}_{P_{t}(\omega^-,\omega^+)}(x),P_{t}(\omega^{-},\omega^+)\right)\mathbf{P}(\mathrm{d}\omega^+)\mu(\mathrm{d}(x,\omega^-))\\&=\mathbf{P}(A)\int_{\mathcal{X}\times\mathcal{B}}\mathcal{Q}_{t}\left((x,\omega^-);\Gamma\right)\mu\left(\mathrm{d}(x,\omega^-)\right)=\mathbf{P}(A)\mu(\Gamma)=(\mu\times\mathbf{P})(\Gamma\times A),
	\end{align*}
	completing the proof.
\qed
This yields the following result.
\begin{corollary}\label{MDS}
	Let $\{\Theta_{t}\}_{t \geq 0}$ be as in Lemma \ref{THe}. Then
	\[
	\left(\mathcal{X} \times \mathcal{B} \times \mathcal{B}, \; \sigma(\mathcal{X}) \otimes \sigma(\mathcal{B}) \otimes \sigma(\mathcal{B}), \; \mu \times \mathbf{P}, \; \Theta \right),
	\]
	is a metric dynamical system.
\end{corollary}
It is natural to expect a connection between the ergodicity of $\lbrace \mathcal{Q}_{t} \rbrace_{t \geq 0}$ and that of $\lbrace\Theta_{t}\rbrace_{t \geq 0}$. This is the content of the next corollary:
\begin{corollary}\label{THETA}
	Let \(\mu\) be the unique ergodic invariant measure for the transition semigroup \(\{ \mathcal{Q}_{t} \}_{t \geq 0}\), and suppose that, for every \(t > 0\), the map \(\vartheta_t : \mathcal{B} \to \mathcal{B}\) is ergodic with respect to the measure \(\mathbf{P}\). Then, for every \(t > 0\), the map 
	\[
	\Theta_{t}:\mathcal{X}\times\mathcal{B}\times\mathcal{B} \rightarrow \mathcal{X}\times\mathcal{B}\times\mathcal{B}
	\]
	is ergodic with respect to the product measure \(\mu \times \mathbf{P}\).
\end{corollary}
\proof
	The statement follows directly from the uniqueness (and thus the ergodicity) of the invariant measure~$\mu$ associated with the transition semigroup 
	$\{ \mathcal{Q}_{t} \}_{t \geq 0}$.
	%\todo{uniqueness implies ergodicity,\textcolor{red}{true}}For more details we refer to~\cite[Theorem B, page 315]{Ca85}.
\qed
The flow property of \( \{ \Theta_{t} \}_{t \geq 0} \), together with the fact that the probability measure \( \mu \times \mathbf{P} \) is preserved under this flow, allows us to define another cocycle by linearizing \( \Phi \).
\begin{corollary}\label{Thetta}
	Suppose that \( \mathcal{X} \) is a separable Banach space and \( \varphi \) is a \( C^k \)-SDS for some \( k \geq 1 \). Let \( \Phi \) be the RDS given in Theorem~\ref{RDS_SDS} and define %\todo{T should stand for a time horizon,\textcolor{red}{true}}
	\begin{align*}
		& \mathcal{T} : \R^{+}\times\mathcal{X}\times\left(\mathcal{X}\times\mathcal{B}\times\mathcal{B}\right)\rightarrow \mathcal{X},\\
		&\left(t,y,\left(x,\omega^-,\omega^+\right)\right)\longrightarrow \mathcal{T} ^{t}_{(x,\omega^-,\omega^+)}(y):= \txtD_{x}\Phi^{t}_{(\omega^-,\omega^+)}(y).
	\end{align*}
	Then, for every \( t, s \geq 0 \) and \( (x, \omega^-, \omega^+) \in \mathcal{X} \times \mathcal{B} \times \mathcal{B} \), it holds that
	\begin{align*}
		\mathcal{T}^{t+s}_{(x,\omega^-,\omega^+)}=\mathcal{T}^{s}_{\Theta_{t}(x,\omega^-,\omega^+)}\circ \mathcal{T}^{t}_{(x,\omega^-,\omega^+)}.
	\end{align*}
	In particular, \( \mathcal{T} \) is a linear continuous RDS over the metric metric dynamical system
	\[
	\left(\mathcal{X} \times \mathcal{B} \times \mathcal{B}, \, \sigma(\mathcal{X}) \otimes \sigma(\mathcal{B}) \otimes \sigma(\mathcal{B}), \, \mu \times \mathbf{P}, \, \Theta \right).
	\]
\end{corollary}
\proof
	From \eqref{cocycle} we have
	\begin{align*}
		\Phi^{t+s}_{(\omega^-,\omega^+)}(x)=\Phi^{s}_{\theta_t(\omega^-,\omega^+)}\circ\Phi^{t}_{(\omega^-,\omega^+)}(x).
	\end{align*}
	The claim follows by differentiating $\Phi$ and using the chain rule.
\qed
\paragraph*{Invariant measure for the RDS} 
For stochastic Markov processes admitting an invariant measure, it is possible to construct a family of random measures that remain invariant under the associated cocycle of the process. This can be achieved by employing the martingale convergence theorem; see \cite[Chapter 5]{V14} for further details.~Since we are dealing with non-Markovian processes, this approach is not applicable.~However, since the invariant measure $\mu$ for a continuous SDS is defined on the product space \(\mathcal{B} \times \mathcal{X}\), we can instead employ the disintegration theorem to construct invariant random measures on the state space \(\mathcal{X}\).~This result will be used in a forthcoming work to investigate the connection between attractors and invariant measures in a non-Markovian setting.  %\todo{do we use them somewhere? when we compute the LE we need the ergodicity of the shift and the marginal,\textcolor{red}{This is mainly related to the second paper, which we, of course, used alot. At the same time, we could present this part as part of our construction and mention that this construction may serve other purposes as well. If in this paper we want to emphasize the construction as well, it might be helpful to include it, although it is mostly relevant to the second paper. I would be grateful for your guidance on the best approach.
		%}yes, we include it here and state that this is important for the second part} 
\begin{lemma}\label{disintegation}
	Suppose $\varphi$ is a continuous SDS and $\mu$ is an invariant measure for this SDS. Then, there exists a family of probability measures \( \lbrace \mu_{\omega^-} \rbrace_{\omega^- \in \mathcal{B}} \) on \( \mathcal{X} \) such that for every $A\times C \in  \sigma(\cX)\otimes\sigma(\mathcal{B})$ the following properties hold:
	\begin{enumerate}
		\item [1)] $\mu_{\omega^-}(\mathcal{X}) = 1$ for $\mathbf{P}$-a.s.\ $\omega^- \in \mathcal{B}$.
		\item [2)]The function defined by $\omega^-\mapsto \mu_{\omega^-}(A)$ is measurable. 
		\item [3)] $\mu(A\times C)=\int_{C}\mu_{\omega^-}(A)\mathbf{P}(\txtd\omega^-).$
		\item [4)] The family of measures $\{\mu_{\omega^-}\}_{\omega^- \in \mathcal{B}}$ is $\mathbf{P}$-almost surely unique. 
		This means that if there exists another family of measures $\{\tilde{\mu}_{\omega^-}\}_{\omega^- \in \mathcal{B}}$ 
		satisfying the three conditions above, then for $\mathbf{P}$-almost every $\omega^- \in \mathcal{B}$ we have 
		\[
		\tilde{\mu}_{\omega^-} = \mu_{\omega^-}.
		\]
	\end{enumerate}
\end{lemma}
\proof
	Since $\mathcal{B}$ and $\mathcal{X}$ are Polish spaces and $(\Pi_{\cX})_{\star} \mu = \mathbf{P}$, the disintegration theorem entails the result.~The uniqueness also follows from the disintegration theorem; see \cite[Proposition 5.1.7 and Theorem 5.1.11]{VO16} for more details.  
\qed
Let $\Phi$ be the RDS defined in Theorem~\ref{RDS_SDS}. Thanks to the previous lemma, if the SDS $\varphi$ admits an invariant measure, then we can obtain a family of random measures on the state space $\mathcal{X}$. We naturally expect this family of random measures to be invariant for $\Phi$. This means that for every $t > 0$, on a $\theta_t$-invariant subset $\tilde{\Omega} \subset \mathcal{B} \times \mathcal{B}$ of full measure, the following invariance property holds
\begin{align}\label{sdsd}
	(\Phi^{t}_{(\omega^-, \omega^+)})_{\star} \mu_{\omega^-} = \mu_{P_{t}(\omega^-, \omega^+)}.
\end{align}
However, there always exists a set of null measure that must be neglected, and this set typically depends on the parameter \( t \). In the context of random dynamical systems, such dependence is undesirable. In what follows, we aim to address this issue by constructing a set of full measure in \( \mathcal{B} \times \mathcal{B} \) that is \( \theta \)-invariant and on which the invariance property~\eqref{sdsd} holds. Our main contribution in this context is to generate such random measures in a non-Markovian setting, without imposing any additional assumptions on the underlying measure $\mu$. Before making these statements rigorous, we first show that these random measures are a natural choice (up to a set of full measure) for the invariance property~\eqref{sdsd}.
\begin{lemma}\label{LABBAASS}
	Assume that \( \lbrace \mu_{\omega^-} \rbrace_{\omega^- \in \mathcal{B}} \) is the family of random measures obtained in Lemma \ref{disintegation}. Then for every $t\geq 0$ we can find a set of full measure $\mathcal{B}_t\subset \mathcal{B}$ such that for every $\omega^-\in \mathcal{B}_t$
	\begin{align*}
		(\varphi^{t}_{\omega^-})_{\star}(\mu_{\vartheta_{t}\omega^-})=\mu_{\omega^-}.
	\end{align*}
\end{lemma}
\proof
	First, we recall the basic fact that the composition of Borel measurable functions is again Borel measurable. This will be used throughout the argument, even when not stated explicitly.
	Since \( \varphi \) is a continuous SDS, it follows that the map
	\begin{align*}
		(t, x, \omega^-) \mapsto \varphi^t_{\omega^-}(x)
	\end{align*}
	is jointly measurable. Let \( A \times C \in  \sigma(\cX)\otimes\sigma(\mathcal{B}) \). Then from Definition \ref{DEAFRS}, Definition \ref{measure} and Lemma \ref{disintegation}, we conclude that
	\begin{align}\label{dis_2}
		\begin{split}
			&\mu(A\times C)=\int_{\mathcal{B}}\mathlarger{\chi}_{C}(\omega^-)\int_{\mathcal{X}}\mathlarger{\chi}_{A}(x)\mu_{\omega^-}(\mathrm{d} x)\mathbf{P}(\mathrm{d}\omega^-)\\&=(\mathcal{Q}_t)_{\star} \mu (A\times C)=\int_{\mathcal{B}}\int_{\mathcal{X}}\int_{\mathcal{B}}\mathlarger{\chi}_{C}\left(P_{t}(\omega^-,\omega^+)\right)\mathlarger{\chi}_{A}(\varphi^{t}_{P_{t}(\omega^-,\omega^+)}(x))\mathbf{P}(\mathrm{d}\omega^+)\mu_{\omega^-}(\mathrm{d} x)\mathbf{P}(d\omega^-)\\&=\int_{\mathcal{B}\times \mathcal{B}}\int_{\mathcal{X}}\mathlarger{\chi}_{C}\left(P_{t}(\omega^-,\omega^+)\right)\mathlarger{\chi}_{A}(\varphi^{t}_{P_{t}(\omega^-,\omega^+)}(x))\mu_{\omega^-}(\mathrm{d} x)\ (\mathbf{P}\times\mathbf{P})(\mathrm{d}(\omega^-,\omega^+)).
		\end{split}
	\end{align}
	We set 
	\begin{align*}
		\tilde{R}(\omega^-)=R(\omega^-,\omega^+):=\int_{\mathcal{X}}\mathlarger{\chi}_{A}(\varphi^{t}_{\omega^-}(x))\mathlarger{\chi}_{C}(\omega^-)\mu_{\vartheta_{t}\omega^-}(\mathrm{d} x).
	\end{align*}
	Thus, from \eqref{dis_2}, Lemma \ref{pushforward} and  \eqref{preserve}, we get 
	\begin{align}\label{dis3}
		\begin{split}
			&\int_{\mathcal{B}}\mathlarger{\chi}_{C}(\omega^-)\int_{\mathcal{X}}\mathlarger{\chi}_{A}(x)\mu_{\omega^-}(\mathrm{d} x)\mathbf{P}(d\omega^-)=\int_{\mathcal{B}\times \mathcal{B}}R(\theta_t(\omega^-, \omega^+))(\mathbf{P}\times\mathbf{P})\mathrm{d}(\omega^-,\omega^{+})\\&\quad=\int_{\mathcal{B}\times \mathcal{B}}R(\omega^-, \omega^+)(\mathbf{P}\times\mathbf{P})\mathrm{d}(\omega^-,\omega^{+})=\int_{\mathcal{B}}\mathlarger{\chi}_{C}(\omega^-)\int_{\mathcal{X}}\mathlarger{\chi}_{A}(\varphi^{t}_{\omega^-}(x))\mu_{\vartheta_{t}\omega^-}(\mathrm{d} x)\mathbf{P}(\txtd\omega^-)\\ &\qquad=\int_{\mathcal{B}}\mathlarger{\chi}_{C}(\omega^-)\int_{\mathcal{X}}\mathlarger{\chi}_{A}(x)(\varphi^{t}_{\omega^-})_{\star}(\mu_{\vartheta_{t}\omega^-})(\mathrm{d} x)\mathbf{P}(d\omega^-).
		\end{split}
	\end{align}
	Since \( A \times C \) was chosen arbitrarily, it follows from statetement 4) of Lemma~\ref{disintegation} that on a set of full measure \( \mathcal{B}_{t} \subseteq \mathcal{B} \), we have
	\begin{align*}
		(\varphi^{t}_{\omega^-})_{\star}(\mu_{\vartheta_{t}\omega^-})=\mu_{\omega^-}.
	\end{align*}
\qed
\begin{remark} 
	Let \( t \geq 0 \) be fixed. Then, for every \( C \in \sigma(\mathcal{B}) \), from \hyperref[IV]{\textbf{(IV)}} and the Fubini theorem, we have
	\begin{align*}
		(P_t)_{\star}(\mathbf{P}\times\mathbf{P})(C)&=\mathbf{P}\times\mathbf{P}\left\lbrace(\omega^-,\omega^+)\in\mathcal{B}\times\mathcal{B} : \ P_{t}(\omega^-,\omega^+)\in C\right)\rbrace\\&=\int_{\mathcal{B}}\mathcal{P}_{t}(\omega^-,C)\mathbf{P}(\mathrm{d}\omega^-)=\mathbf{P}(C).
	\end{align*}
	Therefore, we can find a set of full measure \( \tilde{\mathcal{B}}_t \) such that, for every \( \omega^- \in \tilde{\mathcal{B}}_t \), there exists another set of full measure \( \tilde{\mathcal{B}}_{t, \omega^-} \) with the property that if \( \omega^+ \in \tilde{\mathcal{B}}_{t, \omega^-} \), then
	\begin{align*}
		P_{t}(\omega^-, \omega^+) \in \mathcal{B}_{t}.
	\end{align*}
	In particular, this yields that
	\begin{align}\label{AISOAs}
		\varphi^{t}_{(\omega^-,\omega^+)}(\mu_{\omega^-})=\mu_{P_{t}(\omega^-,\omega^+)},
	\end{align}
	for $(\omega^-,\omega^+)\in\tilde{\mathcal{B}}_t\times \tilde{\mathcal{B}}_{t, \omega^-} $.
\end{remark}
As noted in the previous remark, although the invariance property \eqref{AISOAs} holds, its validity is restricted to a set that may depend on $t$ and $\omega^{-}$ which is contradicts the RDS framework. Moreover, these sets are not necessarily invariant under the shift map \( \theta \). In the next theorem, we rigorously resolve this problem. This result can be interpreted as a perfection-type result in the theory of random dynamical systems.
%\textcolor{red}{This theorem is not polished}
\begin{theorem}\label{disintegration_2}
	There exist a set of full measure \( \tilde{\mathcal{B}} \) in \( \mathcal{B} \) and a family of random measures  
	\( \{ \tilde{\mu}_{\omega^-} \}_{\omega^- \in \mathcal{B}} \) satisfying the assumptions of Lemma \ref{disintegation} (i.e., another disintegration of \( \mu \)), such that for every \( (\omega^-,\omega^+) \in \tilde{\mathcal{B}} \times \mathcal{B} \), we have 
	\begin{align}\label{OLAefa}
		\begin{split}
			&P_{t}(\omega^-,\omega^+) \in \tilde{\mathcal{B}},\\
			&\left(\varphi_{P_{t}(\omega^-,\omega^+)}^{t}\right)_{\star} \tilde{\mu}_{\omega^-} = \tilde{\mu}_{P_{t}(\omega^-,\omega^+)}.
		\end{split}
	\end{align}
	
\end{theorem}
\proof
	Similar to Lemma~\ref{LABBAASS} we use the fact that the composition of Borel measurable functions is again Borel measurable. For every $\omega^-\in\mathcal{B}$, we set
	\begin{align*}
		&\hat{\mu}_{\omega^-}:=\limsup_{n\rightarrow \infty}\frac{1}{n}\int_{0}^{n}(\varphi^{s}_{\omega^-})_{\star}\mu_{\vartheta_{s}\omega^-}\mathrm{d}s,\\
		& \check{\mu}_{\omega^-}:=\liminf_{n\rightarrow \infty}\frac{1}{n}\int_{0}^{n}(\varphi^{s}_{\omega^-})_{\star}\mu_{\vartheta_{s}\omega^-}\mathrm{d}s,
	\end{align*}
	where this means that for every \( C \in \sigma(\mathcal{B}) \)
	\begin{align*}
		\hat{\mu}_{\omega^-}(C) = \limsup_{n \to \infty} \frac{1}{n} \int_0^n (\varphi^s_{\omega^-})_{\star} \mu_{\vartheta_s \omega^-}(C)\mathrm{d}s,
	\end{align*}
	and similarly for \( \check{\mu}_{\omega^-} \).
	First, note that these Lebesque integrals are well-defined. Indeed, for every \( C \in \sigma(\mathcal{X}) \), the following map is jointly measurable
	\begin{align}\label{AAS}
		(s,\omega^-)\longrightarrow (\varphi^{s}_{\omega^-})_{\star}\mu_{\vartheta_{s}\omega^-}(C)= \int_{\mathcal{X}}\mathlarger{\chi}_{C}(\varphi^{s}_{\omega^-}(x))\mu_{\vartheta_{s}\omega^-}(\mathrm{d}x).
	\end{align}
	% (cf. \cite[Propitiation 2.34, Lemma 2.35, Theorem 2.36]{Fol99})%
	To establish this, we begin with the jointly measurable function
	\[
	(s, x, \omega^-) \mapsto \mathlarger{\chi}_{C}(\varphi^{s}_{\omega^-}(x)).
	\]
	%  \todo{how can we apply (3.13) to obtain the measurability of (3.13),\textcolor{red}{sorry it was typo}}
	We then apply the monotone class lemma~\cite[Lemma~2.35]{Fol99}, together with 2) in Lemma~\ref{disintegation} and the fact that $\vartheta_s:\mathcal{B}\to\mathcal{B} $ is a measurable map, to conclude the measurability of the map defined in \eqref{AAS}; see also \cite[Proposition~2.34 and Theorem~2.36]{Fol99}.
	In particular, for every $\omega^-\in\mathcal{B}$, the map
	\[
	s \mapsto (\varphi^{s}_{\omega^-})_{\star}\mu_{\vartheta_{s}\omega^-}(C)
	\]
	is measurable. Therefore, the associated integrals are well-defined.~Moreover, for every $n\geq 1$, the following map 
	\begin{align*}
		\omega^-\longrightarrow\frac{1}{n} \int_0^n (\varphi^s_{\omega^-})_{\star} \mu_{\vartheta_s \omega^-}(C)\mathrm{d}s
	\end{align*}
	is measurable. This in particular yields that $\hat{\mu}_{\omega^-}(C)$ and $\check{\mu}_{\omega^-}(C)$ are measurable (since the lim sup and lim inf are taken over a countable set). Additionally, note that from  \eqref{dis_2} and \eqref{dis3}, for every \(n \geq 1\), the family of \(\left\{ \frac{1}{n}\int_{0}^{n}(\varphi^{s}_{\omega^-})_{\star}\mu_{\vartheta_{s}\omega^-}\mathrm{d}s \right\}_{\omega^-\in\mathcal{B}}\) satisfies the assumption in Lemma \ref{disintegation}. 
	Therefore, by Property~4) of this lemma, for every \( n \in \mathbb{N} \), there exists a measurable subset \( \mathcal{B}^n \subseteq \mathcal{B} \) of full measure such that, for every \( \omega^- \in \mathcal{B}^n \), we have  
	\begin{align}\label{NNNa}  
		\frac{1}{n} \int_{0}^{n} (\varphi^{s}_{\omega^-})_{\star} \mu_{\vartheta_{s} \omega^-} \, \mathrm{d}s = \mu_{\omega^-}.
	\end{align}  
	In particular, by setting \( \mathcal{B}^\prime := \bigcap_{n \in \mathbb{N}} \mathcal{B}^n \), we obtain a measurable subset of \( \mathcal{B} \) with full measure such that \eqref{NNNa} holds for every \( n \in \mathbb{N} \) and every \( \omega^- \in \mathcal{B}^\prime \). This implies that the families \( \lbrace \hat{\mu}_{\omega^-} \rbrace_{\omega^- \in \mathcal{B}} \) and \( \lbrace \check{\mu}_{\omega^-} \rbrace_{\omega^- \in \mathcal{B}} \) also satisfy the assumptions of Lemma \ref{disintegation}. Indeed, the first two conditions in Lemma~\ref{disintegation} follow immediately. For the third condition, note that, using \eqref{dis3} and the Fubini theorem, we obtain
	\begin{align*}
		\int_{C}\mu_{\omega^-}(A)\,\mathbf{P}(\mathrm{d}\omega^-)
		&= \frac{1}{n}\int_{0}^{n}\int_{C} (\varphi^{s}_{\omega^-})_{\star}\mu_{\vartheta_{s}\omega^-}(A)\,\mathbf{P}(\mathrm{d}\omega^-)\,\mathrm{d}s
		\\& = \int_{C}\bigg(\frac{1}{n}\int_{0}^{n}(\varphi^{s}_{\omega^-})_{\star}\mu_{\vartheta_{s}\omega^-}\,\mathrm{d}s\bigg)(A)\mathbf{P}(\mathrm{d}\omega^-).
	\end{align*}
	%  we begin by using   \(\frac{1}{n}\int_{0}^{n}(\varphi^{s}_{\omega^-})_{\star}\mu_{\vartheta_{s}\omega^-}\mathrm{d}s\) in \eqref{dis3}
	%\todo{state how this is used in (3.10),\textcolor{red}{done}}
	Thus, since \eqref{NNNa} holds for every \( n\in\N  \) in \( \mathcal{B}^\prime \), we can apply the dominated convergence theorem to interchange the limit and the integral. This verifies the third condition for both families \( \lbrace \hat{\mu}_{\omega^-} \rbrace_{\omega^- \in \mathcal{B}} \) and \( \lbrace \check{\mu}_{\omega^-} \rbrace_{\omega^- \in \mathcal{B}} \), which are equal on \( \mathcal{B}^\prime \).
	The uniqueness statement in Lemma \ref{disintegation} indicates that for
	\begin{align*}
		\tilde{\mathcal{B}}:=\{\omega^- \in \mathcal{B}: \hat{\mu} = \check{\mu}\},
	\end{align*}
	we have \( \mathbf{P}(\tilde{\mathcal{B}}) = 1 \). Note that \( \mathcal{B}^\prime \subseteq \tilde{\mathcal{B}} \). %However, as will become clear shortly, \( \tilde{\mathcal{B}} \) also satisfies other important invariance properties.
	Now, we prove the for every $\omega^{-}\in \tilde{B}$, $\omega^+\in \mathcal{B}$ and $t\geq 0$
	\begin{align}\label{AYsh}
		P_{t}(\omega^-,\omega^+)\in\tilde{\mathcal{B}}.
	\end{align}
	To this aim, first note that by the cocycle property of \( \varphi \) together with Property~\hyperref[V]{\textbf{(V)}} in Definition~\ref{SDS}, we have
	\[
	\varphi^{t+s}_{P_{t}(\omega^-,\omega^+)} = \varphi^{t}_{P_{t}(\omega^-,\omega^+)} \circ \varphi^{s}_{\omega^-}.
	\]
	Therefore
	\begin{align*}
		\int_0^n (\varphi^{t+s}_{P_{t}(\omega^-,\omega^+)})_{\star} \mu_{\vartheta_s \omega^-}(C)\mathrm{d}s=\int_0^n (\varphi^{s}_{\omega^-})_{\star} \mu_{\vartheta_s \omega^-}\left( \left(\varphi_{P_{t}(\omega^-,\omega^+)}^{t}\right)^{-1}(C)\right)\mathrm{d}s.
	\end{align*}
	Since \(\left(\varphi_{P_{t}(\omega^-,\omega^+)}^{t}\right)^{-1}(C) \in \sigma(\mathcal{X})\) and \(\omega^- \in \tilde{\mathcal{B}}\), we conclude that
	\begin{align}\label{ATGAs}
		\lim_{n\rightarrow\infty}\frac{1}{n}\int_0^n (\varphi^{t+s}_{P_{t}(\omega^-,\omega^+)})_{\star} \mu_{\vartheta_s \omega^-}(C)\mathrm{d}s, \ \ \ \text{exists}.
	\end{align}
	Note that due to~\eqref{V} we have $\vartheta_{t+s} P_{t}(\omega^-,\omega^+)= \vartheta_s P_t(\vartheta_{t}\omega_{-}, P_t(\omega_+,\omega_-))=\vartheta_{s}\omega^-$. This implies that
	\begin{align*}
		&\int_0^n (\varphi^{t+s}_{P_{t}(\omega^-,\omega^+)})_{\star} \mu_{\vartheta_s \omega^-}(C)\mathrm{d}s=\int_{0}^{n}(\varphi^{t+s}_{P_{t}(\omega^-,\omega^+)})_{\star} \mu_{\vartheta_{t+s} P_{t}(\omega^-,\omega^+)}(C)\mathrm{d}s\\&\quad=\int_{t}^{n+t}(\varphi^{s}_{P_{t}(\omega^-,\omega^+)})_{\star} \mu_{\vartheta_{s} P_{t}(\omega^-,\omega^+)}(C)\mathrm{d}s.
	\end{align*}
	Therefore, form \eqref{ATGAs}
	\begin{align*}
		\lim_{n\rightarrow\infty}\frac{1}{n}\int_0^n (\varphi^{s}_{P_{t}(\omega^-,\omega^+)})_{\star} \mu_{\vartheta_{s} P_{t}(\omega^-,\omega^+)}(C)\mathrm{d}s \ \ \ \text{exists}.
	\end{align*}
	This yields \eqref{AYsh}. In particular
	\begin{align}\label{sfddad}
		\begin{split}
			\lim_{n\rightarrow\infty}\frac{1}{n}&\int_0^n (\varphi^{t+s}_{P_{t}(\omega^-,\omega^+)})_{\star} \mu_{\vartheta_s \omega^-}(C)\mathrm{d}s=\lim_{n\rightarrow\infty}\frac{1}{n}\int_0^n (\varphi^{s}_{P_{t}(\omega^-,\omega^+)})_{\star} \mu_{\vartheta_{s} P_{t}(\omega^-,\omega^+)}(C)\mathrm{d}s\\&= \check{\mu}_{P_{t}(\omega^-,\omega^+)}(C)=\hat{\mu}_{P_{t}(\omega^-,\omega^+)}(C).
		\end{split}
	\end{align}
	Now we prove that for \( (\omega^-, \omega^+) \in \tilde{\mathcal{B}} \times \mathcal{B} \) and \( t \geq 0 \),
	\begin{align}\label{55sd65fr}
		\begin{split}  
			\left(\varphi_{P_{t}(\omega^-,\omega^+)}^{t}\right)_{\star} \hat{\mu}_{\omega^-}
			&= \hat{\mu}_{P_{t}(\omega^-,\omega^+)} \\
			&= \left(\varphi_{P_{t}(\omega^-,\omega^+)}^{t}\right)_{\star} \check{\mu}_{\omega^-}
			= \check{\mu}_{P_{t}(\omega^-,\omega^+)}.
		\end{split}
	\end{align}
	First, note that by definition
	\begin{align*}
		&\left(\varphi_{P_{t}(\omega^-,\omega^+)}^{t}\right)_{\star}\hat{\mu}_{\omega^-}(C)=\hat{\mu}_{\omega^-}\left( \left(\varphi_{P_{t}(\omega^-,\omega^+)}^{t}\right)^{-1}(C)\right)=\\
		&\limsup_{n \to \infty} \frac{1}{n} \int_0^n (\varphi^s_{\omega^-})_{\star} \mu_{\vartheta_s \omega}\left( \left(\varphi_{P_{t}(\omega^-,\omega^+)}^{t}\right)^{-1}(C)\right)\mathrm{d}s=\limsup_{n\rightarrow\infty}\frac{1}{n}\int_0^n (\varphi^{t+s}_{P_{t}(\omega^-,\omega^+)})_{\star} \mu_{\vartheta_s \omega^-}(C)\mathrm{d}s.
	\end{align*}
	The same arguments apply to \( \check{\mu} \) as well. Therefore, together with \eqref{sfddad}, we can deduce \eqref{55sd65fr}.
	\begin{comment}
		\begin{align*}
			\left(\varphi_{P_{t}(\omega^-,\omega^+)}^{t}\right)_{\star} \hat{\mu}_{\omega^-} = \left(\varphi_{P_{t}(\omega^-,\omega^+)}^{t}\right)_{\star} \check{\mu}_{\omega^-} = \check{\mu}_{P_{t}(\omega^-,\omega^+)}=\hat{\mu}_{P_{t}(\omega^-,\omega^+)}.
		\end{align*}
	\end{comment}
	
	%Assume now that \(\omega^- \in \tilde{B}\), then
	%\begin{align*}
	%\frac{1}{n}\int_0^n (\varphi^{s}_{\omega^-})_{\star} \mu_{\vartheta_{t+s}\omega^-}(C)\mathrm{d}s
	%\end{align*}	
	Now we can finish the proof by setting either \( \tilde{\mu}_{\omega^-} = \hat{\mu}_{\omega^-} \) or \( \tilde{\mu}_{\omega^-} = \check{\mu}_{\omega^-} \) for every \( \omega^- \in \mathcal{B} \). Note that \( \hat{\mu}_{\omega^-} \) and \( \check{\mu}_{\omega^-} \) are not necessarily equal on \( \mathcal{B} \setminus \tilde{\mathcal{B}} \); however, the measure of this set is zero.
\qed
The preceding theorem yields the following result, which implies that the set \( \tilde{\mathcal{B}}\times\mathcal{B} \) is \( \theta \) invariant.
\begin{corollary}\label{random_measure}
	Let \( \varphi \) be a continuous SDS with invariant measure \( \mu \), and let \( \Phi \) be the continuous RDS defined in Theorem~\ref{RDS_SDS}. Under the assumptions of Theorem \ref{disintegration_2}, the following statements hold:
	\begin{enumerate}
		\item $\mathbf{P}\times\mathbf{P}(\tilde{\mathcal{B}}\times \mathcal{B})=1$.
		\item  $\theta_{t}(\tilde{\mathcal{B}}\times \mathcal{B})\subseteq \tilde{\mathcal{B}}\times \mathcal{B}$ for every $t\geq 0$
		\item For every element \((\omega^-, \omega^+) \in \tilde{\mathcal{B}} \times \mathcal{B}\) and every $t\geq 0$ we have
		\begin{align*}
			(\Phi^{t}_{(\omega^-,\omega^+)})_{\star}\mu_{\omega^-}=\mu_{P_{t}(\omega^-,\omega^+)}.
		\end{align*}
		%\item  For every $(\omega^-,\omega^+)\in \mathcal{B}}\times \mathcal{B} $ and $t\in \R$, we have $\theta_{t}(\omega^-,\omega^+) \in$ 
\end{enumerate}
\end{corollary}
\proof
The first property is clear. To prove the remaining properties, recall that for \( t>0 \) and \( x \in \mathcal{X} \),
\begin{align*}
	\theta_{t}(\omega^-,\omega^+) = \left(P_{t}(\omega^-,\omega^+), \vartheta_{t}\omega^+\right),
\end{align*}
and that
\[
\Phi^{t}_{(\omega^-,\omega^+)}(x) = \varphi^{t}_{P_{t}(\omega^-,\omega^+)}(x).
\]
Thus, as a direct consequence of Theorem~\ref{disintegration_2} %\todo{wrong reference,\textcolor{red}{thanks}}, 
the remaining properties can be derived using these relations.
\qed
\section{Stochastic differential equations driven by fractional Brownian motion}\label{sec:sde}
%\textcolor{blue}{We refer throughout the manuscript to them with Assumptions~\ref{Drift} and never write we accept the assumptions, for e.g. Let Assumptions~\ref{Drift} hold or under the Assumptions~\ref{Drift}}
%Up to this point, we have introduced the fundamental concepts of stochastic and random dynamical systems in an abstract setting. A connection between these two notions has been established. We have also described how a stationary noise process can be associated with fractional Brownian motion.
The aim of this section is to show how to obtain a stochastic dynamical system and thereafter a random dynamical system from stochastic differential equations driven by fractional Brownian motion. More precisely, we consider the following stochastic differential equation
\begin{align}\label{MAIN}
\begin{cases}
	\mathrm{d}Y_t = F(Y_t)\, \mathrm{d}t + \sigma \mathrm{d}B^H_t, \\
	Y_0=x \in \mathbb{R}^d,
\end{cases}
\end{align}
where \( d \geq 1 \) and \( (B^H_t)_{t\geq 0} \) denotes a \( d \)-dimensional fractional Brownian motion with Hurst parameter \( H \in (0,1) \).
%Most of the concepts and results presented here are derived from \cite{Hai05}, where rigorous proofs can be found. The purpose of this overview is to summarize the key points succinctly. Let us first express the problem and also the assumptions that we consider:
\begin{assumptions}\label{Drift} 
%	\todo{the sde is not part of the assumptions. we consider it and then state the assumptions on the coefficients ,\textcolor{red}{done}}
We make the following assumptions on the coefficients of~\eqref{MAIN}.
\begin{enumerate}
	%	\item For $H\in (0,1)\setminus\lbrace\frac{1}{2}\rbrace$, we consider \( B^H = (B^H_1, \ldots, B^H_d) : \mathbb{R}^+ \to \mathbb{R}^d \) as an \(\mathbb{R}^d\)-valued fractional Brownian motion with independent components and Hurst parameter \( H \). The process is defined on a probability space \((\Omega, \mathcal{F}, \mathbb{P})\) and is adapted to a filtration \((\mathcal{F}^t)_{t \in \mathbb{R}^+}\), with \( B_0 = 0 \) almost surely.
	\item [1)] Let \(\sigma : \mathbb{R}^d \rightarrow \mathbb{R}^d\) be an invertible matrix.
	\item [2)]  For the vector field \( F : \mathbb{R}^d \to \mathbb{R}^d \), we assume the existence of constants \( C^F_i > 0 \) for \( i \in \{1,2,3\} \), such that for all \( \xi_1, \xi_2 \in \mathbb{R}^d \),
	\begin{align*}
		\langle F(\xi_2) - F(\xi_1), \xi_2 - \xi_1 \rangle \leq \min \left\{ C_1^F - C_2^F |\xi_2 - \xi_1|^2, C_3^F |\xi_2 - \xi_1|^2 \right\}.
	\end{align*}
	\item [3)] We further assume that \( F \) is differentiable and satisfies the following polynomial growth condition for some constant \( C_F > 0 \) and \( N \geq 1 \):
	\begin{align*}
		|F(\xi)| \leq C_F(1 + |\xi|)^N, \quad |\mathrm{D}_\xi F| \leq C_F(1 + |\xi|)^N \quad \text{for all } \xi \in \mathbb{R}^d.
	\end{align*}
	\item [4)] If \( \frac{1}{2} < H < 1 \), we assume that the derivative of \( F \) is globally bounded.
\end{enumerate}
\end{assumptions}
%Now, we briefly explain how we can generate a stochastic dynamical system according to the solution of equation \ref{MAIN}. For the precise proofs, we refer to the original paper \cite{Hai05}. 
%During the discussion, we assume that Assumption \ref{Drift} holds.

\subsection{Well-posedness and generation of a random dynamical system}
%\textcolor{red}{Since changing the symbol throughout the paper would require many modifications and might introduce unintended errors, I have chosen to follow the approach of using the capital letter $Y$ to denote the solution.
%I have also made sure that whenever a solution appears, it is consistently written in capital letters ($Y_t$ or $Z_t$), leaving lowercase letters available for other uses. In this way, we are avoiding the same notation for different purposes with minimal changes to the text.
%The paper has now been updated according to this convention. There might still be a few places where this convention has not yet been applied, and we will find them.
%}\\
For the sake of completeness, we first recall the solution concept of~\eqref{MAIN} together with a well-posedness result established in~\cite{Hai05}.~This ensures that~\eqref{MAIN} generates a continuous RDS from which we can obtain a continuous RDS from Theorem~\ref{RDS_SDS} . 
\begin{definition}\label{SASDdfe}
Let $(Y_t)_{t \geq 0}$ be a stochastic process with continuous trajectories. We call \(Y\) a solution of the SDE \eqref{MAIN}, if the stochastic process 
\begin{align*}
	Y_t-Y_0-\int_{0}^{t}F(Y_s)\mathrm{d}s,
\end{align*}
is equal in law to $\sigma B^{H}_t$.
\end{definition}
\begin{remark}
\item [1)] One can solve the SDE~\eqref{MAIN} for a broader class of initial data, called generalized initial conditions~\cite[Definition 2.14]{Hai05}). However, since we are also interested in a random dynamical systems approach, we work with initial data $Y_0=x\in \R^d$.% belonging to \(\mathbb{R}^d\).
\item [2)] The SDE~\eqref{MAIN} can be solved pathwise, using a standard flow transformation by subtracting the noise.
\end{remark}

\begin{lemma}{\em(\cite[Lemma 3.9]{Hai05})}\label{OALSdfe}
Let Assumption~\ref{Drift} hold and fix \(T > 0\). For \(x \in \mathbb{R}^d\) %\todo{combine this with lemma 4.8 in one statement ,\textcolor{red}{Done. also the subsequence remark is updated}}
and \(\omega \in C_0([0,T], \mathbb{R}^d)\), the ODE with random non-autonomous coefficients
\begin{align}\label{ASdd}
	\dot{Z}_t = F\big(Z_t + x + \sigma\omega(t)\big), \quad Z_0 = 0,
\end{align}
has a unique solution for \(t \in [0,T]\). In particular, setting
\begin{align}\label{Sjdks}
	\tilde{\Phi}_T(t,x,\omega):=\tilde{\Phi}_T(x,\omega)(t) := Z_t + x + \sigma\omega(t),
\end{align}
we have for all \(t \in [0,T]\) that
\begin{align}\label{Sjdks33}
	\tilde{\Phi}_T(t,x,\omega) = x + \sigma \omega(t) + \int_0^t F\big(\tilde{\Phi}_T(s,x,\omega)\big) \, \mathrm{d}s.
\end{align}
Moreover, the map
\begin{align*}
	\tilde{\Phi}_T : \mathbb{R}^d \times C_0([0,T], \mathbb{R}^d) &\to C([0,T], \mathbb{R}^d), \\
	(x,\omega) &\mapsto \tilde{\Phi}_T(x,\omega),
\end{align*}
is locally Lipschitz continuous. Also, for
\begin{align}\label{TGBAS}
	\begin{split}
		\varphi : \mathbb{R}^+ \times \mathbb{R}^d \times \mathcal{B}_H &\to \mathbb{R}^d, \\
		(t,x,\omega) &\mapsto \varphi^t_\omega(x) := \tilde{\Phi}_t\bigl(t, x, R_t \mathcal{D}_H \omega\bigr),
	\end{split}
\end{align}
it holds that \( \varphi \) is a continuous stochastic dynamical system over \( \left(\mathcal{B}_H, \{P_t\}_{t \geq 0}, \mathbf{P}, \{\vartheta_t\}_{t \geq 0} \right) \).
\end{lemma}

\begin{corollary}\label{CCC11}
For each $\omega \in C_0([0,T]; \mathbb{R}^d)$ and $t \in [0,T]$, the mapping
\[
x \longmapsto \tilde{\Phi}_t(t, x, \omega)
\]
is of class $C^1$ with respect to $x$. Moreover, $\varphi$ is a $C^1$ stochastic dynamical system.
\end{corollary}
\proof
This follows from \eqref{ASdd} and the \(C^1\)-regularity of the vector field \(F\).
\qed
\begin{remark}\label{REMMAS} 
For every \(0 \leq s \leq t\), \(\omega \in C_{0}([0,t], \mathbb{R}^d)\) and \(x \in \mathbb{R}^d\), we have
\begin{align*}
	\tilde{\Phi}_t(s, x, \omega)=\tilde{\Phi}_s(s, x, \omega),
\end{align*}
since 
\begin{align*}
	\tilde{\Phi}_t(s, x, \omega)= x+\sigma \omega(s) +\int_0^s F( \tilde{\Phi}_t(r,x,\omega) )~\txtd r
\end{align*}
and 
\begin{align*}
	\tilde{\Phi}_s(s, x, \omega) = x+\sigma \omega(s) + \int_0^s F( \Tilde{\Phi}_s(r,x,\omega) )~\txtd r. 
\end{align*}

%\textcolor{red}{Thanks, this is exactly the point you made. This means that these two trajectories solve the same equation, and therefore they are identical. } \textcolor{blue}{i see the argument but we get once $\int_0^s F( \Tilde{\Phi}_s(r,x,\omega) )~\txtd r$ and once $\int_0^s F(\tilde{\Phi}_t(r,x,\omega) )~\txtd r$, the time interval of the two trajectories must also coincide, i don't think that this follows from the uniqueness of solutions but just by notation that we solve the sde on $[0,s]$.  }
\end{remark}
As a consequence of Lemma \ref{OALSdfe}, we can now construct a stochastic dynamical system for the SDE \eqref{MAIN}.~This will allow us to obtain an RDS by Theorem~\ref{RDS_SDS}.  First, we need the following definition.
\begin{definition}
For \( t \geq 0 \) we define the continuous shift operator
\begin{align} \label{R_RT}
	\begin{split}
		R_t : C_0((-\infty,0], \mathbb{R}^d) &\to C_0([0,t], \mathbb{R}^d), \\
		(R_t \omega)(s) &:= \omega(s - t) - \omega(-t), \quad s \in [0,t].
	\end{split}
\end{align}
\end{definition}
Based on this, we introduce a continuous SDS associated with the SDE \eqref{MAIN}. To this aim, we first recall that process \(\left(\mathcal{B}_H, \{P_t\}_{t \geq 0}, \mathbf{P}, \{\vartheta_t\}_{t \geq 0}\right)\) is the stationary noise process obtained in Lemma~\ref{Olafads}. %\todo{wrong reference,\textcolor{red}{now true!}}
\begin{comment}  
\begin{lemma}{\em(\cite[Lemma 3.9]{Hai05})}\label{SCXVZBS}
	We define
	\begin{align}\label{TGBAS}
		\begin{split}
			\varphi : \mathbb{R}^+ \times \mathbb{R}^d \times \mathcal{B}_H &\to \mathbb{R}^d, \\
			(t,x,\omega) &\mapsto \varphi^t_\omega(x) := \tilde{\Phi}_t\bigl(t, x, R_t \mathcal{D}_H \omega\bigr).
		\end{split}
	\end{align}
	Then \(\varphi\) is a  continuous stochastic dynamical system over \(\left(\mathcal{B}_H, \{P_t\}_{t \geq 0}, \mathbf{P}, \{\vartheta_t\}_{t \geq 0}\right)\).
\end{lemma}
\begin{remark}\label{AISapsqd}
	Thanks to Corollary~\ref{CCC11}, \(\varphi\) is a \(C^1\)- stochastic dynamical system.
\end{remark}
\end{comment}

%\textcolor{red}{The rest is not polished}
Having established the existence of a continuous SDS \(\varphi\), we can now apply Theorem~\ref{RDS_SDS} to obtain a continuous RDS.

\begin{corollary}\label{RRRDDDSS}
%	\todo{There exists a continuous $\R^d$-valued RDS $\Phi$ over the metric dynamical system...,\textcolor{red}{done}}
There exists a continuous  $\R^d$-valued RDS \(\Phi\) over the metric dynamical system
\[
\left(\mathcal{B}_H \times \mathcal{B}_H, \; \sigma(\mathcal{B}_H) \otimes \sigma(\mathcal{B}_H), \; \mathbb{P}, \; \theta \right),
\]
where $\mathbb{P}=\mathbf{P} \times \mathbf{P}$. %\( \lbrace\theta_t\rbrace_{t\in\R} \) is defined by \eqref{semigr} and $\mathbb{P}=\mathbf{P} \times \mathbf{P}$.
\begin{comment}
	\((\R^d\times \Omega_H, \mathcal{F}, \mathbb{P})\), where
	\[
	\Omega_H := \mathcal{B}_H \times \mathcal{B}_H, \quad \mathcal{F} := \sigma(\mathcal{B}_H) \otimes \sigma(\mathcal{B}_H), \quad \mathbb{P} := \mathbf{P} \times \mathbf{P}.
	\]
\end{comment}
In addition, for $(\omega^-,\omega^+)\in \mathcal{B}_H \times \mathcal{B}_H $ %\todo{above we write $\cB_H\times \cB_H$,\textcolor{red}{done}}and $t\geq 0$, we have
\begin{align}
	\Phi^{t}_{(\omega^-,\omega^+)}(x)= \tilde{\Phi}_t\left(t,x,(R_{t}	\mathcal{D}_{H}P_{t}(\omega^-,\omega^+))\right).
\end{align}
\end{corollary}

\begin{comment}
content...

and be definition 
\begin{align*}
	\Phi^{t}_{(\omega^-,\omega^+)}(x)=\tilde{z}_{t}+x+(R_{t}	\mathcal{D}_{H}P_{t}(\omega^-,\omega^+))(t)
\end{align*}

Also, we have
Recall that for path $\tilde{z}$ statifyinh 
\begin{align*}
	Z_{t}=\int_{0}^{t}F()
\end{align*}

\begin{align*}
	\Phi^{t}_{(\omega^-,\omega^+)}(x)=Z_{t}+x+(R_{t}	\mathcal{D}_{H}P_{t}(\omega^-,\omega^+))(t)=Z_{t}+x-\mathcal{D}_{H}P_{t}(\omega^-,\omega^+)(-t)
\end{align*}
\end{comment}
We now show that \(\Phi\) solves the SDE \eqref{MAIN} in the sense of Definition~\ref{SASDdfe}.
\begin{corollary}\label{DCHYJAS85}
For every \(x \in \mathbb{R}^d\), \((\omega^-, \omega^+) \in \Omega_H\) and \(t \geq 0\), it holds that
\begin{align*}
	\Phi^{t}_{(\omega^-,\omega^+)}(x)=\int_{0}^{t}F\left(\Phi^{s}_{(\omega^-,\omega^+)}(x)\right)\mathrm{d}s+x-\sigma %B^H_t(P_t(\omega^-,\omega^+)). %
	(\mathcal{D}_{H}P_{t}(\omega^-,\omega^+))(-t).
\end{align*}
\end{corollary}
\proof
For \(0 \leq s \leq t\), we claim that 
\begin{align}\label{BII}
	\tilde{\Phi}_t\left(s,x,(R_{t}	\mathcal{D}_{H}P_{t}(\omega^-,\omega^+))\right)=\tilde{\Phi}_s\left(s,x,(R_{s}	\mathcal{D}_{H}P_{s}(\omega^-,\omega^+))\right)=\Phi^{s}_{(\omega^-,\omega^+)}(x).
\end{align}
Since \( R_t \mathcal{D}_H P_t(\omega^-, \omega^+) \in C_0([0,t], \mathbb{R}^d) \), by Remark~\ref{REMMAS} we have
\begin{align}\label{AII}
	\tilde{\Phi}_t\left(s,x,(R_{t}	\mathcal{D}_{H}P_{t}(\omega^-,\omega^+))\right)=\tilde{\Phi}_s\left(s,x,(R_{t}	\mathcal{D}_{H}P_{t}(\omega^-,\omega^+))\right).
\end{align}
From Lemma~\ref{OALSdfe}, one can prove \eqref{BII} by showing that for all   $\tau\in[0,s] $ we have
\begin{align*}
	\quad \big(R_{t} \mathcal{D}_{H} P_{t}(\omega^-, \omega^+)\big)(\tau) = \big(R_{s} \mathcal{D}_{H} P_{s}(\omega^-, \omega^+)\big)(\tau),
\end{align*}
which holds due to \eqref{RR_TT1}. %\todo{more details, we have 2 statements in that lemma,\textcolor{red}{done!}}
Finally, from \eqref{Sjdks33} and \eqref{R_RT} we infer that
\begin{align*}
	\Phi^{t}_{(\omega^-,\omega^+)}(x)&=\int_{0}^{t}F\left(\tilde{\Phi}_t\left(s,x,(R_{t}	\mathcal{D}_{H}P_{t}(\omega^-,\omega^+))\right)\right)\mathrm{d}s+x+\sigma (R_{t}	\mathcal{D}_{H}P_{t}(\omega^-,\omega^+))(t)\\&=\int_{0}^{t}F\left(\Phi^{s}_{(\omega^-,\omega^+)}(x)\right)\mathrm{d}s+x-\sigma (\mathcal{D}_{H}P_{t}(\omega^-,\omega^+))(-t).
\end{align*}
\qed
\begin{remark}\label{sodosa}
Since $\varphi$ is a $C^1$-SDS, we obviously get that \(\Phi\) is a \(C^1\)-RDS by Remark~\ref{MJNFSA}.
\end{remark}
\paragraph*{Invariant measure for the SDS}
%\textcolor{blue}{Proposition 4.1. from the main file.}
After constructing an RDS from the SDS generated by the SDE~\eqref{MAIN}, we address the existence of a unique invariant measure for the SDS $\varphi$.~To this aim we state the following result.
\begin{proposition}\em{(\cite[Theorem 1.2 and Theorem 1.3]{Hai05})}\label{YTDS}
Let Assumption~\ref{Drift} hold. Then the SDS \(\varphi\) admits a unique invariant measure $\mu$ in the sense of Definition~\ref{measure}.  
In addition, for the law of the solution \( (Y_{t,x})_{t \geq 0} \) %\todo{we don't have $y$,\textcolor{red}{now we fixed this issue}} starting from \( x \in \mathbb{R}^d \),
we have the convergence
\begin{align}
	\mathcal{L}(Y_{t,x}) \xrightarrow{t \to \infty} (\Pi_{\mathbb{R}^d})_{\star} \mu,
\end{align}
in total variation.
Moreover, $(\Pi_{\R^d})_\star \mu$ has  finite moments of order $m$ for every \( m \geq 1 \), i.e.
\begin{align*}
	\int_{\mathbb{R}^d} |x|^{m} \, (\Pi_{\mathbb{R}^d})_{\star}\mu(\mathrm{d}x) < \infty.
\end{align*}
\begin{comment}
	\begin{align*}
		\lim_{t\rightarrow \infty}\Vert\mathcal{L}( Y_{t,x})-(\Pi_{\R^d})_{\star}\mu\Vert_{\mathrm{TV}}=0
	\end{align*}   
	%Here \( \mathcal{L}(Y_{t,x}) \) denotes the law of the solution and $\Vert .\Vert_{\text{TV}}$ denotes usual total variation norm. 
\end{comment}
\end{proposition}
\textbf{Two approaches for the generation of RDS for SDEs with additive fractional noise.}\label{FCRDS}
Let us now revisit the definition of a classical random dynamical system arising from the SDE \eqref{MAIN} and compare it to the RDS constructed in Corollary~\ref{RRRDDDSS}. As already stated, in the RDS framework one incorporates the future of the noise, while the SDS considers the past of the noise. In the following, we make this more precise. 
%Notably, the SDS \( \varphi \) defined in Lemma~\ref{SCXVZBS} does not directly solve the SDE; rather, it describes a process that shares the same law as the actual solution.  
%In contrast, the RDS framework constructs the cocycle using the solution itself, which explicitly depends on the driving noise.  
%The objective here is to provide a conceptual comparison between the classical construction and the RDS described in Corollary~\ref{RRRDDDSS}. %\textcolor{blue}{While we outline the key ideas, we do not pursue full mathematical rigor;} 
%We outline the key ideas here; for precise formulations and proofs, we refer the reader to \cite{GLS10,BRS17}. 
The classical construction of an RDS associated with the SDE \eqref{MAIN} proceeds as follows.
\begin{enumerate}
\item One considers \( \tilde{\Omega} := C_0(\mathbb{R}, \mathbb{R}^d) \) endowed with the compact open topology and $\cF=\cB(C_0(\R,\R^d))$.~It is well-known that one can define a Gaussian measure $\P$ such that the associated canonical process is a two-sided fractional Brownian motion with Hurst parameter \( H\in(0,1) \). For every \( \tilde{\omega} \in \tilde{\Omega} \) and \( t \in \mathbb{R} \), we define the shift $\theta_t :\tilde{\Omega}\to\tilde{\Omega}$
\begin{align*}
	\theta_{t} \tilde{\omega}(s) := \tilde{\omega}(t+s) - \tilde{\omega}(t).
\end{align*}
Moreover $(\tilde{\Omega},\cF,\P, (\theta_t)_{t\in\R}))$ is an ergodic metric dynamical system.
\item Let \( T > 0 \) be fixed and let \(\tilde{\omega}_{[0,T]}\) denote the restriction of \(\tilde{\omega}\) to the interval \([0,T]\). For every \( t \in [0,T] \) and \( x \in \mathbb{R}^d \), define
\[
\Psi^{t}_{\tilde{\omega}}(x) := \tilde{\Phi}_{T}(t, x, \tilde{\omega}_{[0,T]}).
\]
In this setting, due to the uniqueness of the solution, one can verify the cocycle property, i.e. for all \( t, s \geq 0 \) 
\[
\Psi^{t+s}_{\tilde{\omega}}(x) = \Psi^{s}_{\theta_{t} \tilde{\omega}} \circ \Psi^{t}_{\tilde{\omega}}(x).
\]
\end{enumerate}
For the existence of RDS for S(P)DEs with additive fractional noise as sketched above, we refer to~\cite{MS,FGS16b}. 
%In general terms, the lemma~\ref{FARASSa} asserts that for almost every \( \tilde{\omega} \in C_0(\mathbb{R}, \mathbb{R}^d) \), a corresponding random pair \( (\omega^-, \omega^+) \in \mathcal{B}_H \times \mathcal{B}_H \) can be associated by the transformations \( \mathcal{D}_H \) and the family \( \lbrace P_t\rbrace_{t \geq 0} \). For simplicity-and with the caveat that this is not mathematically %\todo{i would completely drop this, we cannot state in paper something which is mathematically not precise, i would put in a comparison on RDS and SDS in section 3 and not here emphasizing that we have two different noise spaces and the advantages below. $\tilde\omega$ is defined for $\R$ and $\cB_H$ is defined for $\R_-$ so both components of the noise strictly look into the past in SDS, i would point this out more and not mention a correspondence.\textcolor{red}{I appreciate that you are very precise. Let's discuss it in person-I see your point, and we can summarize them. As already mentioned, the reason this is included here is twofold: first, to compare two different alternatives, and second, to compare SDS and RDS in a concrete example.} } precise we denote this correspondence as \( \tilde{\omega} = (\omega^-, \omega^+) \in \mathcal{B}_H \times \mathcal{B}_H \).
In the SDS setting, as seen above we consider the space $\Omega:=\cB_H\times \cB_H$ and have that 
\begin{align*}%\label{GBAS}
\Psi^{t}_{(\omega^-, \omega^+)}(x) = \varphi^{t}_{P_t(\omega^-, \omega^+)}(x),
\end{align*}
where $\varphi$ is the SDS associated to the SDE~\eqref{MAIN}.
%Equivalently, one may write:
%\begin{align}\label{YJASdd}
%	\Psi^{t}_{\theta_{-t} \tilde{\omega}}(x) = \varphi^{t}_{\tilde{\omega}}(x).
%\end{align}

%Thanks to Theorem~\ref{RDS_SDS}, we can conclude that  the two RDS constructions - i.e., the classical one and the one given in Corollary~\ref{RRRDDDSS} - are based on different noise spaces, but they coincide under the correspondence discussed above.
%he main distinction lies in the structure of the noise: in our construction, the noise is modeled via the product space \( \mathcal{B}_H \times \mathcal{B}_H \), whereas in the classical approach, it arises directly from \( \tilde{\Omega} \). 
We mention some advantages of the formulation of RDS given by Corollary~\ref{RRRDDDSS} which are exploited in this work. 
\begin{enumerate}
\item We are working with the space $\cB_H$ which can be canonically equipped with the Wiener measure $\mathbf{P}$ and therefore incorporates Markovian dynamics. 
%\todo{i don't think it's really stronger, it just depends on the past} concept of pullback attractor for SDE~\eqref{MAIN} in Theorem~\ref{ATTRRP}.
%, this RDS framework is more robust. %Thanks to the specific norms defined in \(\mathcal{B}_H\), one can obtain precise estimates for every element in the noise space during the calculations. The precise estimates can be obtained also with the other omega
%	\item This RDS reflects the adaptiveness of the solution. This means that it filters out unnecessary information from the noise that depends on the future.
%  We show that both approaches are consistent, illustrating this~for attractors of~\eqref{MAIN} in Appendix~\ref{ATTRATT}.
\item In order to compute Lyapunov exponents, we need ergodicity and information about the density of the invariant measure of~\eqref{MAIN}. For this reason is convenient to work with RDS. 
%\item We prove the existence of an appropriate concept of pullback attractors for~\eqref{MAIN}.
\end{enumerate}

\subsection{A local stable manifold theorem}\label{Linearization}
%Let $\Phi$ be a $C^1$-RDS, and suppose that the map $x \mapsto \Phi^t_{\omega}(x)$ is nonlinear. A common approach to studying the dynamical behavior of the trajectory $t \mapsto \Phi^t_{\omega}(x)$ is to linearize the system with respect to the state variable. This leads to the consideration of the map $t \mapsto \txtD_x \Phi^t_{\omega}$, which is linear in the state variable and therefore more amenable to analytical treatment.
%The next step is to investigate how the properties of the linearized system can enhance our understanding of the original RDS $\Phi$. In this section, we aim to explore these questions within a rigorous mathematical framework. To that end, 
We recall and further fix some notations that will be used throughout this subsection. 
%\todo{reformulate, it is standard to linearize the equation to investigate the dynamics,\textcolor{red}{I summarize it a bit}}
%\todo{we linearize the SDE, the RDS is always linear,\textcolor{red}{Here by RDS ,we mean the solution of equation which is nonlinear} usually the RDS is the solution OPERATOR associated to the sde,\textcolor{red}{I reformulated this part; it's just a casual preface.}}
%\todo{notations not assumptions,\textcolor{red}{done}}
\begin{notation}\label{ASSMM}
%\todo{this and the items below disappear in the draft where everything is restructured,\textcolor{red}{Thank you. I am happy to hear your ideas on how to best organize and clarify them. I did my best to summarize them.}}
\begin{itemize}
%\item Assumption~\ref{Drift} holds.
\item We write  \( \mathcal{B} \) instead of   \( \mathcal{B}_H \). Therefore \( \Omega := \mathcal{B} \times \mathcal{B} \) and \( \mathbb{P}= \mathbf{P} \times \mathbf{P} \). Accordingly, each element \( \omega \in \Omega \) is represented as before as a pair \( \omega = (\omega^-, \omega^+) \).
\item Furthermore \( \Phi \) denotes the \( C^1 \)-RDS defined in Corollary~\ref{RRRDDDSS}. Moreover, the measure $\mu$ refers to the invariant measure established by Proposition~\ref{YTDS}, which is unique and therefore ergodic. %\todo{uniqueness implies ergodicity, here and throughout the manuscript this must be fixed}
\item  The family \( \{\Theta_t\}_{t \geq 0} \) denotes the semigroup generated in Lemma~\ref{THe} associated with the RDS \( \Phi \). 
%\item By \( \tilde{\mathcal{B}} \), we refer to the measurable subset of \( \mathcal{B} \) introduced in the statement of Theorem~\ref{disintegration_2}.% and Corollary~\ref{random_measure}. 
%\todo{\textcolor{red}{I don't understand what you mean} why do we refer to 2 statements for just 1 object $\tilde{B}$ ?,\textcolor{red}{done}}
% when these results are applied to the RDS \( \Phi \).
%\todo{recall also $\Theta$,\textcolor{red}{done}}
\begin{comment}
	content...
	
	Our probailty sapce that For $(\omega^-,\omega^+)\in \mathcal{B}\times\mathcal{B}$ , Ww also use $\Phi^{t}_{(\omega^-,\omega^+)}(x)$ to denote the RDS that arises from this equation. In particular for $\omega=(\omega^-,\omega^+)$ and from \eqref{TGBAS} and \eqref{GBAS}, we have
	\begin{align*}
		\Phi^{t}_{\omega}(x)=\varphi^{t}_{P_{t}(\omega^-,\omega^+)}(x)=\tilde{\Phi}_T\left(t, (R_{t}P_{t}(\omega^-,\omega^+)\right),x),
	\end{align*}
	where $T\geq t$.
\end{comment}
\end{itemize}
\end{notation}
%\todo{the SDS and RDS for the SDE will be stated in a theorem. and there the $\Omega$ will be fixed, \textcolor{red}{done}}
%\todo{state the equation before the proposition, since everything applies to this equation,\textcolor{red}{done}}
%\todo{what is written below should be a lemma with a statement and proof.,\textcolor{red}{It now incorporate in the proof of next proposition}}
We now apply the multiplicative ergodic theorem which is the statement of the following proposition.
\begin{proposition} \label{MET}
There exists a sequence of deterministic values called Lyapunov exponents $\lambda_k < \ldots < \lambda_1$, where $\lambda_i \in [-\infty, \infty)$  and a $(\Theta_t)_{t\geq 0}$-invariant set of full measure $\tilde{\Omega} \subseteq   \mathbb{R}^d\times \mathcal{B} \times \mathcal{B}$ with respect to the probability measure $  \mu\times \mathbf{P}$. 
Defining %\todo{it is not consistent with the initial condition of the SDE and the referee might point this out so we either do it now or at the revision, i suggest to simply denote the initial data of the SDE $\bar{x}$, this results in the least changes and we avoid rejection or comments from referees who won't accept this  \textcolor{red}{Thank you for the suggestion. If there is no substantial improvement, I would kindly prefer to keep the notation as it is. While this choice may not be entirely optimal, it ensures consistency with the rest of the text, and altering $x$ to another notation might unintentionally introduce mistakes, given how frequently it is used throughout.
	%} }
	\begin{align*}
G_\lambda(\omega,x):=\left\lbrace z\in\R^d: \ \ \limsup_{t\rightarrow\infty}\frac{1}{t}\log\left|\txtD_{x}\Phi^{t}_{\omega}(z)\right|\leq \lambda\right\rbrace
\end{align*}  
for every $(x,\omega)=(x,\omega^-,\omega^+) \in \tilde{\Omega}$ and $\lambda\in[-\infty,\infty)$,  the following assertions hold.
\begin{enumerate}
\item	$G_{\lambda_1}((x,\omega))=\R^d$.
\item  For every $1\leq i<k$ 
\begin{align*}
	z\in G_{\lambda_i}(x,\omega)\setminus G_{\lambda_{i+1}}(x,\omega)\xLeftrightarrow{\hspace{.5cm}} \lim_{t\rightarrow\infty}\frac{1}{t}\log\left|\txtD_{x}\Phi^{t}_{\omega}(z)\right|=\lambda_i.
\end{align*}
\item  For every $z\in G_{\lambda_k}(\omega,x)\setminus\lbrace 0\rbrace$
\begin{align*}
	\lim_{t\rightarrow\infty}\frac{1}{t}\log\left|\txtD_{x}\Phi^{t}_{\omega}(z)\right|=\lambda_k.
\end{align*}
\item For every $1 \leq i < k$, we can find a measurable subspace $H_i(x, \omega)$ of $\mathbb{R}^d$ such that
\begin{align*}
	G_{\lambda_i}(x,\omega)= G_{\lambda_{i+1}}(x,\omega)\oplus H_{i}(x,\omega).
\end{align*}	
\end{enumerate}
\end{proposition}
\proof
%\todo{write the idea, what one has to prove in order to apply met. above $U_{\lambda_k}$ is a typo,\textcolor{red}{done}}
From Remark~\ref{sodosa} we know that \( \Phi \) is \( C^1 \). Thus, by Corollary~\ref{DCHYJAS85} we obtain for any arbitrary \( \epsilon > 0 \) and $x_0\in\R^d$ that
\begin{align*}
\Phi^{t}_{\omega}(x+\epsilon x_0)-\Phi^{t}_{\omega}(x)=\int_{0}^{t}\left(F\left(\Phi^{s}_{\omega}(x+\epsilon x_0)\right)-F\left(\Phi^{s}_{\omega}(x)\right)\right)\mathrm{d}s+\epsilon x_0.
\end{align*}
Therefore, by the dominated convergence theorem as \( \epsilon \to 0 \), we conclude that %\todo{we don't need labels (4.13), (4.14) below,\textcolor{red}{the second will be used later }}
\begin{align*}%\label{JKASS}
\txtD_{x}\Phi^{t}_{\omega}(x_0)=\int_{0}^{t}\txtD_{\Phi^{s}_{\omega}(x)}F(\txtD_{x}\Phi^{s}_{\omega}(x_0))\ \mathrm{d}s + x_0.
\end{align*}
In particular, this yields
\begin{align}\label{UASs}
\frac{\mathrm{d}}{\mathrm{d} t}\txtD_{x}\Phi^{t}_{\omega}(x_0)=\txtD_{\Phi^{t}_{\omega}(x)}F(\txtD_{x}\Phi^{t}_{\omega}(x_0)), \ \ \txtD_{x}\Phi^{0}_{\omega}(x_0)=x_0.
\end{align}
Consequently, from Assumption \ref{Drift} 2)
\begin{align*}%\label{ASAS4df}
\begin{split}
	&\frac{\mathrm{d}}{\mathrm{d} t}\left|\txtD_{x}\Phi^{t}_{\omega}(x_0)\right|^2=2\left\langle \txtD_{\Phi^{t}_{\omega}(x)}F(\txtD_{x}\Phi^{t}_{\omega}(x_0)),\txtD_{x}\Phi^{t}_{\omega}(x_0)\right\rangle\\&=2\lim_{\epsilon\rightarrow 0}\frac{1}{\epsilon^2}\left\langle F\left(\Phi^{t}_{\omega}(x+\epsilon x_0)\right)-F\left(\Phi^{t}_{\omega}(x)\right),\Phi^{t}_{\omega}(x+\epsilon x_0)-\Phi^{t}_{\omega}(x)\right\rangle\leq 2C_3^F \left|\txtD_{x}\Phi^{t}_{\omega}(x_0)\right|^2.
\end{split}
\end{align*} 
Using Gr\"onwall's lemma, we obtain
\begin{align}\label{ILASOsd69a}
\left\| \txtD_{x}\Phi^{t}_{\omega} \right\|_{L(\mathbb{R}^d, \mathbb{R}^d)} \leq \exp(C_3^F t),
\end{align}
which is a deterministic bound and hence integrable with respect to \(\mu \times \mathbf{P}\).
By Corollary~\ref{Thetta}, \( \left( \txtD_{x} \Phi^{t}_{\omega} \right)_{t\geq 0} \) is a RDS over the metric dynamical system
\begin{align}\label{ASAd}
\left( \mathbb{R}^d \times \mathcal{B} \times \mathcal{B}, \; \sigma(\mathbb{R}^d) \otimes \sigma(\mathcal{B}) \otimes \sigma(\mathcal{B}), \; \mu \times \mathbf{P}, \; \Theta \right).
\end{align}
Note that for every \( t > 0 \), the map \( \vartheta_t : \mathcal{B} \to \mathcal{B} \) is ergodic with respect to \( \mathbf{P} \). Moreover, the measure \( \mu \) is also ergodic. This proves the ergodicity of the metric dynamical system above by Corollary~\ref{THETA}. % the above metric dynamical system is ergodic.
Consequently, the multiplicative ergodic theorem~\cite[Theorem 3.4.1]{Arn98} proves the statement. 
\qed
%\begin{remark}
%Since \( \lbrace\Theta_t\rbrace_{t \geq 0} \) is not invertible, the complementary subspaces \( (H_i(x, \omega))_{1 \leq i \leq k} \) are not necessarily defined in a canonical way; see \cite[Theorem 0.2]{GVR23A}. However, for our purposes, it is sufficient that these subspaces are measurable \footnote{This means measurability of the subspace as a map into the Grassmannian manifold.}. 
%This is a well-established result, typically proven by  measurable selection theorems.%\todo{what is a canonical way,\textcolor{red}{explained by adding a reference}} 
%\end{remark}
This theorem yields several important results. The one particularly relevant to our setting is the local stability property, which holds in particular when the top Lyapunov exponent is negative. We begin by stating a preparatory lemma.
\begin{lemma}{\em (\cite[Lemma~7.2.1]{Arn98})}\label{YAHs}
For $\omega\in\Omega$ and $x,z \in\R^d$, set
\begin{align*}%\label{85qwerz}
\overline{\Phi}^{t}_{(x,\omega)}(z):=\Phi^{t}_{\omega}(x+z)-\Phi^{t}_{\omega}(x).
\end{align*}
Then \( \overline{\Phi} \) is a \( C^1 \)-RDS over the metric dynamical system~\eqref{ASAd}. In addition, zero is a fixed point and the Dirac measure $\delta_0$ is an invariant measure for $\overline{\Phi}$.
\begin{comment}
Consider the probability space $\mathcal{B} \times \mathcal{B} \times \mathcal{X}$ with probability measure $\mathbf{P} \times \mu$ and the family of shift maps $(\Theta_t)_{t\geq 0}$, which is defined in Lemma \ref{THe}. For $\Omega=(\omega^-,\omega^+)$ \todo{this is not $\Omega$} and $x, y\in\R^d$, set
\begin{align}\label{85qwerz}
	\bar{\Phi}^{t}_{(x,\omega)}(y):=\Phi^{t}_{\omega}(x+y)-\Phi^{t}_{\omega}(x).
\end{align}    
he cocycle property with respect to $\Theta$, i.e.,  
\begin{align}\label{85unzvase}
	\bar{\Phi}^{t+s}_{(x,\omega)}(y) = \bar{\Phi}^{s}_{\Theta_{t}(x,\omega)} \circ\overline{\Phi}^{t}_{(x,\omega)}(y),
\end{align}
for all $t, s \geq 0$, $(x,\omega) \in  \mathbb{R}^d\times \mathcal{B} \times \mathcal{B}$, and $y \in \mathbb{R}^d$.
\end{comment}
\end{lemma}

We need an auxiliary lemma that provides an integrable %\todo{w.r.t.,\textcolor{red}{I think it is clear form the context that a priori bounds depend on $t,\omega$ and $x$}} 
a priori bound for the solution of \eqref{MAIN} under Assumption \ref{Drift}.
\begin{lemma}\label{PRIORI}
Let \( t_0 > 0 \). Then there exist a random variable \( \Sigma = \Sigma_{t_0} \) and a positive constant \( L \) such that for every \( x \in \mathbb{R}^d \)
\begin{align*}
\sup_{0 \leq t\leq t_0}\vert\Phi^{t}_{\omega}(x)\vert\leq L\left(1+\vert x\vert^N +\Sigma(\omega)\right)
\end{align*}
and $\Sigma\in\bigcap_{1\leq q<\infty}L^{q}(\Omega).$
\end{lemma}
\proof
From Corollary~\ref{DCHYJAS85}, it follows that for every \( t \geq 0 \) and \( x \in \mathbb{R}^d \),
\begin{align*}
\Phi^{t}_{\omega}(x)=\int_{0}^{t}F\left(\Phi^{s}_{\omega}(x)\right)\mathrm{d}s+x-\sigma (	\mathcal{D}_{H}P_{t}(\omega^-,\omega^+))(-t).
\end{align*}
For $t\geq 0$, we set 
\begin{align}\label{TAYeS6}
\begin{split}
	&B^{H}_{t}(P_{t}(\omega^-,\omega^+))=- (\mathcal{D}_{H}P_{t}(\omega^-,\omega^+))(-t),\\
	&U_{t}(\omega,x):=x+\sigma B^{H}_{t}(P_{t}(\omega^-,\omega^+)).
\end{split}
\end{align}
Then \( \Phi^{t}_{\omega}(x) = V_{t}(\omega,x) + U_{t}(\omega,x) \) where 
\begin{align*}%\label{UJASsdd}
\frac{\mathrm{d}}{\mathrm{d}t} V_{t}(\omega,x) = F\left(V_{t}(\omega,x) + U_{t}(\omega,x)\right), \ \ \	V_{0}(\omega,x) = 0.
\end{align*}
This further implies using Assumption~\ref{Drift} 
\begin{align*}
\frac{\mathrm{d}}{\mathrm{d}t} \left\vert V_{t}(\omega,x)\right\vert^2&=2\left\langle F\left(V_{t}(\omega,x) + U_{t}(\omega,x)\right)-F\left( U_{t}(\omega,x)\right),V_{t}(\omega,x)\right\rangle+2\left\langle F\left( U_{t}(\omega,x)\right),V_{t}(\omega,x) \right\rangle\\&\leq 2C_{1}^{F}-2C_{2}^{F}\left|V_{t}(\omega,x)\right|^2+\frac{1}{C_2^F} \left| F\left(U_{t}(\omega,x)\right)\right|^2+C_2^F \left|V_{t}(\omega,x)\right|^2\\&\leq -C_{2}^{F}\left|V_{t}(\omega,x)\right|^2+2C_{1}^{F}+\frac{(C_F)^2}{C_2^F}\left(1 + \left\vert U_{t}(\omega,x)\right\vert\right)^{2N}\\&
\leq -C_{2}^{F}\left|V_{t}(\omega,x)\right|^2+ \tilde{C}_{F}\left(1 + \left\vert U_{t}(\omega,x)\right\vert\right)^{2N},
\end{align*}
for another arbitrary constant \( \tilde{C}_{F} \). Thus, by Gr\"onwall's inequality we obtain
\begin{align*}%\label{usdmwwq}
\left|V_{t}(\omega,x)\right|^2\leq \exp(-C_{2}^{F}t)\vert x\vert^2+\tilde{C}_{F}\int_{0}^{t}\exp\left(-C_{2}^{F}(t-\tau)\right)\left(1 + | U_{\tau}(\omega,x)|\right)^{2N}\mathrm{d}\tau.
\end{align*}
This, combined with Lemma~\ref{FerniqueAA}, proves the claim. %\todo{we should state Fernique in the preliminaries, here we give a reference, in section 6 we use it a lot without ref.,\textcolor{red}{done, in other placed are a are updated accordingly}}
\qed
The following result provides a sufficient condition to prove the existence of invariant manifolds. For stability (i.e.~a negative top Lyapunov exponent), we are particularly interested in the stable manifold theorem. To this aim we state the following auxiliary result which enables us to prove the stable manifold theorem, i.e. Proposition \ref{STAB}.
\begin{proposition}\label{YASd}
Assume that for some \( r \in (0, 1] \), and for every \( \xi_1, \xi_2 \in \mathbb{R}^d \) there exist constants $\overline{C}_F$ and $p_1\geq 1 $ such that
\begin{align}\label{DEEE}
\| \txtD_{\xi_2}F - \txtD_{\xi_1}F \|_{L(\mathbb{R}^d, \mathbb{R}^d)} \leq \overline{C}_F(1+\vert \xi_1\vert+\vert \xi_2\vert)^{p_1} |\xi_2 - \xi_1|^r.
\end{align}
%where $\tilde{C}_F>0$ and $\tilde{N}\geq 1$. 
For every \( t_0 > 0 \), there exist \( p_2 \in \mathbb{N} \) and a positive random variable \( f \) such that for all \( \omega \in \Omega \) and \( x, y, z \in \mathbb{R}^d \)
\begin{align}\label{ASAx}
\begin{split}
	&\sup_{0\leq t\leq t_0}\left\vert{\Phi}^{t}_{\omega}(x+y)-{\Phi}^{t}_{\omega}(x+z)-\txtD_{x}{\Phi}^{t}_{\omega}(y-z)\right\vert\\&\quad\leq \vert y-z\vert(\vert y\vert^r+\vert z\vert^r)\exp\left(1+(\vert y\vert+\vert z\vert)^{p_2}\right)\exp\left(f(\omega)+\vert x\vert^{p_2}\right)
\end{split}
\end{align}
where \( f \in \bigcap_{1\leq q<\infty} L^{q}(\Omega) \).
\end{proposition}  
\proof
First, note that for \( a, b, x \in \mathbb{R}^d \) and using \eqref{UASs}, we have
\begin{align*}
\frac{\mathrm{d}}{\mathrm{d} t} \left\vert \txtD_{b}\Phi^{t}_{\omega}(a)-\txtD_{x}\Phi^{t}_{\omega}(a)\right\vert^2&=2\left\langle \txtD_{\Phi^{t}_{\omega}(b)}F\left(\txtD_{b}\Phi^{t}_{\omega}(a)\right)-\txtD_{\Phi^{t}_{\omega}(x)}F\left(\txtD_{x}\Phi^{t}_{\omega}(a)\right),\txtD_{b}\Phi^{t}_{\omega}(a)-\txtD_{x}\Phi^{t}_{\omega}(a)\right\rangle\\
&=2\left\langle  \txtD_{\Phi^{t}_{\omega}(b)}F\left(\txtD_{b}\Phi^{t}_{\omega}(a)-\txtD_{x}\Phi^{t}_{\omega}(a)\right),\txtD_{b}\Phi^{t}_{\omega}(a)-\txtD_{x}\Phi^{t}_{\omega}(a)\right\rangle\\&+2\left\langle\left(\txtD_{\Phi^{t}_{\omega}(b)}F-\txtD_{\Phi^{t}_{\omega}(x)}F\right)\left(\txtD_{x}\Phi^{t}_{\omega}(a)\right),\txtD_{b}\Phi^{t}_{\omega}(a)-\txtD_{x}\Phi^{t}_{\omega}(a)\right\rangle.
\end{align*}
Thus from \eqref{DEEE}, Assumption~\ref{Drift} and using H\"older's and Minkowski's inequalities, we obtain that
\begin{align}\label{NMIOas}
\begin{split}
	&\frac{\mathrm{d}}{\mathrm{d} t} \left\vert \txtD_{b}\Phi^{t}_{\omega}(a)-\txtD_{x}\Phi^{t}_{\omega}(a)\right\vert^2\leq  2C_F\left(1 + \left|\Phi^{t}_{\omega}(b)\right|\right)^N \left\vert \txtD_{b}\Phi^{t}_{\omega}(a)-\txtD_{x}\Phi^{t}_{\omega}(a)\right\vert^2\\&+2\overline{C}_F\left(1+\left\vert \Phi^{t}_{\omega}(b)\right\vert+\left\vert \Phi^{t}_{\omega}(x)\right\vert\right)^{p_1}\left|\Phi^{t}_{\omega}(b)-\Phi^{t}_{\omega}(x)\right|^r \left|\txtD_{x}\Phi^{t}_{\omega}(a)\right|\left\vert \txtD_{b}\Phi^{t}_{\omega}(a)-\txtD_{x}\Phi^{t}_{\omega}(a)\right\vert\\&\leq  \underbrace{\left(2C_F\left(1 + \left|\Phi^{t}_{\omega}(b)\right|\right)^N+1\right)}_{\kappa(t,b,x,a)}\left\vert \txtD_{b}\Phi^{t}_{\omega}(a)-\txtD_{x}\Phi^{t}_{\omega}(a)\right\vert^2\\&+\underbrace{(\overline{C}_{F})^2\left(1+\left\vert \Phi^{t}_{\omega}(b)\right\vert+\left\vert \Phi^{t}_{\omega}(x)\right\vert\right)^{2p_1}\left|\Phi^{t}_{\omega}(b)-\Phi^{t}_{\omega}(x)\right|^{2r}\left|\txtD_{x}\Phi^{t}_{\omega}(a)\right|^2}_{\tilde\kappa(t,b,x,a)}\\=&\kappa(t,b,x,a)\left\vert \txtD_{b}\Phi^{t}_{\omega}(a)-\txtD_{x}\Phi^{t}_{\omega}(a)\right\vert^2+\tilde\kappa(t,b,x,a).
\end{split}
\end{align}
Therefore, by Gr\"onwall's inequality
\begin{align}\label{ASAc}
\left\vert \txtD_{b}\Phi^{t}_{\omega}(a)-\txtD_{x}\Phi^{t}_{\omega}(a)\right\vert^2\leq \left[\int_{0}^{t}\tilde\kappa(s,b,x,a)\exp\left(\int_{s}^{t}\kappa(\tau,b,x,a)\mathrm{d}\tau\right)\mathrm{d}s\right].
\end{align}
Note that
\begin{align}\label{ASAAsc}
\begin{split}
	\big\vert{\Phi}^{t}_{\omega}(x+y)&-{\Phi}^{t}_{\omega}(x+z)-\txtD_{x}{\Phi}^{t}_{\omega}(y-z)\big\vert^2\\&
	\leq\int_{0}^{1}\left\vert\left( \txtD_{x+z+\theta(y-z)}\Phi^{t}_{\omega}\left(y-z\right)-\txtD_{x}\Phi^{t}_{\omega}\left(y-z\right)\right)\right\vert^2\mathrm{d}\theta.
	%  \left\vert\int_{0}^{1}\left( D_{x+z+\theta(y-z)}\Phi^{t}_{\omega}\left(y-z\right)-D_{x+z}\Phi^{t}_{\omega}\left(y-z\right)\right)\mathrm{d}\theta\right\vert^2\\&
\end{split}
\end{align}
%Now, our claim follows from \eqref{ILASOsd69a}, Lemma~\ref{PRIORI}, \eqref{ASAc}, and \eqref{ASAAsc}.

To establish the claim, it suffices to use \eqref{ASAc} to estimate the right-hand side of \eqref{ASAAsc} in terms of the functions $\tilde{k}$ and $\kappa$. Based on \eqref{NMIOas}, Lemma~\ref{PRIORI} together with \eqref{ILASOsd69a} prove the statement.

%\todo{more details, instead of 4 ref,\textcolor{red}{I appreciate your comment. However, this is really an annoying algebraic computation.  
	%I will explain more, but doing algebra here is quite frustrating.} i don't want algebra here, i agree is totally unnecessary but it's just not a good style to state that the claim follows from 4-5-6 consecutive labels,\textcolor{red}{true, I hope now you are satisfied with the way that is formulated}}
\qed
The next result establishes the existence of a local stable manifold when the top Lyapunov exponent $\lambda_1$ is negative.~Such a manifold contains the neighborhood of the origin and within this neighborhood, trajectories decay exponentially toward the origin%.~For a similar argument in the Markovian setting we refer to~\cite[Lemma 3.1]{FGS16a}.   
\begin{proposition}\label{STAB} 
%\todo{$\cM$ not introduced,\textcolor{red}{I changed the statement to be more clear}}
Assume that the conditions of Proposition~\ref{YASd} hold and that the top 
Lyapunov exponent satisfies $\lambda_1 < 0$. Then there exists a set of full 
measure $\mathcal{M} \subseteq \mathbb{R}^d \times \Omega$ such that for every $0 < \nu < -\lambda_1$, there exists a positive random variable 
\[
R^\nu : \mathcal{M} \to (0,\infty)
\]
such that, for every $(x,\omega) \in \mathcal{M}$ and every $y \in \mathbb{R}^d$ 
with $\lvert y \rvert \leq R^\nu(x,\omega)$, we have
\begin{align*}
\sup_{t\geq 0}\exp(t\nu)\left|\Phi^{t}_{\omega}(x+y)-\Phi^{t}_{\omega}(x)\right|< \infty.
\end{align*}
\end{proposition}
\proof
We first recall the cocycle \(\overline{\Phi} \) defined by
\[
\overline{\Phi}^{t}_{(x,\omega)}(y) = \Phi^{t}_{\omega}(x + y) - \Phi^{t}_{\omega}(x)
\]
defined in Lemma~\ref{YAHs}.~Fixing an arbitrary time step \( t_0 > 0 \), we consider the discretized cocycle \( (\overline{\Phi}^{n t_0}_{(x,\omega)})_{n \in \mathbb{N}} \) and prove the existence of a local stable manifold for this system.  %\todo{reformulate the text above, remove the steps, just a full proof without two steps. First we state what is a local stable manifold, refer to Gess for a similar statement in the Markovian setting, say we have to verify the condition in your paper that gives the claim. more structure, essential information, less confusing text }
This can be achieved applying~\cite[Theorem~2.10]{GVR23A}.~To this aim we need to verify that 
\begin{align}\label{5sdww}
\begin{split}      
	\sup_{0\leq t\leq t_0} \big\vert\overline{\Phi}^{t}_{(x,\omega)}(y) &-\overline{\Phi}^{t}_{(x,\omega)}(z) - \txtD_0\bar{\Phi}^{t}_{(x,\omega)}(y-z) \big\vert  
	\\&\leq G(\omega,x) \vert y - z \vert  
	\leq \vert y - z \vert (\vert y \vert^r + \vert z \vert^r) g(\vert y \vert + \vert z \vert),
\end{split}
\end{align}
where \( \log^{+} G(x,\omega) \) is integrable with respect to \( \mu \times \mathbf{P} \) for an arbitrary function $G$, \( 0 < r \leq 1 \), \( g \) is a positive, increasing \( C^1 \)-function, and \( \txtD_0 \) denotes the derivative at the origin.~We now verify \eqref{5sdww}. For every \( (x, \omega) \in \mathbb{R}^d \times \Omega \) and \( y, z \in \mathbb{R}^d \) and from Proposition \ref{YTDS}, we have
\begin{align}\label{85as}
\begin{split}
	&\sup_{0\leq t\leq t_0}\left\vert\overline{\Phi}^{t}_{(x,\omega)}(y)-\overline{\Phi}^{t}_{(x,\omega)}(z)-\txtD_0\overline{\Phi}^{t}_{(x,\omega)}(y-z)\right\vert\\&=\sup_{0\leq t\leq t_0}\left\vert{\Phi}^{t}_{\omega}(x+y)-{\Phi}^{t}_{\omega}(x+z)-\txtD_x{\Phi}^{t}_{\omega}(y-z)\right\vert\\&\quad \leq \vert y-z\vert(\vert y\vert^r+\vert z\vert^r)\exp\left(1+(\vert y\vert+\vert z\vert)^{p_2}\right)\exp\left(f(\omega)+\vert x\vert^{p_2}\right).
\end{split}
\end{align}
From Propositions \ref{YTDS}  and \ref{YASd} given the fact that using the fact that $f$ does not depend on the state variable $x$, we have
\begin{align}\label{85as3}
\int_{ \R^d\times\Omega }\left(f(\omega)+\vert x\vert^{p_2}\right)( \mu\times\mathbf{P})\left(\mathrm{d}(x,\omega^-,\omega^+)\right)<\infty.
\end{align}
This verifies~\eqref{5sdww} and~\cite[Theorem 2.10]{GVR23A} gives us the existence of a local invariant stable manifold around the origin for cocycle \( (\overline{\Phi}^{nt_0}_{(x,\omega)})_{n \in \mathbb{N}} \).~This can be extended to the continuous time cocycle \( (\overline{\Phi}^t_{(x,\omega)})_{t \geq 0} \)by~\cite[Remark 2.13]{GVR26}.
%\todo{ \textcolor{red}{Thanks. In the Proposition \ref{YASd}, we state that $f$ is integrable of any order, or I misunderstood you?} yes but we have $\int_{\R^d}$ and separately $\int_\Omega$ and the statement is $\int_{\R^d\times\Omega}$ and for the joint integrability we need independence. it is correct but we have to write this,\textcolor{red}{Thanks. I think what you proposing is additional explanation. Since we have probability measures and $f$ does not depend on the state variable, here we are brief and not explicit. I added a few sentences} exactly}
\qed

\section{Lyapunov exponents}\label{sec:LP}

In this section we prove that by increasing the intensity of the noise, we can ensure that the Lyapunov exponent of~\eqref{MAIN} becomes negative. %In the existing literature, two general \todo{i already wrote this in the intro}frameworks are well established. The first involves order-preserving structures, which are particularly useful for proving synchronization; in this setting, Markovianity is not necessarily crucial. However, in the multidimensional case, such order-preserving structures fail. In these situations, the literature often employs tools like the Fokker-Planck equation to estimate the sign of the Lyapunov exponent. Unfortunately, these tools are not applicable to the type of equations considered in this manuscript. 
Since we cannot use the Fokker-Planck equation to compute the stationary density of the SDE as in~\cite{FGS16a}, we rely on the results in~\cite{LPS23} which provide Gaussian estimates for this density. 
%\todo{which solutions? and why not solution of the concrete equation we are investigating}
We consider the same SDE~\eqref{MAIN} given by
\begin{align}\label{aasass}
\begin{cases}
&\mathrm{d}Y_t^{\sigma} = F(Y_t^\sigma) \, \mathrm{d}t + \sigma \, \mathrm{d} B^{H}_t\\
& Y_0^\sigma=x \in \mathbb{R}^d,
\end{cases}
\end{align}
where we additionally keep track of the dependence of the solution $(Y^\sigma_t)_{t\in \R^+}$ on the noise intensity $\sigma$. 
We first specify some assumptions. 
\begin{assumptions}\label{Drift2}   
%\todo{we already have this, is the SDE (4.1). We specify additional assumptions but not rewrite the same SDE and give it a different label,\textcolor{red}{I appreciate your point. However, there is a dependence on $\sigma$ here, which is why it is written again. In the previous version, we did not consider this parameter} yes but it is the SAME SDE. We can say that we denote the solution of (4.1) by $Y^\sigma_t$ since we want to emphasize its dependence on $\sigma$ but not rewrite the same equation and give it another label since its the same object. \textcolor{red}{I appreciate your point. Following your rule for precision, I have rewritten the equation to avoid disturbing the reader. Referring back to it at the end with a few explanatory sentences seems equivalent to what we have done here, but I trust your judgment on the best approach. We should be very careful with the labeling, as changing it for the rest may cause some problems.
	%}}
	
	We assume that \( F \in C^1(\mathbb{R}^d) \) and that its derivative is globally bounded. Moreover, \(F\) is eventually strictly monotone, i.e.~there exist constants \(R > 0\) and \(C_4^F > 0\) such that, for all \(\xi_1, \xi_2 \in \mathbb{R}^d\) with \(|\xi_1|, |\xi_2| \geq R\) we have %\todo{i would work with assumption 4.1. throughout the manuscript and just add a remark that ev strictly monotone implies 2) and 3). it is more clear than switch or repeat assumptions  }
	\begin{align}\label{monotone st}
\langle F(\xi_2) - F(\xi_1), \xi_2 - \xi_1 \rangle \leq -C_4^F |\xi_2-\xi_1|^2.
\end{align}
As before, \( \sigma : \mathbb{R}^d \rightarrow \mathbb{R}^d \) is an invertible matrix. %\todo{let $\sigma$ be..}
\end{assumptions}
\begin{remark}\label{Notation}
Note that Assumption~\ref{Drift2} implies  Assumption~\ref{Drift}. In particular
\begin{align*}
\langle F(\xi_2) - F(\xi_1), \xi_2 - \xi_1 \rangle \leq \min \left\{ C_1^F - C_2^F |\xi_2 - \xi_1|^2,\, C_3^F |\xi_2 - \xi_1|^2 \right\}
\end{align*}
and
\begin{align*}
|F(\xi)| \leq C_F(1 + |\xi|), \quad |\txtD_\xi F| \leq C_F,
\end{align*}
for all \(\xi, \xi_1, \xi_2 \in \mathbb{R}^d\). Thus, together with Proposition~\ref{YTDS}, we can conclude that the SDS associated with~ SDE \eqref{aasass} admits a unique invariant measure.
\end{remark}
We further specify a restriction on \( \sigma \).
\begin{definition}\label{theth}
Let us consider the singular value decomposition of \( \sigma \) as 
\[
\sigma = U \Sigma V^T,
\]
where \( U, V \in \mathbb{R}^{d \times d} \) are orthogonal matrices, and \( \Sigma = \mathrm{diag}(\beta_1, \beta_2, \dots, \beta_d) \) is a diagonal matrix with nonzero entries, known as the singular values of \( \sigma \), such that
\[
0 < \beta_1 \leq \beta_2 \leq \dots \leq \beta_d.
\]

Then, for every \( \theta \geq 1 \), we define the set
\[
T_\theta := \left\{ \sigma \in \mathbb{R}^{d \times d} \,\middle|\, \sigma \text{ is invertible and } \frac{|\beta_d|}{|\beta_1|} \leq \theta \right\}.
\]
Note that in this case \( \Vert \sigma \Vert = \beta_d \) and \( \Vert \sigma^{-1} \Vert = \frac{1}{\beta_1} \). For \( \kappa > 0 \), we set
\[
T_{\theta, \kappa} := \left\{ \sigma \in T_{\theta} \,\middle|\, \Vert \sigma \Vert \geq \kappa \right\}.
\]
%Throughout this section, we fix the parameters \( \theta \geq 1 \) and \( \kappa >0  \).
Throughout this section, the parameters \( \theta \geq 1 \) and \( \kappa > 0 \) are fixed.
\end{definition} 
%We impose the following condition.
%\todo{is this really an assumption? we can simply fix the parameters in the previous def,\textcolor{red}{Thanks, since still the final structure of the paper(or papers) is not clear, at the moment is written in this way, after that we can incorporate your comment.} this has nothing to do with the structure,\textcolor{red}{i moved to the definition} }
%\begin{assumptions} 

%\end{assumptions}
%\begin{remark}
%\textcolor{red}{I suggest that we don't say it since then many parts are affecting, I can explain in person why. yes we make  some remarks, but first work with a setting where everything is fixed and not mixed}I mean don't fixed $\kappa=1$, since then many constants and notations are affected. i will explain to you, i have to go through everything to make it consistent. Ok, in my opinion this really can change many statements and better just to say $\kappa>0$. i see, but let me find a way to do everything consistent
%\end{remark}
%\subsection{Stationary density}
%We summarize the results by Xue-Mei here.
%The following upper and lower Gaussian bounds for the density of the SDE have been established in~\cite{LPS23}.
%\textcolor{red}{Would you consider it appropriate to combine this part with the next one? I thought it might help streamline the exposition, but I am, of course, happy to follow whatever structure you find most suitable. Note that Remark~\ref{ASaoas} clearly reflects your comment.
%}

\subsection{A rescaling argument}

For our purposes, it is more convenient  to rescale~\eqref{aasass} such that its diffusion coefficient lies on the unit sphere. This is the content of the next lemma.

\begin{lemma}\label{LANAMSWQ}
We consider the following SDE
\begin{align} \label{aasass1}
\begin{cases}
	\txtd Z_t^\sigma = \|\sigma\|^{-1} F(\|\sigma\| Z_t^\sigma) \, \txtd t + \|\sigma\|^{-1} \sigma \, \txtd B_t^H, \\
	Z_0^\sigma=x \in \mathbb{R}^d,
\end{cases}
\end{align}
and  %\todo{$Y$ appears later,\textcolor{red}{could you please tell me what you mean? $Y$ is already defined before of this Lemma}} and \( (Z_{t,x}^\sigma)_{t \geq 0} \) 
denote its solution by \( (Z_{t,x}^\sigma)_{t \geq 0} \). %\todo{this does not match, in (6.4) the initial data is $Z^\sigma_0$,\textcolor{red}{now is match}}
Then, for every \( t \geq 0 \) and \( x \in \mathbb{R}^d \), we have
\begin{align*}
Z_{t,x}^{\sigma} = \|\sigma\|^{-1} Y_{t, \|\sigma\| x}^{\sigma},
\end{align*}
where $Y^\sigma_{t,x}$ denotes the solution of~\eqref{aasass} starting in $x\in\R^d$. 
\end{lemma}
\proof
The result follows directly from the uniqueness of solutions of the SDE~\eqref{aasass1} stated in Lemma~\ref{OALSdfe}. %Substituting \( Z_t^\sigma = \|\sigma\|^{-1} y_t^\sigma \), one verifies that the rescaled process satisfies the dynamics of \eqref{aasass1}. Hence, the identity holds by uniqueness.
\qed
We also have the following straightforward result for the rescaled drift.
\begin{lemma}\label{Straightforward}
Let $\sigma\in T_{\theta,1}$ and assume that Assumption \ref{Drift2}  holds. For every \(\xi \in \mathbb{R}^d\) we define %\todo{i am sorry but $x$ is referred as the initial condition of the SDE which in all manuscripts in the literature is fixed and here we use $x$ again for something else.  no referee will accept this,\textcolor{red}{Thank you very much for your suggestion regarding this part. I have avoided using 
	%$x$ here, and I hope this addresses your concern}}
	\begin{align}\label{IK63asd}
F^{\sigma}(\xi) := \|\sigma\|^{-1} F(\|\sigma\|\xi).
\end{align}
Then, for every \(\xi_1,\xi_2 \in \mathbb{R}^d\), we have
\begin{align}\label{fsigma}
\begin{split}
	\langle F^{\sigma}(\xi_2) - F^{\sigma}(\xi_2), \xi_2-\xi_1 \rangle 
	&\leq \min \left\{ C_1^F - C_2^F |\xi_2-\xi_1|^2,\, C_3^F |\xi_2-\xi_1|^2 \right\}, \\
	C_{F^{\sigma}} &= C_F,
\end{split}
\end{align}
as imposed in Assumption~\ref{Drift}. In addition, the SDS associated with SDE~\eqref{aasass1} admits a unique invariant measure.
%Therefore, by Proposition~\ref{YTDS}, the SDE~\eqref{aasass1} admits a unique invariant measure \( \tilde{\mu}^{\sigma} \).
\end{lemma}
\proof
By the definition of \( F^{\sigma} \) and Remark~\ref{Notation}, the condition~\eqref{fsigma} can be readily verified.
%\todo{this is not an equation, \textcolor{red}{true}}. 
Moreover, the existence of a unique invariant measure follows from Proposition~\ref{YTDS}.
%\eqref{fsigma} follows from basic calculus. The claim regarding the existence of a unique invariant measure follows from  \eqref{fsigma} and Proposition~\ref{YTDS}.
%\todo{reformulate, a sentence does not begin with a formula and the basic calculus should be replaced with the reference to the ass on $F$ or remark 6.2 which implies the statement,\textcolor{red}{based on your comment the sentence is revised}}
\qed
\begin{remark}
The assumption $\sigma\in T_{\theta,1}$ implies that $\|\sigma\|\geq 1$. This condition can be replaced by $\|\sigma\|\geq \kappa$ for an arbitrary $\kappa>0$. Since this obviously does not change the statement of Lemma~\ref{Straightforward} but only the constants involved, we stated for simplicity~\eqref{fsigma} for  $\kappa=1$. 
\end{remark}
%We will use the following notation throughout the remainder of this manuscript:
\begin{notation}
We use \( \mu^{\sigma} \) and \( \tilde{\mu}^{\sigma} \) to denote the invariant measures of the stochastic dynamical systems associated with the SDEs~\eqref{aasass} and~\eqref{aasass1}. The corresponding first marginals are denoted by \( \pi^\sigma \) and \( \tilde{\pi}^\sigma \) meaning that
\begin{align*}
(\Pi_{\mathbb{R}^d})_{\star} \mu^{\sigma} &= \pi^{\sigma}, \\
(\Pi_{\mathbb{R}^d})_{\star} \tilde{\mu}^{\sigma} &= \tilde{\pi}^{\sigma}.
\end{align*}
\end{notation}
%In order to emphasize the dependence on \( \sigma \) for the invariant measures associated with~\eqref{aasass} and~\eqref{aasass1}, we denote them by \( \mu^\sigma \) and \( \tilde{\mu}^\sigma \), respectively.
We now focus on the connection between the invariant measures of the SDEs \eqref{aasass} and \eqref{aasass1} which immediately follows by rescaling. 
%We now focus on the ergodic invariant measure for the that exist .
%Its existence in the extended phase space \(  \mathbb{R}^d\times  \mathcal{B} \) was established in Proposition \ref{YTDS}.
\begin{lemma}\label{contrere} 
%Let \(\mu^{\sigma}\) be the invariant measure of the SDE~\eqref{aasass} such that \((\Pi_{\mathbb{R}^d})_{\star} \mu^{\sigma} = \pi^{\sigma}\) and let \( \tilde{\mu}^{\sigma} \) be the invariant measure of the rescaled SDE~\eqref{aasass1} such that \( (\Pi_{\mathbb{R}^d})_{\star} \tilde{\mu}^{\sigma} = \tilde{\pi}^{\sigma} \). 
For every Borel set \( A \subseteq \mathbb{R}^d \), we have
\begin{align}\label{UAISd96a}
\pi^{\sigma}(A) = \tilde{\pi}^{\sigma}(\|\sigma\|^{-1} A).
\end{align}
\end{lemma}
\proof
%Due to Lemma~\ref{Straightforward}, the rescaled drift term $F^\sigma$ satisfies Assumption~\ref{Drift}.
%Let \( (y_{t,x}^\sigma)_{t \geq 0} \) and \( (Z_{t,x}^\sigma)_{t \geq 0} \) denote the solutions to the SDEs~\eqref{aasass} and~\eqref{aasass1}, respectively, with initial condition \( x \in \mathbb{R}^d \). 
From Proposition~\ref{YTDS}, we have
\begin{align}\label{SAUDajdksd}
\lim_{t \rightarrow \infty} \left\| \mathcal{L}\left(Z_{t,x}^{\sigma}\right) - \tilde{\pi}^{\sigma} \right\|_{\mathrm{TV}} = 0, \quad 
\lim_{t \rightarrow \infty} \left\| \mathcal{L}\left( Y_{t, \|\sigma\| x}^{\sigma} \right) - \pi^{\sigma} \right\|_{\mathrm{TV}} = 0.
\end{align}
This, together with Lemma~\ref{LANAMSWQ} implies the claim.
\qed
%Let us close this part with two remarks:
%\todo{i don't think we need the next remark,\textcolor{red}{Adding this remark was based on your general rule that 'extra explanation improves the value of the paper.' I leave it to you to decide whether to include or remove it.} i said we don't need it. if i write something this is a task which i expect to be done,\textcolor{red}{Thank you. Since at the beginning you said "think," I thought you wanted my reason (I apologize for the misunderstanding); that is why I did not remove it. It has now been removed.
%}}
\begin{comment}
\begin{remark}\label{ASaoas} 
Since $\sigma$ is an invertible matrix, the associated measure $\pi^\sigma$ of the SDE is absolutely continuous with respect to the Lebesgue measure on $\mathbb{R}^d$. Moreover, its density function is positive almost everywhere. In fact, a more general result is established in~\cite[Theorem~1.1]{LPS23}, where the authors provide two implicit functions that yield lower and upper bounds for the density function $\pi^\sigma$, thereby ensuring its positivity almost everywhere. These implicit bounds additionally exhibit Gaussian tails. Obviously, a similar result holds for the SDE~\eqref{aasass1}.
%function can be bounded above and below by two positive functions that exhibit Gaussian-type tails.
\end{remark}
\end{comment}

\begin{remark}\label{comaprison}
In~\cite{LPS23} the following SDE depending on a parameter $\lambda\in\R^d$ was considered
\begin{align}\label{sde:param:xm1}
\txtd Z^\lambda_t = b(\lambda, Z^\lambda_t)~\txtd t + \sigma~\txtd B^H_t, \quad Z^\lambda_0 = x.
\end{align}
In particular, \cite[Theorem~1.5]{LPS23} establishes a uniform bound for the stationary density of $(\Pi_{\mathbb{R}^d})_{\star} \mu^{\lambda}$ as $\lambda$ ranges over a compact set.
This framework is closely related to the SDE \eqref{aasass1}, where we analyze the dependence of the density of $(\Pi_{\mathbb{R}^d})_{\star} \tilde{\mu}^{\sigma}$ on the parameter $\sigma$.~In our case, both drift and diffusion coefficients exhibit this parametric dependence.~This parameter is crucial for our aims, since it's choice will ensure the negativity of the top Lyapunov exponent, as seen below. Therefore we revisit the results \cite{LPS23} keeping track of this parameter. 
%However, by the rescaling argument where the diffusion term lies on the unit sphere, we can effectively handle this additional dependence. %We will formalize these arguments in the remainder of this manuscript.
\end{remark}
\subsection{Negativity of the top Lyapunov exponent}
The main goal of this subsection is to track the dependence of the measure \( \pi^{\sigma} \) ($\tilde{\pi}^{\sigma}$) on the parameter \( \sigma \). As previously mentioned, this measure admits a density with respect to the Lebesgue measure. Our main contribution is to analyze how this density varies with respect to $\sigma$.~More precisely, under Assumption~\ref{Drift2}, we show that the top Lyapunov exponent of the SDE~\eqref{aasass} becomes negative as the intensity of the noise $\sigma$ increases.~We first estimate the top Lyapunov exponent in terms of  the invariant measure \( \pi^{\sigma} \).
\begin{lemma}\label{LYASIGMA}
Let \( \lambda_{1}^{\sigma} \) denote the top Lyapunov exponent of the SDE~\eqref{aasass}, as defined in Proposition~\ref{MET}. Then there exists a constant \( C > 0 \), depending only on the drift term \( F \), such that
\begin{align*}
\lambda_1^{\sigma} \leq -C_{4}^{F} \pi^{\sigma}\left(\mathbb{R}^d \setminus B(0,R)\right) + C \pi^{\sigma}\left(B(0,R)\right).
\end{align*}
%	where \( B(0,R) \) denotes the open ball in \( \mathbb{R}^d \) of radius \( R \) centered at the origin.% and \( R \) is the constant introduced in Assumption~\ref{Drift2}.
\end{lemma}
\proof
%We first proceed as in \cite[Example 3.6]{FGS16a}. 
Let \( \Phi \) denote the RDS associated with the SDE~\eqref{aasass}. From~\eqref{UASs}, it follows that for every \( \tilde{\omega} = (\tilde{\omega}^-, \tilde{\omega}^+) \in \mathcal{B} \times \mathcal{B} \) and for all \( x_1, x_0 \in \mathbb{R}^d \),
\begin{comment}
content...

\begin{align*}
	\frac{\mathrm{d}}{\mathrm{d} t} D_{x}\Phi^{t}_{\omega}(x_0)
	= D_{\Phi^{t}_{\omega}(x_0)}F \big( D_{x}\Phi^{t}_{\omega}(x_0) \big), 
	\quad D_{x}\Phi^{0}_{\omega}(x_0) = x_0.
\end{align*}
Hence, from Assumption~\ref{aasass}, we obtain that
%\todo{wrong label for the assumption, again it is not stated which equation is linearized}
\end{comment}
\begin{align*}
\begin{split}
	\frac{\mathrm{d}}{\mathrm{d} t}\left|\txtD_{x_1}\Phi^{t}_{\tilde{\omega}}(x_0)\right|^2
	= 2 \left\langle D_{\Phi^{t}_{\tilde{\omega}}(x_1)}F\left(\frac{\txtD_{x_1}\Phi^{t}_{\tilde{\omega}}(x_0)}{\left|\txtD_{x_1}\Phi^{t}_{\tilde{\omega}}(x_0)\right|}\right), \frac{\txtD_{x_1}\Phi^{t}_{\tilde{\omega}}(x_0)}{\left|\txtD_{x_1}\Phi^{t}_{\tilde{\omega}}(x_0)\right|} \right\rangle \left|\txtD_{x_1}\Phi^{t}_{\tilde{\omega}}(x_0)\right|^2.
\end{split}
\end{align*}
Therefore,
\begin{align*}
\left|\txtD_{x_1}\Phi^{t}_{\tilde{\omega}}(x_0)\right|^2
= 2\vert x_0\vert^2
\exp\left(\int_{0}^{t} \left\langle D_{\Phi^{s}_{\tilde{\omega}}(x_1)}F\left(\frac{\txtD_{x_1}\Phi^{s}_{\tilde{\omega}}(x_0)}{\left|\txtD_{x_1}\Phi^{s}_{\tilde{\omega}}(x_0)\right|}\right), \frac{\txtD_{x_1}\Phi^{s}_{\tilde{\omega}}(x_0)}{\left|\txtD_{x_1}\Phi^{s}_{\tilde{\omega}}(x_0)\right|} \right\rangle \mathrm{d}s \right).
\end{align*}
From Proposition~\ref{MET}, there exists \( x_0 \neq 0 \) such that for \( \mu^{\sigma} \times \mathbf{P} \)-almost every \( (x_1, \tilde{\omega}) \in \mathbb{R}^d \times \mathcal{B} \times \mathcal{B} \),
%\todo{$\lambda, \lambda^\sigma,\lambda^\sigma_1$...please choose a notation but make it consisent everywhere, not only where the todo is written,\textcolor{red}{I?m sorry ? please accept my apologies. Rest assured that a thorough and careful proofreading of the draft is planned. We?re all human, and occasional mistakes are inevitable.} it's totally fine}
\begin{align}\label{8sd2feww}
\begin{split}
	\lambda_1^{\sigma}&=\lim_{t\rightarrow \infty}\frac{1}{t}\int_{0}^{t}\left\langle D_{\Phi^{s}_{\tilde{\omega}}(x_1)}F\left(\frac{(\txtD_{x_1}\Phi^{s}_{\tilde{\omega}}(x_0)}{\left|\txtD_{x_1}\Phi^{s}_{\tilde{\omega}}(x_0)\right|}\right),\frac{(\txtD_{x_1}\Phi^{s}_{\tilde{\omega}}(x_0)}{\left|\txtD_{x_1}\Phi^{s}_{\tilde{\omega}}(x_0)\right|}\right\rangle\mathrm{d}s \\&\leq \limsup_{t\rightarrow\infty}\frac{1}{t}\int_{0}^{t}\lambda^{+}(\Phi^{s}_{\tilde{\omega}}(x_1))\mathrm{d}s,
\end{split}
\end{align}
where \( \lambda^{+}(y) := \max_{\vert v \vert = 1} \langle \txtD_{y}F(v), v \rangle \).
We define the function
\begin{align*}
\overline{\lambda} \colon \mathbb{R}^d \times \mathcal{B} \times \mathcal{B} \to \mathbb{R}, \quad
\overline{\lambda}(x, \omega) := \lambda^{+}(x).
\end{align*}
Thanks to Assumptions~\ref{Drift2}, one can see that $\overline{\lambda}$ is continuous and bounded.
%\todo{mention that the boundedness of $\lambda^+$ is a consequence of the assumptions on $F$,\textcolor{red}{done}} 
From Corollary~\ref{THETA}, the family \( \{ \Theta_t \}_{t > 0} \) is ergodic with respect to the measure \( \mu^{\sigma} \times \mathbf{P} \). Therefore, by \eqref{8sd2feww}, the monotone convergence theorem, and Birkhoff's ergodic theorem, we conclude that for \( \mu^{\sigma} \times \mathbf{P} \)-almost every \( (x_1, \tilde{\omega}) \in \mathbb{R}^d \times \mathcal{B} \times \mathcal{B} \),
\begin{align*}
\lambda_1^{\sigma} 
&\leq \lim_{t \rightarrow \infty} \frac{1}{t} \int_{0}^{t} \overline{\lambda}(\Theta_s(x_1, \tilde{\omega})) \, \mathrm{d}s \\
&= \int_{\mathbb{R}^d \times \mathcal{B} \times \mathcal{B}} \overline{\lambda}(x, \omega) \, \mu^{\sigma} \times \mathbf{P} \left( \mathrm{d}(x, \omega^-, \omega^+) \right) \\
&= \int_{\mathbb{R}^d} \lambda^{+}(x) \, \pi^{\sigma}(\mathrm{d}x).
\end{align*}
%\todo{relate $\Theta$ in the formula below when you apply Birkhoff's ergodic theorem. you recall that $\Theta$ is ergodic but the reader does not see any $\Theta$ afterwards,\textcolor{red}{true} }
\begin{comment}
\begin{align*}
	\lambda_1^{\sigma}\leq \int_{\R^d\times\mathcal{B}\times\mathcal{B}}\lambda^{+}(x) \mu^{\sigma} \times \mathbf{P}\left(\mathrm{d}(x,\omega^-,\omega^+)\right)=\int_{\R^d}\lambda^{+}(x)\pi^{\sigma}(\mathrm{d}x).
\end{align*}
\end{comment}
Consequently, %\todo{\textcolor{red}{I would like to explain my choice of using $x$ here. Actually, the integral is over the internal variable, and based on your philosophy of precision and logical consistency, I believe this choice is reasonable.
	%} 
%in prop 4.13 you use $\xi$ instead of $x$. here is not so bad but no referee will accept how is written with initial condition SDE $x$, law of the process starting in $x$, $\Phi(x)$ and then $x$ for something else. i don't have time to argue about this anymore and there is nothing that i can do if we get a rejection or a revision about this. \textcolor{red}{sorry for that one, i changed to $x$, just to be sure is it now ok? Please gibe me time to see how i can handle your point in the other parts} }
\begin{align}
\begin{split}
	\lambda_1^{\sigma}&\leq \int_{\R^d \setminus B(0,R)}\lambda^{+}(x)\pi^{\sigma}(\mathrm{d}x)+\int_{B(0,R)}\lambda^{+}(x)\pi^{\sigma}(\mathrm{d}x)\\ &\leq -C_{4}^{F}\pi^{\sigma}\left(\R^d \setminus B(0,R)\right)+\sup_{y\in \R^d}\lambda^{+}(y)\pi^{\sigma}\left(B(0,R)\right),
\end{split}
\end{align}
where we used~\eqref{monotone st} to bound the first term.
\qed
Therefore to make the first Lyapunov exponent $\lambda^\sigma_1$ negative, it is sufficient to prove that as \( \frac{1}{\Vert\sigma^{-1}\Vert}   \) grows, not too much mass of \( \pi^{\sigma} \) concentrates on \( B(0,R) \).  We provide a well-known example that motivates this phenomenon. 
\begin{example}
Recall that the fractional Ornstein-Uhlenbeck process is defined as
\[
\tilde{x}_t^\sigma = \sigma \int_{0}^{t} e^{-(t-s)} \, \mathrm{d}B^{H}_{s},
\]
where \( (B^{H}_t)_{t\geq 0} \) is a fractional Brownian motion. This is a Gaussian process whose covariance matrix is given by
\[
2H \exp(-t) \sigma \sigma^{T} \int_{0}^{t} s^{2H-1} \cosh(t-s) \, \mathrm{d}s.
\]
Therefore, the density of the law \( \mathcal{L}(\tilde{x}_t^\sigma) \) on $\R^d$ becomes flatter as \( \|\sigma^{-1}\| \) increases. In particular, its probability mass within any bounded subset of \( \mathbb{R}^d \) becomes small. This is precisely what we aim to prove for the measure \( \pi^{\sigma} \).
\end{example}
Let us briefly outline the strategy we will follow to show that, by increasing \( \frac{1}{\|\sigma^{-1}\|} \), the Lyapunov exponent \( \lambda^\sigma_1 \) becomes negative. According to Lemma \ref{LYASIGMA}, it is sufficient to show that
\begin{align}\label{UAJSASsd}
\lim_{\Vert\sigma\Vert\rightarrow \infty}\pi^{\sigma}\left(B(0,R)\right)=0.
\end{align}
\begin{itemize}
\item \textbf{Step 1)} To prove this %\todo{statements or assertions, \textcolor{red}{This is just an outline of the proof for the reader convenience. By stating that "These steps will be made mathematically precise in the sequel," I have addressed your point.} i know it's an outline, but still i don't want to write items but and call them steps, statements, assertions etc, \textcolor{red}{So i changed to steps}}, we don't directly deal with $\pi^{\sigma}$ but with the 
we work with the rescaled SDE \eqref{aasass1}. Thus, from Lemma \eqref{contrere}, it is enough to show that
\begin{align}\label{asasqewe}
\lim_{\Vert\sigma\Vert\rightarrow \infty}\tilde{\pi}^{\sigma}\left(\|\sigma\|^{-1}B(0,R)\right)=0.
\end{align}
\item \textbf{Step 2)} We derive a uniform bound for the density of \eqref{aasass1} when \(\sigma\) belongs to \(T_{\theta,\kappa}\) %\todo{$T_{\theta,\kappa}$ should be consistent everywhere,\textcolor{red}{true}} 
for every arbitrary but fixed value of \(\theta \geq 1\). To achieve this, we follow the approach in \cite{LPS23} keeping track of the dependence of $\tilde{\pi}^\sigma$ on $\sigma$. 
%\item \textbf{Step 4)}  The conclusion follows from \eqref{asasqewe}.% since by increasing \( \frac{1}{\|\sigma^{-1}\|} \) (and thus \( \|\sigma\| \)) the set becomes smaller, and the uniform bound on the density ensures that a sudden change close to the Dirac measure never occurs.
\end{itemize}
%To remain consistent with other parts of the paper, we adopt the following notation: this should always the case in a paper and this obvious fact should not be emphasized
\begin{notation}
Recall that $B^H_t\left(P_t(\omega^-, \omega^+)\right) = -\left(\mathcal{D}_H P_t(\omega^-, \omega^+)\right)(-t).$
For simplicity, we write \( B^H_t \) instead of \( B^H_t\left(P_t(\omega^-, \omega^+)\right) \). Furthermore, \( \mathbb{P} = \mathbf{P} \times \mathbf{P} \) as before.
\end{notation}
We now focus on the density of the SDE~\eqref{aasass1}. 
%We need some auxiliary steps to achieve our goal. While we mostly follow \cite{LPS23}, as noted in Remark~\ref{comaprison}, what we aim for cannot be directly deduced from it. Therefore, several of its results must be modified to fit the current context. Let us start with the following lemma:

%\todo{why do we not state it with $\sup \sigma \in T_{\theta,\kappa}$ ?,\textcolor{red}{$\kappa$ is added actually here $\theta$ is not necessary. The result is stated with minimum assumption} we stated in exactly our setting, if we assume $\sigma\in T_{\theta,\kappa}$ we stick to it and make remarks for what statements we can weaken the assumptions,\textcolor{red}{true, So i will add a remark}}
\begin{lemma}\label{sds85sdarr}
There exists a constant \(\rho_{\kappa,1} > 0\) such that
\[
\sup_{\sigma\in T_{\theta,\kappa} }\int_{\mathbb{R}^d} \exp\bigl(\rho_{\kappa,1} |y|^2\bigr) \, \tilde{\pi}^\sigma(\mathrm{d}y) < \infty.
\]

\end{lemma}
\proof
Recall that
\begin{align*}
\begin{cases}
	&\mathrm{d}Z_{t}^{\sigma} =F^{\sigma}(Z_{t}^{\sigma}) \, \mathrm{d}t +  \, \Vert\sigma\Vert^{-1}\sigma\mathrm{d} B^{H}_t\\ & Z_{0}^{\sigma}=0
\end{cases}
\end{align*}
and consider the linear SDE
\begin{align}\label{ASAPOq}
\begin{cases}
	&\mathrm{d}X_{t}^{\sigma} =-X_{t}^{\sigma} \, \mathrm{d}t +  \, \Vert\sigma\Vert^{-1}\sigma\mathrm{d} B^{H}_t \\ &X_{0}^{\sigma}=0.  
\end{cases}
\end{align}
The goal is to obtain a bound for \( Z_t^\sigma \) in terms of a fractional Ornstein-Uhlenbeck process that is independent of \(\sigma\). Let \( x^I \) denote the solution to the linear SDE~\eqref{ASAPOq}, where \(\sigma\) is replaced by the identity matrix \(I\). Then we have
\begin{align*}
X_{t}^{\sigma} = \|\sigma\|^{-1} \sigma \int_0^t \exp(-(t - s)) \, \mathrm{d}B^{H}_{s},
\end{align*}
and $|X_{t}^{\sigma}| \leq |X_{t}^{I}|$. Thanks to Lemma~\ref{Straightforward} and Young's inequality, we obtain
\begin{align*}
\frac{\mathrm{d}}{\mathrm{d}t}\vert Z_{t}^{\sigma}-X_{t}^{\sigma}\vert^2&= 2\left\langle F\left(X_{t}^{\sigma}+\left(Z_{t}^{\sigma}-X_{t}^{\sigma}\right)\right) - F(X_{t}^{\sigma}), Z_{t}^{\sigma}-X_{t}^{\sigma}\right\rangle\\&+ 2\langle F(X_{t}^{\sigma})+X_{t}^{\sigma}, Z_{t}^{\sigma}-X_{t}^{\sigma}\rangle\\&\leq 2C_{1}^{F}-2C_{2}^{F}\left| Z_{t}^{\sigma}-X_{t}^{\sigma}\right|^2+\frac{1}{C_2^F} \left|X_{t}^{\sigma} + F^{\sigma}(X_{t}^{\sigma})\right|^2+C_2^F \left|Z_{t}^{\sigma}-X_{t}^{\sigma}\right|^2\\&\leq -C_2^F \left|Z_{t}^{\sigma}-X_{t}^{\sigma}\right|^2 +2C_1^F + \frac{1}{C_2^F} \left|X_{t}^{\sigma} + F^{\sigma}(X_{t}^{\sigma})\right|^2.
\end{align*}
\begin{comment}  
\begin{align*}
	\frac{\mathrm{d}}{\mathrm{d}t}\vert Z_{t}^{\sigma}-X_{t}^{\sigma}\vert^2\leq  -C_2^F \left|Z_{t}^{\sigma}-X_{t}^{\sigma}\right|^2 +2C_1^F + \frac{1}{C_2^F} \left|X_{t}^{\sigma} + F^{\sigma}(X_{t}^{\sigma})\right|^2.
\end{align*}
\end{comment}

By Gr\"onwall's inequality and the growth condition $|F(X^\sigma_t)| \leq C_F(1+|X^\sigma_t|)$, we conclude that there exist constants \( A_1 \) and \( A_2 \) depending on \( F \) (and also on $\kappa$ assumed here to be equal to one for simplicity) %\todo{not a footnote,\textcolor{red}{Thanks. Please note that, following your rules about being precise with every sentence, it is sometimes necessary to use a footnote. For example, here, including these details in the body of the paper might be distracting. I hope you accept this exception}} \footnote{They also depend on \( \kappa \)  which we have assumed to be equal to one for simplicity.} 
such that
%\todo{reformulate the except...is very confusing to state it like that, $\alpha$ is used for fractional calculus in appendix,\textcolor{red}{done, I know you don't like footnote, but i think here is ok}}
\begin{align}\label{Ioas63as}
\vert Z_{t}^{\sigma}-X_{t}^{\sigma}\vert&\leq A_1+A_2{\sqrt{\int_{0}^{t}\exp\left(-C_{2}^{F}(t-\tau)\right) |X_{\tau}^{\sigma}|^{2}\mathrm{d}\tau}}.
\end{align}
For \( C > 0 \) and \( \bar{X} \in C([0,t],\mathbb{R}^d) \), we define the following norm
\[
\|\bar{X} \|_{\mathcal{N}_{C}[0,t]} := \left( \int_0^t \exp \bigl(-C(t - \tau)\bigr)\, |\bar{X} _{\tau}|^{2} \,\mathrm{d}\tau \right)^{\!1/2}.
\]
\begin{comment}  
\[
\mathcal{N}_{C}[0,t] := \Biggl\{ x \in C([0,t],\mathbb{R}^d) \;\Bigg|\; 
\|x\|_{\mathcal{N}_{C}[0,t]} := \left( \int_0^t \exp\!\bigl(-C(t - \tau)\bigr)\, |x_{\tau}|^{2} \,\mathrm{d}\tau \right)^{\!1/2} < \infty \Biggr\}.
\]

\[
\| x^{\sigma} \|_{\mathcal{N}_{C_{2}^{F}}[0,t]} := \left( \int_0^t \exp\!\bigl(-C_{2}^{F}(t - \tau)\bigr)\, |x_{\tau}^{\sigma}|^{2} \,\mathrm{d}\tau \right)^{\!1/2}.
\]
\end{comment}

Then, for \( C = C_{2}^{F} \), it is clear that
\[
\| X^{\sigma} \|_{\mathcal{N}_{C_{2}^{F}}[0,t]} \leq \| X^{I} \|_{\mathcal{N}_{C_{2}^{F}}[0,t]}.
\]
Hence, by \eqref{SAUDajdksd}, it follows that for \( Z_{t}^{\sigma} = Z_{t,0}^{\sigma} \) and for every \( \mathcal{R} > 2 A_1 \), we have
\begin{align}\label{Umas96as}
\begin{split}
	\mathbb{P}\left(\left\vert Z_{t}^{\sigma}\right\vert\geq \mathcal{R} \right)&\leq  \mathbb{P}\left(\left\vert Z_{t}^{\sigma}-X_{t}^{\sigma}\right\vert\geq\frac{\mathcal{R}}{2} \right)+\mathbb{P}\left(\left\vert X_{t}^{\sigma}\right\vert\geq\frac{\mathcal{R}}{2} \right)\\&\leq  \mathbb{P}\left(\Vert X^{I}\Vert_{\mathcal{N}_{C_{2}^{F}}[0,t]}\geq \frac{\frac{\mathcal{R}}{2}-A_1}{A_2}\right)+\mathbb{P}\left(\left\vert X_{t}^{I}\right\vert\geq\frac{\mathcal{R}}{2} \right).
\end{split}
\end{align}
Since \( X^{I} \) is a fractional Ornstein-Uhlenbeck process and hence Gaussian, it follows from \cite[Lemma 2.7]{LPS23} that there exist constants \( C_1, C_2 > 0 \), independent of \( \sigma \), such that %\todo{the bound below holds for any $\cR\in \R^d$ not for $\cR>2A_1$,\textcolor{red}{True, but since it is better not to introduce extra notation and \(\mathcal{R}\) is already specified, I remove this condition.}}
\begin{align}\label{Umas96ass}
\max \left\lbrace 
\sup_{t \geq 0} \mathbb{P} \left( \left\| X^{I} \right\|_{\mathcal{N}_{C_{2}^{F}}[0,t]} \geq \mathcal{R} \right),\ 
\sup_{t \geq 0} \mathbb{P} \left( \left| X_{t}^{I} \right| \geq \frac{\mathcal{R}}{2} \right) 
\right\rbrace 
\leq C_1 \exp(-C_2 \mathcal{R}^{2}). %, \quad \forall~ \mathcal{R} > 2 A_1.
\end{align}
From \eqref{SAUDajdksd}, it follows that
\begin{align*}%\label{A{SPa52a}}
\lim_{t\rightarrow \infty}\mathbb{P}\!\left( \left| Z_{t}^{\sigma} \right| \geq \mathcal{R} \right)
= \tilde{\pi}^{\sigma}\!\left( y \in \mathbb{R}^d : |y| \geq \mathcal{R} \right)\leq C_1 \exp (-C_2 \cR^2),
\end{align*}
%Thus, thanks to \eqref{Umas96as} and \eqref{Umas96ass}, we can conclude that 
% 
%decays at the rate \(\exp(-C_2 \mathcal{R}^{2})\) for large \(\mathcal{R}\) 
proving the claim.
%\todo{add a detail relating this to the claim\textcolor{red}{Following your comment, some details have been added.}}
%This combined with \eqref{Umas96as} and \eqref{Umas96ass} proves the claim. 
%\todo{the solution is denoted by the same symbol, no $Z$ and $z$,\textcolor{red}{When we use capital letter, we want to state the initial value (here zero)} i don't agree with the notation. we state the initial data or give another parameter but do not change to capital letters for the solution just to give the dependence on the initial data, \textcolor{red}{It should be now ok}}
\qed
\begin{remark}
Note that it is not necessary to assume that \( \sigma \in T_{\theta,\kappa} \) for the statement of Lemma~\ref{sds85sdarr}, but only that \( \sigma \) is an invertible matrix and that \( \|\sigma\| \geq \kappa \). 
However, 
we state Lemma~\ref{sds85sdarr} under this stronger assumption, since this will be required for the following results.
\end{remark}
\begin{comment}
\begin{remark}\label{UJAs65as}
Similar to what we stated in Remark \ref{WEKA}, the condition \( \Vert \sigma \Vert \geq 1 \) in Lemma \ref{sds85sdarr} can also be replaced by \( \Vert \sigma \Vert \geq \kappa \) for any fixed \( \kappa > 0 \).  
The choice \( \kappa = 1 \) is made purely for notational simplicity.
\end{remark}
\end{comment}
As a consequence of Lemma~\ref{sds85sdarr}, we obtain the following result.
\begin{corollary}\label{Fernique1}
There exists a constant \( \rho_{\kappa,2} > 0 \) such that
\[
\sup_{\sigma\in T_{\theta,\kappa} }\int_{\mathbb{R}^d} \exp\left(\rho_{\kappa,2}\left(|x|^2+\vert \omega^-\vert_{\mathcal{B}}^2\right) \right) \, \tilde{\mu}^{\sigma}\left(\mathrm{d}(x,\omega^-)\right) < \infty.
\]
\end{corollary}
\proof
This follows  from Lemma~\ref{sds85sdarr} and the fact that \( \mathbf{P} \) is a Gaussian measure on \( \mathcal{B} \). %For more details, see \cite[Corollary 2.9]{LPS23}.
\qed
Recalling that $\tilde{B}^H_t$ is the Liouville operator defined in Lemma \ref{SDASASASAa} and the decomposition of fbm stated in \eqref{SALOSWa}, we state the following result on the representation of the density of \eqref{aasass1} at an arbitrary time $t_0>0$. 
\begin{proposition}\label{law1}
Let \(\sigma \in T_{\theta,\kappa}\), \( l \in C^{H-}_{\mathrm{loc}}(\mathbb{R}^+, \mathbb{R}^d) \), and \( t_0 > 0 \). Then, the following equation  %\todo{$\tilde{B}^H$, we recall it in rem 5.17 and state why $\omega^+$,\textcolor{red}{ Let us discuss this matter in person. I added one sentence before, and I hope this can address your point.}. let's discuss on Monday, i would relate only $\tilde{B}$ to the lemmas, not the following proposition. }  
\begin{align}\label{WEEEAAS}
\tilde{\Psi}_{t}^{\sigma}(l) 
= l(t) 
+ \int_{0}^{t} F^{\sigma}\!\left(\tilde{\Psi}_{s}^{\sigma}(l)\right)\,\mathrm{d}s 
+ \Vert \sigma \Vert^{-1} \sigma \tilde{B}_{t}^{H}(\omega^+)
\end{align}
has a weak solution which is unique in law. 
Furthermore, \(\tilde{\Psi}^{\sigma}_{t_0}(l)\) possesses a density with respect to the Lebesgue measure on \(\mathbb{R}^d\). This density is given by
\begin{align}\label{UAJSMs}
\tilde{p}^{\sigma}_{t_0}(l;y)
= \frac{1}{\left(\sqrt{2\pi} \rho_H {t_0}^H\right)^d \det(\|\sigma\|^{-1} \sigma)}
\exp\left(-\|\sigma\|^2 \frac{|\sigma^{-1}(y - l({t_0}))|^2}{\rho_H^2 {t_0}^{2H}}\right)
G_{t_0}^{\sigma}(l;y),
\end{align}
where \(\rho_H > 0\) is a normalization constant. Furthermore, for \( z := \|\sigma\| \sigma^{-1}(y - l(t_0)) \), we have  %\todo{why not $G^\sigma_{t_0}(z):=$ instead of $G^{\sigma}_{t_0}(l;y)=$\textcolor{red}{Now that everything is written in this style, let?s stick to it, as it has several advantages} i see but if one strictly compares the def of $G(;y)$ there is no $y $ on the lhs, it's clear that $z$ depends on $y$ but usually the arguments must coincide, anyway we can leave it like this for now  }
\begin{align}\label{G_T}
G_{t_0}^{\sigma}(l;y) = \mathbb{E}\left[
\exp\left(
\int_0^{t_0} \left\langle \mathcal{L}^{\sigma,l,z,t_0}_s, \, \mathrm{d}X_s^{z,t_0} \right\rangle
- \frac{1}{2} \int_0^{t_0} \left| \mathcal{L}^{\sigma,l,z,t_0}_s \right|^2 \, \mathrm{d}s
\right)
\right],
\end{align}
with
\begin{align}\label{LLpp2a}
\mathcal{L}^{\sigma,l,z,t_0}_s := \frac{\|\sigma\| \sigma^{-1}}{\rho_H}
\left(\mathcal{J}_{+}^{\frac{1}{2} - H} F^{\sigma}\left(l + \frac{\rho_H}{\|\sigma\|} \sigma \mathcal{J}_{+}^{ H-\frac{1}{2}} X^{z,t_0}\right)\right)\left(s\right),
\end{align}
and \((X_s^{z,t_0})_{s \in [0,t_0]}\) is a semimartingale defined by
\begin{align}\label{SEMIMAR}
\begin{split}
	\mathrm{d}X_s^{z,t_0} &= \frac{2H}{\rho_H}(t_0 - s)^{H - \frac{1}{2}} \left( \frac{z}{t_0^{2H}} - \rho_H \int_0^s (t_0 - u)^{-H - \frac{1}{2}} \, \mathrm{d}W_u \right) \mathrm{d}s + \mathrm{d}W_s \\
	&= K_{s}(z,t_0,W)\,\mathrm{d}s + \mathrm{d}W_s, \quad X_0^{z,t_0} = 0.
\end{split}
\end{align}
\begin{comment}
content...
\begin{align*}
	\mathrm{d}X_s^z = \frac{2H}{\rho_H}(t - s)^{-H - \frac{1}{2}} \left(z - \rho_H \int_0^s (t - u)^{H-\frac{1}{2}} \, \mathrm{d}X_u^z\right) \mathrm{d}s + \mathrm{d}W_s, \quad X_0^z = 0.
\end{align*}
\end{comment}
\end{proposition}
\proof
Follows from \cite[Proposition 3.4, Lemma 3.6, Corollary 4.2, Lemma 4.4 and Proposition 4.5]{LPS23}.
\qed
%\begin{remark} 

%\item [2)] One can show that the process \( (X_s^{z,t_0})_{s \in [0, t_0]} \) satisfies %\todo{$\tilde{B}$?\textcolor{red}{now clarified}}
%\begin{align*}
%	(X_s^{z,t_0})_{s \in [0, t_0]} \overset{d}{=} \mathcal{L} \left( (W_s)_{s \in [0, t_0]} \,\middle|\, \tilde{B}_{t_0}^H = z \right).
%\end{align*}
%where \(\tilde{B}^{H}_{t_0}\) is defined in \eqref{FF15}. 
%This regular conditional law is referred to as the {Wiener-Liouville bridge}, see \cite[Corollary 4.2]{LPS23} for more details. 
%\end{itemize}

%\todo{mention that we have a strong solution as a second point of the remark,\textcolor{red}{done}}% one can readily deduce~\eqref{LALAS} since \( (\tilde{\Phi}^{t, \sigma})_{t \geq 0} \) is a strong solution of~\eqref{aasass}. 
%\end{remark}
As a consequence of  Proposition~\ref{law1}, and~\cite[Proposition 2.4]{LPS23}, the density of \(\tilde{\pi}^\sigma\) can be expressed in terms of the density \(\tilde{\Psi}^\sigma_{t_0}\) at an arbitrary time \(t_0 > 0\) as follows.\\

In our situation, substituting \( l = x + \|\sigma\|^{-1} \sigma \, \mathcal{P}(\omega^-) \) into~\eqref{WEEEAAS}, we obtain the unique strong solution of the SDE~\eqref{aasass1} given by Lemma~\ref{OALSdfe}. This will be considered below.
\begin{proposition}	\label{REPREDDA}
Fix \( t_0 > 0 \) and for \( \omega^- \in \mathcal{B} \) and \( x \in \mathbb{R}^d \), define
\[
l^{x, \sigma, \omega^-} := x + \| \sigma \|^{-1} \sigma \, \mathcal{P}(\omega^-).
\]
Then the law of \( \tilde{\Psi}_{t_0}^{\sigma}(l^{x, \sigma, \omega^-}) \) admits a density with respect to the Lebesgue measure denoted by \( \tilde{p}_{t_0}(l^{x, \sigma, \omega^-}; y) \).  Moreover, the following identity holds
\begin{align}\label{sdse52s}
\tilde{p}^{\sigma}_{\infty}(y) = \int_{\mathbb{R}^d \times \mathcal{B}} \tilde{p}_{t_0}(l^{x, \sigma, \omega^-}; y) \, \tilde{\mu}^{\sigma}\left(\mathrm{d}(x, \omega^-)\right),
\end{align}
where \( \tilde{p}^{\sigma}_{\infty}(y) \) denotes the density of \( \tilde{\pi}^{\sigma} \).
\end{proposition}
In light of Proposition~\ref{law1} and Proposition~\ref{REPREDDA}, in order to track the dependence of $\tilde{p}^{\sigma}_{\infty}$ on $\sigma$, it is necessary to bound the terms $G^\sigma$ and $\mathcal{L}^\sigma$ introduced in~\eqref{G_T} and~\eqref{LLpp2a}. 
To this end, we require several auxiliary results, which are collected in Appendix~\ref{AAZAZA} for the convenience of the reader.\\

The following statement provides a uniform Gaussian bound for the stationary density $\tilde{p}^\sigma_\infty$ on bounded subsets of $\mathbb{R}^d$, provided that $\sigma$ is chosen sufficiently large.
\begin{proposition}\label{INFORM}
For every bounded set \( B(0, R) \subset \mathbb{R}^d \), there exist constants \( C_1, C_2 > 0 \), depending on \( \theta \), \( H \), \( \kappa \), and \( R \), such that
\begin{align*}
\sup_{\substack{\sigma \in T_{\theta,\kappa}}} \tilde{p}^{\sigma}_{\infty}(y) \leq C_1 \exp(-C_2 |y|^2) \text{  for all  } y \in B(0,R).
\end{align*}
\end{proposition}
\proof
We first outline the strategy of the proof. Based on Lemma~\ref{SIMLIF}, the first step is to estimate the expression
\begin{align}\label{UJKASsd}
\exp\big( -\frac{ | Z^{x,y,\sigma,\omega^-} |^2 }{ \rho_H^2 t_0^{2H} } \big) G_{t_0}^{\sigma}(l^{x, \sigma, \omega^-}; y).  
\end{align}
%\todo{reformulate this bound..remains uniformly bounded,\textcolor{red}{done} it is still written this bound, if \(t_0\) is sufficiently small, for all \(\sigma \in T_{\theta,\kappa}\), when integrated with respect to the measure \(\tilde{\mu}^{\sigma}\), remains uniformly bounded.\textcolor{red}{I re-write this part form the beginning}}
To this end, we use Proposition~\ref{YHHNA7as} to conclude that there exist functions ${\Xi}$ and $\tilde{\Xi}$ such that 
\[
\exp\left( -\frac{ \lvert Z^{x,y,\sigma,\omega^-} \rvert^2 }{ \rho_H^2 t_0^{2H} } \right) 
G_{t_0}^{\sigma}(l^{x, \sigma, \omega^-}; y) 
\leq \Xi(t_0,y)\,\tilde{\Xi}(t_0,x,\omega^-),
\]
where $\Xi$ and $\tilde{\Xi}$ also depend on \( \theta \), \( \kappa \), \( R \), and \( H \).  
We then show that once $t_0$ is sufficiently small,  
\[
\int_{\mathbb{R}^d \times \mathcal{B}} \tilde{\Xi}(t_0,x,\omega^-)\, 
\tilde{\mu}^{\sigma}\!\left( \mathrm{d}(x, \omega^-) \right) < \infty.
\]
%and that 
%\[
%\Xi(t_0,y) 
%\int_{\mathbb{R}^d \times \mathcal{B}} 
%\tilde{\Xi}(t_0,x,\omega^-)\, \tilde{\mu}^{\sigma}\!\left( \mathrm{d}(x, \omega^-) \right)
%\]
which proves~\eqref{NM12a}.  
%Based on these considerations, we now proceed to the proof.  
The proof is divided into two cases corresponding to the range of the parameter \( H \).  
From Lemma~\ref{SIMLIF} and Proposition~\ref{YHHNA7as}, we have:
%\todo{they will be fixed, usually we do not write this in a proof,\textcolor{red}{Since several notations will be used in what follows, we begin by recalling the key parameters} recall also $\alpha_1,\alpha_2\in(0,1)$ etc, \textcolor{red}{We need to find a way to address that. At the moment, I wrote at the beginning that we are referring to those results, so the reader is likely aware that these values are positive.}}
\begin{itemize}
\item if \( 0 < H < \frac{1}{2} \):
\begin{align}\label{YB120as}
	\begin{split}
		&\exp\big( -\frac{ | Z^{x,y,\sigma,\omega^-} |^2 }{ \rho_H^2 t_0^{2H} } \big) G_{t_0}^{\sigma}(l^{x, \sigma, \omega^-}; y)\\&\leq \Gamma_{1}(t_0,\theta)K^{2}_{H,\theta,\kappa}\exp\bigg(K^{1}_{H,\theta,\kappa}t_{0}^{\alpha_1}\Vert l^{x, \sigma, \omega^-}\Vert_{\infty;[0,t_0]}^2 \bigg)\exp\bigg(\big(K^{1}_{H,\theta,\kappa}t_{0}^{\alpha_1}-\frac{ 1 }{ \rho_H^2 t_0^{2H} }\big)| Z^{x,y,\sigma,\omega^-} |^2\bigg)\\&=\Gamma_{1}(t_0,\theta)\Lambda_{1}(t_0,x,\sigma,\omega^-) \exp\bigg(\big(K^{1}_{H,\theta,\kappa}t_{0}^{\alpha_1}-\frac{ 1 }{ \rho_H^2 t_0^{2H} }\big)| Z^{x,y,\sigma,\omega^-} |^2\bigg),
	\end{split}
\end{align}
\item if \( \frac{1}{2} < H < 1 \):
\begin{align}\label{YB120as1}
	\begin{split}
		&\exp\big( -\frac{ | Z^{x,y,\sigma,\omega^-} |^2 }{ \rho_H^2 t_0^{2H} } \big) G_{t_0}^{\sigma}(l^{x, \sigma, \omega^-}; y)\leq\Gamma_{2}(t_0,\theta)K^{2}_{H,\theta,\kappa}\exp(K^{3}_{H,\theta,\kappa}t_{0}^{1-2H})\times \\&\exp\bigg(K^{1}_{H,\theta,\kappa}t_{0}^{\alpha_1}\Vert l^{x, \sigma, \omega^-}\Vert_{H-\frac{1}{2}+\eta;[0,{t_0}]}^2\bigg) \exp\bigg(\big(K^{1}_{H,\theta,\kappa}t_{0}^{1-2H}-\frac{ 1 }{ \rho_H^2 t_0^{2H} }\big)| Z^{x,y,\sigma,\omega^-} |^2\bigg)\\&=\Gamma_{2}(t_0,\theta)\Lambda_{2}(t_0,x,\sigma,\omega^-)\exp\bigg(\big(K^{1}_{H,\theta,\kappa}t_{0}^{1-2H}-\frac{ 1 }{ \rho_H^2 t_0^{2H} }\big)|Z^{x,y,\sigma,\omega^-}|^2\bigg).
	\end{split}
\end{align}
\end{itemize}
We can choose \( t_0 \) sufficiently small (depending only on \( H \), \( \theta \), and \( \kappa \)) such that %\todo{better to state everything first for $0<H<1/2$ and then $1/2<H<1$ }
\begin{align}\label{AUSA63ad}
\begin{split}
	&\exp\bigg(\big(K^{1}_{H,\theta,\kappa}t_{0}^{\alpha_1}-\frac{ 1 }{ \rho_H^2 t_0^{2H} }\big) | Z^{x,y,\sigma,\omega^-} |^2\bigg)
	\leq \exp\big(-\frac{| Z^{x,y,\sigma,\omega^-} |^2}{2 \rho_H^2 t_0^{2H}} \big), \quad \text{if }\  0 < H < \frac{1}{2}, \\
	&\exp\bigg(\big(K^{1}_{H,\theta,\kappa}t_{0}^{1-2H}-\frac{ 1 }{ \rho_H^2 t_0^{2H} }\big)| Z^{x,y,\sigma,\omega^-} |^2\bigg)
	\leq \exp\big(-\frac{| Z^{x,y,\sigma,\omega^-} |^2}{2 \rho_H^2 t_0^{2H}} \big), \quad \text{if } \ \frac{1}{2} < H < 1.
\end{split}
\end{align}
%by choosing \( t_0 \) small enough so that \eqref{AUSA63ad} holds,
We now estimate term $Z^{x,y,\sigma,\omega^-}$.	First, since \( y \in B(0, R) \), it follows from Assumption \ref{theth} that for all \( \sigma \in T_{\theta,\kappa} \)
\begin{align*}
\|\sigma\| \sigma^{-1} y \in B\left(0, \theta R\right).
\end{align*}
Moreover, for every \(x \in \mathbb{R}^d\) and \(\omega^-\in\mathcal{B}\), we have 
\begin{align*}
\left| \|\sigma\| \sigma^{-1} x + \mathcal{P}(\omega^-)(t_0) \right| \leq \theta |x| + \left| \mathcal{P}(\omega^-)(t_0) \right|.
\end{align*}
Therefore, since
\begin{align*}
Z^{x,y,\sigma,\omega^-} = \|\sigma\| \sigma^{-1} y - \left( \|\sigma\| \sigma^{-1} x + \mathcal{P}(\omega^-)(t_0) \right),
\end{align*}
Lemma~\ref{uja63sfdaa}  implies that for every \(\zeta > 0\), there exists a constant \(\zeta^{R,t_0,\theta} > 0\) such that 

\begin{align}\label{YB120as2}
\exp\big(-\frac{| Z^{x,y,\sigma,\omega^-} |^2}{2 \rho_H^2 t_0^{2H}} \big)\leq \exp\left(\zeta\left(\vert x\vert^2+\vert \mathcal{P}(\omega^-)(t_0)\vert^2\right)\right)\exp\left(-\zeta^{R,t_0,\theta}\vert y\vert^2\right).
\end{align}
Now, by combining Lemma~\ref{SIMLIF} with \eqref{YB120as}--\eqref{YB120as2}, %\todo{Do not enumerate all these labels; take the final form. \textcolor{red}{Thanks for the comment. As you know, the argument is based on all of them, and we need to consider each. I have made a few changes to address your point.}}
we conclude that if \( t_0 \) is sufficiently small, then for every \( y \in B\!\left(0, \theta R\right) \)   
\begin{itemize}
\item if \( 0 < H < \frac{1}{2} \), we have
\begin{align}\label{PY11}
	\begin{split}
		&	\tilde{p}^{\sigma}_{\infty}(y)\leq \mathcal{R}(t_0, H,\theta)\Gamma_{1}(t_0,\theta)\exp\left(-\zeta^{R,t_0,\theta}\vert y\vert^2\right)\times \\&\int_{\R^d \times \mathcal{B}} \Lambda_{1}(t_0,x,\sigma,\omega^-) \exp\left(\zeta\left(\vert x\vert^2+\vert \mathcal{P}(\omega^-)(t_0)\vert^2\right)\right)\tilde{\mu}^{\sigma}\left(\mathrm{d}(x, \omega^-)\right).
	\end{split}
\end{align}
\item if \( \frac{1}{2} < H < 1 \), we have
\begin{align}\label{PY22}
	\begin{split}
		&	\tilde{p}^{\sigma}_{\infty}(y)\leq \mathcal{R}(t_0, H,\theta)\Gamma_{2}(t_0,\theta)\exp\left(-\zeta^{R,t_0,\theta}\vert y\vert^2\right)\times\\&\int_{\R^d \times \mathcal{B}} \Lambda_{2}(t_0,x,\sigma,\omega^-) \exp\left(\zeta\left(\vert x\vert^2+\vert \mathcal{P}(\omega^-)(t_0)\vert^2\right)\right)\tilde{\mu}^{\sigma}\left(\mathrm{d}(x, \omega^-)\right).
	\end{split}
\end{align}
\end{itemize}
For \( 0 < \eta < \min\!\left\{ H, \tfrac{1-H}{2} \right\} \) it holds that %\todo{state what's the next goal, what do we want to estimate relating to what we already proved to make a transition to the next step, the previous estimates do not depend on $l$ why only need to estimate $\mathcal{P}$ which follows from the continuity so it's not clear for the reader why we state estimates for $l$,\textcolor{red}{based on this comment, i added few lines}}
\begin{align*}%\label{OLA584sdx}
\begin{split} 
	&\left\Vert l^{x, \sigma, \omega^-}\right\Vert_{\infty;[0,t_0]}^2\leq 2\left(\vert x\vert^2+\left\Vert\mathcal{P}(\omega^-)\right\Vert_{\infty;[0,{t_0}]}^2\right) , \quad \text{if } 0 < H < \frac{1}{2}, \\
	&\left\Vert l^{x, \sigma, \omega^-}\right\Vert_{H-\frac{1}{2}+\eta;[0,{t_0}]}^2\leq 2\left(\vert x\vert^2+\left\Vert\mathcal{P}(\omega^-)\right\Vert_{H-\frac{1}{2}+\eta;[0,{t_0}]}^2\right)  , \quad \text{if } \frac{1}{2} < H < 1.
\end{split}
\end{align*}
The goal now is to estimate the norm of $\mathcal{P}(\omega^-)$ in terms of the norm $|\omega^-|_{\mathcal{B}}$.
Note that \( H - \frac{1}{2} + \eta < \frac{H}{2} \) and therefore by Lemma \ref{Liouville}, the map %\todo{2.26 is a lemma, not a definition,\textcolor{red}{true and sorry}} the mapping
\[
\mathcal{P} : \mathcal{B} \to C^{H - \frac{1}{2} + \eta}([0,1])
\]
is continuous.
%\footnote{By assumption, \( 0 < t_0 \leq 1 \).}. 
Thus there exists a constant \( C \) such that
\begin{align*}%\label{YHAsd63a}
\max\left\lbrace\Vert \mathcal{P}(\omega^-)\Vert_{H - \frac{1}{2} + \eta;[0,t_0]},\left\Vert\mathcal{P}(\omega^-)\right\Vert_{\infty;[0,{t_0}]}\right\rbrace\leq C\vert\omega^-\vert_{\mathcal{B}}.
\end{align*}
Based on this and on \eqref{PY22} we can find constants \( K \) and \( \tilde{K} \), both depending on \( H \), \( t_0 \), \( \theta \) and \( \kappa \) such that
\begin{align}\label{IK36a}
\begin{split}
	&	\Lambda_{1}(t_0, x, \sigma, \omega^-) \leq \tilde{K} \exp\left(K t_0^{\alpha_1} \left( \vert x \vert^2 + \vert \omega^- \vert_{\mathcal{B}}^2 \right) \right) , \quad \text{if } 0 < H < \frac{1}{2},\\
	& 	\Lambda_{2}(t_0, x, \sigma, \omega^-) \leq \tilde{K}\exp(K^{3}_{H,\theta,\kappa}t_{0}^{1-2H}) \exp\left(K t_0^{\alpha_1} \left( \vert x \vert^2 + \vert \omega^- \vert_{\mathcal{B}}^2 \right) \right) , \quad \text{if } \frac{1}{2} < H < 1.
\end{split}
\end{align}
For the rest of the argument, we do not need to distinguish between \( H \in \left(0, \tfrac{1}{2}\right) \) and \( H \in \left(\tfrac{1}{2}, 1\right) \), as the proof remains the same in both cases. Therefore, we focus on the case \( 0 < H < \tfrac{1}{2} \). From \eqref{PY11} and \eqref{IK36a}, we obtain
\begin{align}\label{INTEFa}
\begin{split}
	&\tilde{p}^{\sigma}_{\infty}(y)\leq \tilde{K} \mathcal{R}(t_0, \kappa,\theta)\Gamma_{1}(t_0,\theta)\exp\left(-\zeta^{R,t_0,\theta}\vert y\vert^2\right)\times\\&\int_{\R^d \times \mathcal{B}} \exp\left(\left(K t_0^{\alpha_1}+\zeta\right)\vert x\vert^2+\left(K t_0^{\alpha_1}+\zeta\right) \vert \omega^- \vert_{\mathcal{B}}^2\right)\tilde{\mu}^{\sigma}\left(\mathrm{d}(x, \omega^-)\right).
\end{split}		
\end{align}
Now, by choosing \( t_0 \) and \( \zeta \) sufficiently small, we can use Corollary~\ref{Fernique1} to uniformly bound the integral in \eqref{INTEFa} for every \( \sigma \in T_{\theta, \kappa} \) yielding the claim. From \eqref{INTEFa}, we can track the dependence of the constants on \( R \), \( t_0 \), \( \kappa \), and \( \theta \).
\qed
Based on this result, we can now obtain the negativity of the top Lyapunov exponent for~\eqref{aasass}. %\todo{we restructure in draft1, where this theorem is stated right after lemma 6.6 and uses the previous computations,\textcolor{red}{thanks}} return to the earlier motivation in this section, which led us to adapt the results of \cite{LPS23}. Recall that we aimed to argue that obtaining negative Lyapunov exponents for the type of SDE investigated in this manuscript is achievable when the noise intensity is sufficiently large.
%\todo{this is not a statement of the theorem,\textcolor{red}{true, so it is reformulated}}
\begin{theorem}\label{thm:main}
Let Assumptions~\ref{Drift2} and~\ref{theth} hold.  
Then for every \( \theta \geq 1 \), there exists a constant \( \kappa_{\theta} > 0 \) such that for every \( \sigma \in T_{\theta, \kappa_{\theta}} \), the top Lyapunov exponent \( \lambda^{\sigma}_1 \) of~\eqref{aasass} satisfies
\[
\lambda^{\sigma}_1 < 0.
\]
\end{theorem}
%\todo{1-why do we need to state this inclusion here. 2-relate it to the constant $M_\theta$ from the statement of the theorem.3- bounded set of what .4-refer to the bound or recall it here \textcolor{red}{Based on these four comments, the proof has been completely rewritten}
%}

%\textcolor{red}{.}
\proof
From Lemma \ref{LYASIGMA} we have
\begin{align}
\lambda_1^{\sigma} \leq -C_{4}^{F} \pi^{\sigma}\left(\mathbb{R}^d \setminus B(0,R)\right) + C \pi^{\sigma}\left(B(0,R)\right).
\end{align}
As already stated, we have to prove that increasing \(\sigma\) ensures that \(\pi^{\sigma}\big(B(0, R)\big)\) becomes small. From Lemma \ref{contrere}
\begin{align*}
\pi^{\sigma}\left(B(0,R)\right) = \tilde{\pi}^{\sigma}\left((\|\sigma\|^{-1} B(0,R)\right).
\end{align*}
Since \(\tilde{\pi}^{\sigma}\) has the density function \(\tilde{p}^{\sigma}_{\infty}\), we conclude that 
\begin{align}\label{OLLAs63a}
\pi^{\sigma}\big(B(0, R)\big) = \int_{\|\sigma\|^{-1} B(0,R)} \tilde{p}^{\sigma}_{\infty}(y)\,\mathrm{d}y.
\end{align}
Obviously as \( \|\sigma\| \) increases, the rescaled set \( \|\sigma\|^{-1} B(0, R) \) remains bounded, while its Lebesgue measure decreases to zero.
Moreover, for $\kappa_1 > \kappa_2$ we have the inclusion %\todo{state where does this exactly apply,\textcolor{red}{I think we stated this in the next line, where we used this fact to conclude the results. I added a line anyway}}
\begin{align*}\label{KLassd}
T_{\theta, \kappa_1} \subset T_{\theta, \kappa_2}.
\end{align*}
Thus, when $\theta \geq 1$ is fixed, increasing $\kappa$ makes $T_{\theta,\kappa}$  become smaller. Consequently, Proposition~\ref{INFORM} implies that there exist constants 
\( C_1, C_2 > 0 \), depending on \( R \), \( \kappa \), and \( \theta \), 
such that for all \( \sigma \in T_{\theta,\kappa} \) with \( \kappa \geq 1 \), we have
\[
\tilde{p}^{\sigma}_\infty(y) \leq C_1 \exp(-C_2 |y|^2) \quad \text{for all } y \in B(0, R).
\]
Therefore, by choosing \( \kappa_{\theta} \geq 1 \) sufficiently large, it follows from~\eqref{OLLAs63a} that \( \pi^{\sigma}\big( B(0, R) \big) \) can be made arbitrarily small for every \( \sigma \in T_{\theta,\kappa_{\theta}} \). Consequently we infer that
\[
\lambda_1^{\sigma} < 0.
\]
\qed
\appendix

\section{The decomposition of fractional Brownian motion}
We recall some properties of the decomposition of fractional Brownian motion as stated in Lemma~\ref{SDASASASAa}.

\begin{lemma}\label{Liouville}
For every \( T > 0 \) and \( \gamma \leq \frac{H}{2} \), we define the operator
\[
\mathcal{P}: \mathcal{B}_H \longrightarrow C^{\gamma}([0, T], \mathbb{R}^d)
\]
as the continuous extension of the mapping
\begin{align*}%\label{FF14}
\mathcal{P}(\omega)(t) := \frac{1}{\alpha_H} \int_{-\infty}^{0} \left[ (t - r)^{H - \frac{1}{2}} - (-r)^{H - \frac{1}{2}} \right] \dot{\omega}(r) \, \mathrm{d}r, \quad \text{for all } t \in [0, T],
\end{align*}
from \( C_0^\infty(\mathbb{R}^-, \mathbb{R}^d) \) to \( \mathcal{B}_H \). In particular, the operator \( \mathcal{P} \) is well-defined and continuous.
\end{lemma}
\proof
The proof closely follows the arguments in  \cite[Lemma~6.5]{DPT19}. We provide it here for the sake of completeness.
Let \( \omega \in C_0^\infty(\mathbb{R}^-, \mathbb{R}^d) \), and let \( t, s \in [0,T] \) with \( s < t \). Setting \( h := t - s \) we have for every \( r \leq 0 \) that
%\todo{in the inequality we need to multiply the first term with an $H$-dependent constant,\textcolor{red}{Thank you for your careful observation. I have filled the gap accordingly, and I sorry for the mistake.}}
\begin{align}\label{M120}
\left| (t - r)^{H - \frac{3}{2}} - (s - r)^{H - \frac{3}{2}} \right| 
&\leq \Big(\frac{3}{2}-H\Big) (s - r)^{H - \frac{5}{2}}h.
\end{align}
Moreover, since \( \omega(0) = 0 \),
\begin{align}\label{M121}
\lvert \omega(r) \rvert \leq \Vert \omega \Vert_{\mathcal{B}_H} \, 
\sqrt{1 + \lvert r \rvert} \, \lvert r \rvert^{\tfrac{1 - H}{2}} .
\end{align}
Furthermore, for \( s, t \in [0,T] \) we have
\begin{align}\label{M122}
\begin{split}
	& \int_{-\infty}^{-h} (s - r)^{H - \tfrac{5}{2}} 
	\sqrt{1 + \lvert r \rvert} \, \lvert r \rvert^{\tfrac{1 - H}{2}} \, \mathrm{d}r 
	\;\lesssim\; h^{\tfrac{H}{2} - 1}, \\[0.4em]
	& \Biggl\lvert \int_{-h}^{0} 
	\Bigl( (t - r)^{H - \tfrac{3}{2}} - (s - r)^{H - \tfrac{3}{2}} \Bigr) 
	\lvert r \rvert^{\tfrac{1 - H}{2}} \, \mathrm{d}r \Biggr\rvert 
	\;\lesssim\; h^{\tfrac{H}{2}} .
\end{split}
\end{align}
Therefore, by integration by parts and using \eqref{M120}--\eqref{M122}, we obtain
\begin{align*}%\label{IASLAsw}
\begin{split}
	& \bigl\lvert \mathcal{P}(\omega)(t) - \mathcal{P}(\omega)(s) \bigr\rvert  = \frac{1}{\alpha_H} \Biggl\lvert \int_{-\infty}^{0} 
	\Bigl( (t - r)^{H - \tfrac{1}{2}} - (s - r)^{H - \tfrac{1}{2}} \Bigr) \dot{\omega}(r) \, \mathrm{d}r \Biggr\rvert \\[0.3em]
	&= \frac{\lvert H - \tfrac{1}{2} \rvert}{\alpha_H} \Biggl\lvert \int_{-\infty}^{0} 
	\Bigl( (t - r)^{H - \tfrac{3}{2}} - (s - r)^{H - \tfrac{3}{2}} \Bigr) \omega(r) \, \mathrm{d}r \Biggr\rvert \\[0.3em]
	&\lesssim \Biggl\lvert \int_{-\infty}^{-h} \Bigl( (t - r)^{H - \tfrac{3}{2}} - (s - r)^{H - \tfrac{3}{2}} \Bigr) \omega(r) \, \mathrm{d}r \Biggr\rvert + \Biggl\lvert \int_{-h}^{0} \Bigl( (t - r)^{H - \tfrac{3}{2}} - (s - r)^{H - \tfrac{3}{2}} \Bigr) \omega(r) \, \mathrm{d}r \Biggr\rvert \\[0.3em]
	&\lesssim h \, \Vert \omega \Vert_{\mathcal{B}_H} 
	\int_{-\infty}^{-h} (s - r)^{H - \tfrac{5}{2}} \sqrt{1 + \lvert r \rvert} \, \lvert r \rvert^{\tfrac{1 - H}{2}} \, \mathrm{d}r  \\
	&\quad + \Vert \omega \Vert_{\mathcal{B}_H} \Biggl\lvert \int_{-h}^{0} \Bigl( (t - r)^{H - \tfrac{3}{2}} - (s - r)^{H - \tfrac{3}{2}} \Bigr) 
	\lvert r \rvert^{\tfrac{1 - H}{2}} \sqrt{1 + \lvert r \rvert} \, \mathrm{d}r \Biggr\rvert \\[0.3em]
	&\lesssim (t-s)^{\tfrac{H}{2}} \, \Vert \omega \Vert_{\mathcal{B}_H} .
\end{split}
\end{align*}
Thus, by the definition of the Banach space \( \mathcal{B}_{H} \) and a density argument, one concludes that the previous estimate holds for every \( \omega \in \mathcal{B}_{H} \). This completes the proof.
\begin{remark}\label{ASA||AQ}
This lemma is a pathwise version, i.e.~for every $\omega\in\cB_H$ of \cite[Lemma~6.5]{DPT19}, which provides less H\"older regularity in time for $\cP$, i.e.~$\gamma$ for $\gamma\leq H/2$.~This is natural, since we do not employ any probabilistic arguments which entail the $\gamma$-H\"older continuity of $\cP$ with $H<\gamma$
\end{remark}
\qed
Furthermore, we discuss some properties of the  Liouville operator. 
\begin{lemma}\label{Liouville1}
For every \( T > 0 \), one can define the operator
\begin{align*}
\tilde{B}^{H}: \mathcal{B}_H \longrightarrow C([0, T], \mathbb{R}^d)
\end{align*}
as the continuous extension of the mapping
\begin{align*}%\label{FF15}
\tilde{B}^{H}(\omega)(t)=\tilde{B}^{H}_t(\omega) := -\frac{1}{\alpha_{H}}\int_{-t}^{0} (r + t)^{H - \frac{1}{2}} \dot{\omega}(r) \, \mathrm{d}r, \quad \text{for all } t \in [0, T]
\end{align*}
from \( C_0^\infty(\mathbb{R}^-, \mathbb{R}^d) \) to \( \mathcal{B}_H \). %In particular, the operator \( \tilde{B}^{H} \) is well-defined and continuous. 
Moreover, there exists a constant \( C_{T,H} > 0 \) such that
\begin{align}
\left| \tilde{B}^{H}_t(\omega) \right| 
\leq C_{T,H}\, t^{\tfrac{H}{2}} \, \| \omega \|_{\mathcal{B}_H}.
\end{align}
\end{lemma}
\proof
Let \( \omega \in C_0^\infty(\mathbb{R}^-, \mathbb{R}^d) \). Then we have
\begin{align*}
- \int_{-t}^{0} (r+t)^{H-\frac{1}{2}} \dot{\omega}(r) \, \mathrm{d}r
= -\omega(-t) \, t^{H-\frac{1}{2}}
+ \left(H - \tfrac{1}{2}\right) \int_{-t}^{0} (r+t)^{H - \frac{3}{2}} \bigl( \omega(r) - \omega(-t) \bigr) \, \mathrm{d}r.
\end{align*}
This identity follows by integration by parts.
Therefore using the definition of the norm of \( \omega \) in \( \mathcal{B}_{H} \), we obtain the existence of a constant \( C_{T,H} \) such that
\begin{align*}
\frac{1}{\alpha_{H}}\left\vert -\int_{-t}^{0} (r+t)^{H-\frac{1}{2}} \dot{\omega}(r) \, \mathrm{d}r \right\vert
\leq C_{T,H} t^{\frac{H}{2}} \| \omega \|_{\mathcal{B}_H},
\end{align*}
as stated. 
\qed
\begin{remark}
One can also define
\[
\tilde{B}_{t}^{H}(\omega) = - \frac{1}{\alpha_{H}}\int_{-t}^{0} (r + t)^{H - \frac{1}{2}} \, \mathrm{d}\omega(r)
\]
as an It\^o integral.~However, as we already indicated in Remark \ref{ASA||AQ}, it suffices to define it pathwise for every \( \omega\in\mathcal{B}_{H} \).% without relying on any probabilistic tools.
\end{remark}
%Thanks to Lemma~\ref{Liouville} and Lemma~\ref{Liouville1},

\section{Bounds for the stationary density of a rescaled SDE with fractional Brownian motion}\label{AAZAZA}
For the convenience of the reader, we collect the intermediate steps required in the proofs of Proposition~\ref{INFORM} and Theorem~\ref{thm:main}.
\begin{lemma}\label{UKALSsdd}
Let $(M_t)_{t \in[0,t_0]}$ be a local martingale such that 
\begin{align}\label{OASLA;s}
\mathbb{E}\left[\exp\left(8 \langle M \rangle_{t_0} \right)\right] < \infty.
\end{align}
Then \((M_t)_{t\in[0,t_0]}\) is a martingale and the following inequality holds
\begin{align}\label{AISs85a}
\mathbb{E}\left[\exp\left(2M_{t_0}-\langle  M_{t_0}\rangle\right)\right]
\leq \sqrt{ \mathbb{E} \left[ \exp\left(6 \langle M \rangle_{t_0} \right) \right] }.
\end{align}
\end{lemma}
\proof
From H\"older's inequality, we have
\begin{align}\label{UKAssd}
\begin{split}
	\mathbb{E}\left[\exp\left(2M_{t_0} - \langle M \rangle_{t_0}\right)\right]
	&= \mathbb{E}\left[\exp\left(2M_{t_0} - 4 \langle M \rangle_{t_0}\right) \exp\left(3 \langle M \rangle_{t_0}\right)\right] \\
	&\leq \sqrt{\mathbb{E}\left[\exp\left(4M_{t_0} - 8 \langle M \rangle_{t_0}\right)\right]} \sqrt{\mathbb{E}\left[\exp\left(6 \langle M \rangle_{t_0}\right)\right]}.
\end{split}
\end{align}
Since \(\frac{1}{2} \langle 4M \rangle_{t_0} = 8 \langle M \rangle_{t_0}\), it follows from \eqref{OASLA;s} and \cite[Proposition~5.12]{SHK91} that
\[
\mathbb{E}\left[\exp\left(4M_{t_0} - 8 \langle M \rangle_{t_0}\right)\right] = 1,
\]
which establishes the claim.
\qed
Let us consider the following local martingale
\[
M_t := \int_0^t \left\langle \mathcal{L}_s^{\sigma, l, z, t_0}, \, \mathrm{d}W_s \right\rangle,
\]
whose quadratic variation at time \( t_0 \) is given by
\[
\langle M \rangle_{t_0} = \int_0^{t_0} \left| \mathcal{L}_s^{\sigma, l, z, t_0} \right|^2 \, \mathrm{d}s.
\]
Choosing \( t_0 \) sufficiently small, then
\[
\mathbb{E}\!\left[\exp \bigl(8 \langle M \rangle_{t_0}\bigr)\right] < \infty.
\]
Keeping this in mind, we derive a bound for $G^\sigma_{t_0}(l;y)$.
\begin{lemma}\label{YAHSSSSQQ} 
Consider the situation of Proposition~\ref{law1}. For $y \in \mathbb{R}^d$, choose \( t_0 > 0 \) sufficiently small such that for $z = \|\sigma\| \sigma^{-1} (y - l(t_0))$,
we have
\begin{align}\label{AS3987as}
\mathbb{E}\!\left[\exp\!\left(8  \int_0^{t_0} \left| \mathcal{L}_s^{\sigma, l, z, t_0} \right|^2 \, \mathrm{d}s \right)\right] < \infty.
\end{align}
Then
\begin{align}\label{LOP56as1}
\begin{split}
	G_{t_0}^{\sigma}(l;y) &\leq \left(\mathbb{E}\left[
	\exp\left(2\int_{0}^{t_0} \left\langle\mathcal{L}^{\sigma,l,z,t_0}_s,K_{s}(z,t_0,W)\right\rangle\mathrm{d}s\right)\right]\right)^{\frac{1}{2}}\times\\&\quad  \left(\mathbb{E}\left[\exp\left(6\int_{0}^{t_0} \left| \mathcal{L}^{\sigma,l,z,t_0}_s \right|^2\mathrm{d}s\right)\right]\right)^{\frac{1}{4}}
\end{split}
\end{align}

%  \end{comment}
\end{lemma}
\proof
The semimartingale \(\big(X_t^{z,t_0}\big)_{t\in[0,t_0]}\) given in \eqref{SEMIMAR} can be decomposed as 
\begin{align}\label{XYZAa}
X_t^{z,t_0} = \int_0^t K_s(z, t_0, W) \, \mathrm{d}s + W_t.
\end{align}
Consequently,
\begin{align}\label{AIS85a}
\begin{split}
&\exp\left(
\int_0^{t_0} \left\langle \mathcal{L}^{\sigma,l,z,t_0}_s, \, \mathrm{d}X_s^{z,t_0} \right\rangle
- \frac{1}{2} \int_0^{t_0} \left| \mathcal{L}^{\sigma,l,z,t_0}_s \right|^2 \, \mathrm{d}s
\right)\\&=\exp\left(\int_{0}^{t_0} \left\langle\mathcal{L}^{\sigma,l,z,t_0}_s,K_{s}(z,t_0,W)\right\rangle\mathrm{d}s\right)\exp\left(\int_{0}^{t_0}\left\langle \mathcal{L}^{\sigma,l,z,t_0}_s, \, \mathrm{d}W_s \right\rangle
- \frac{1}{2}\int_0^{t_0} \left| \mathcal{L}^{\sigma,l,z,t_0}_s \right|^2 \, \mathrm{d}s\right).
\end{split}
\end{align}
%where $\int_0^{t_0} |\mathcal{L}^{\sigma,l,z,t_0}_s|^2~\txtd s$ is the quadratic variation of the local martingale  \( M_{t_0} \) given by
%\[
%	M_{t_0} = \int_0^{t_0} \left\langle \mathcal{L}_s^{\sigma, l, z, t_0}, \, \mathrm{d}W_s \right\rangle.
%	\]
Using Lemma~\ref{UKALSsdd} and H\"older's inequality, we have
\begin{align}\label{sjdjsa}
\begin{split}
G_{t_0}^{\sigma}(l;y) &= \mathbb{E}\left[
\exp\left(\int_{0}^{t_0} \left\langle\mathcal{L}^{\sigma,l,z,t_0}_s,K_{s}(z,t_0,W)\right\rangle\mathrm{d}s\right)\exp\left(\left(M_{t_0}-\frac{1}{2}\langle M_{t_0}\rangle\right)\right)
\right] \\
&\leq \left(\mathbb{E}\left[
\exp\left(2\int_{0}^{t_0} \left\langle\mathcal{L}^{\sigma,l,z,t_0}_s,K_{s}(z,t_0,W)\right\rangle\mathrm{d}s\right)\right]\right)^{\frac{1}{2}}\times\\&\quad  \left(\mathbb{E}\left[\exp\left(6\int_{0}^{t_0} \left| \mathcal{L}^{\sigma,l,z,t_0}_s \right|^2\mathrm{d}s\right)\right]\right)^{\frac{1}{4}}
\end{split}
\end{align}
which proves the claim.
\qed
\begin{remark}\label{RAEDASAs}
As indicated in \eqref{sjdjsa}, obtaining an upper bound for \( G_{t_0}^{\sigma}(l; y) \) requires a careful analysis of the process \( \mathcal{L}^{\sigma,l,z,t_0} \). Provided that \( \| \mathcal{L}^{\sigma,l,z,t_0} \|_{\infty,[0,t_0]} < \infty \), inequality~\eqref{sjdjsa} yields the estimate
\begin{align}\label{LOP56as}
\begin{split}
G_{t_0}^{\sigma}(l; y) &\leq \left( \mathbb{E}\left[
\exp\left( 2 \left\| \mathcal{L}^{\sigma,l,z,t_0} \right\|_{\infty,[0,t_0]} \int_{0}^{t_0} \left| K_s(z, t_0, W) \right| \, \mathrm{d}s \right) \right] \right)^{1/2} \\
&\quad \times \left( \mathbb{E}\left[ \exp\left( 6 t_0 \left\| \mathcal{L}^{\sigma,l,z,t_0} \right\|^2_{\infty,[0,t_0]} \right) \right] \right)^{1/4}.
\end{split}
\end{align}
For \( H \in (0, \tfrac{1}{2}) \) one can establish an upper bound for \( \| \mathcal{L}^{\sigma,l,z,t_0} \|_{\infty,[0,t_0]} \), see~Lemma~\ref{IOO1}. However, when \( H \in (\tfrac{1}{2}, 1) \), obtaining an upper bound for \( G_{t_0}^{\sigma}(l; y) \) requires a slightly different approach since \( \mathcal{L}^{\sigma,l,z,t_0} \) has a singularity in zero, as we will show in Lemma~\ref{IOO1}.
\end{remark}
We first focus on the semimartingale \( (X_t^{z,t_0})_{t \in [0, t_0]} \).
\begin{lemma}\label{IOO}
For \( \gamma \in (0, \tfrac{1}{2}) \) and \( t_0 \in (0,1] \), there exists a constant \( C_H > 0 \) such that
\begin{align}\label{sudisl5s}
\begin{split}
&	\| X^{z,t_0} \|_{\gamma; [0, t_0]} \leq C_{H}\left(t_{0}^{\frac{1}{2}-H-\gamma}\vert z\vert+t_{0}^{\frac{1}{2}-\gamma}\Upsilon(t_0,W)\right)+\Vert W\Vert_{\gamma; [0, t_0]},\\
& \int_{0}^{t_{0}}\left|K_{s}(z, t_0, W)\right|\mathrm{d}s\leq C_{H}\left(t_{0}^{\frac{1}{2}-H}\vert z\vert+t_{0}^{\frac{1}{2}} \Upsilon(t_0,W)\right),\\
& \int_{0}^{t_{0}} s^{\frac{1}{2}-H}\left|K_{s}(z, t_0, W)\right|\mathrm{d}s\leq C_H \left(t_{0}^{1-2H}\vert z\vert+t_{0}^{1-H} \Upsilon(t_0,W)\right),
\end{split}
\end{align}
where \( \Upsilon(t_0,W) \) is a positive random variable satisfying
\[
\mathbb{E}\ \left[\exp\ \bigl( r (\Upsilon(t_0,W))^2 \bigr)\right] < \infty
\]
for all sufficiently small \( r > 0 \).
\end{lemma}
\proof
From \eqref{SEMIMAR} we have for $s\in[0,t_0]$
\begin{align}\label{BB0}
\begin{split}
X_s^{z,t_0}&=\int_{0}^{s}K_{u}(z, t_0, W)\mathrm{d}u+W_s=\frac{2H}{t_{0}^{2H}\rho_H(H+\frac{1}{2})}\left(t_{0}^{H+\frac{1}{2}}-(t_{0}-s)^{H+\frac{1}{2}}\right)z\\&-2H\int_{0}^{s}(t_{0}-u)^{H-\frac{1}{2}}\left(\int_{0}^{u}(t_0-\tau)^{-H-\frac{1}{2}}\mathrm{d}W_\tau\right)\mathrm{d}u+W_s=:T_{s}(z,t_0)+U_{s}(t_0,W)+W_s.
\end{split}
\end{align}
Obviously, there exists a constant $C_H>0$ such that
\begin{equation} \label{BB1}
\|T(z,t_0)\|_{\gamma;[0,t_0]} \leq C_{H} t_{0}^{\frac{1}{2} - H - \gamma} |z|.
\end{equation}
We now estimate the $\gamma$-H\"older norm of $U$. To this aim we have by It\^o's isometry that
\[
\mathbb{E}\left[\left\vert\frac{\int_{0}^{u} (t_0 - \tau)^{-H - \frac{1}{2}} \, \mathrm{d}W_\tau}{(t_0 - u)^{-H}} \right\vert^2\right]
= \frac{1}{2H} \left( 1 - \left(\frac{t_0 - u}{t_0}\right)^{2H} \right)
\leq \frac{1}{2H}.
\]
Therefore we can assume that
\[
\sup_{u \in [0, t_0]} \frac{\left| \int_{0}^{u} (t_0 - \tau)^{-H - \frac{1}{2}} \, \mathrm{d}W_\tau \right|}{(t_0 - u)^{-H}} \leq \Upsilon(t_0,W),
\]
where \( \Upsilon(t_0,W) \) is a random variable satisfying
\[
\mathbb{E}\left[ \exp\left( r \left( \Upsilon(t_0,W) \right)^2 \right) \right] < \infty
\]
for all sufficiently small \( r > 0 \).
This follows from the Gaussianity of the stochastic integral and Lemma~\ref{FerniqueAA}.
Thus, for every \( 0 < \gamma < \frac{1}{2}\), we obtain that 
\begin{align}\label{BB2}
\Vert U(t_0,W) \Vert_{\gamma;[0,t_0]}\leq C_H t_{0}^{\frac{1}{2}-\gamma}\Upsilon(t_0,W).
\end{align}
Putting these estimates together, we obtain the first bound in \eqref{sudisl5s}.  
Furthermore,
\begin{align*}
\int_{0}^{t_{0}}\left|K_{s}(z, t_0, W)\right|\mathrm{d}s&=\int_{0}^{t_{0}}\left\vert \frac{2H}{\rho_H}(t_0 - s)^{H - \frac{1}{2}} \left( \frac{z}{t_0^{2H}} - \rho_H \int_0^s (t_0 - u)^{-H - \frac{1}{2}} \, \mathrm{d}W_u \right)  \right\vert\mathrm{d}s\\&\leq \frac{2H t_{0}^{\frac{1}{2}-H}}{\rho_H (H+\frac{1}{2})}\vert z\vert+4Ht_{0}^{\frac{1}{2}} \Upsilon(t_0,W),
\end{align*}
which proves the second inequality in \eqref{sudisl5s}. The third inequality in~\eqref{sudisl5s} can be similarly obtained. Indeed, we have
\begin{align*}
\int_0^{t_0} s^{\frac{1}{2}-H} |K_s(z,t_0,W)|~\txtd s &\leq \frac{2H |z|}{\rho_Ht^{2H}_0}  \int_0^{t_0} s^{\frac{1}{2}-H} (t_0-s)^{H-\frac{1}{2}}~\txtd s \\&+ 2H \Upsilon(t_0,W) \int_0^{t_0} s^{\frac{1}{2}-H}(t_0-s)^{-\frac{1}{2}}~\txtd s 
\\&\leq  C_H \left(t_{0}^{1-2H}\vert z\vert+t_{0}^{1-H} \Upsilon(t_0,W)\right). 
\end{align*}
\qed
According to Remark~\ref{RAEDASAs} we now focus on the term \(\mathcal{L}^{\sigma, l, z, t_0}\).
\begin{lemma}\label{IOO1} %\todo{Mazyar: The range of $\eta$ probably better to be change to $0<\eta<\min{H,\frac{1-H}{2}}$}
Consider the setting of Proposition~\ref{law1}, and assume that \( 0 < \eta < \min\!\left\{ H, \tfrac{1-H}{2} \right\} \). Then there exists a constant \( C_{\theta,\kappa} > 0 \) such that, for all \( t \in [0, t_0] \), the following estimate holds:
\begin{comment}
Let \( 0 < \eta < \min\left\{ H, \tfrac{1}{2} \right\} \), \( \sigma \in T_{\theta,\kappa} \), \( l \in C^{H-}_{\mathrm{loc}}(\mathbb{R}_+; \mathbb{R}^d) \), \( y \in \mathbb{R}^d \), and define \( z = \|\sigma\| \sigma^{-1} (y - l(t_0)) \). Then, there exists a constant \( C_{\theta,\kappa} > 0 \) such that for all \( t \in [0, t_0] \), the following holds:
\end{comment}
\begin{itemize}
\item [1)] if $H\in(0,1/2)$:
\begin{align}\label{OP963a}
\left\Vert \mathcal{L}^{\sigma, l, z, t_0} \right\Vert_{\infty;[0,t_0]} 
\leq C_{\theta,\kappa}\, t_0^{\frac{1}{2} - H} \left( 1 + \left\Vert l \right\Vert_{\infty;[0,t_0]} + t_0^{\eta} \left\Vert X^{z,t_0} \right\Vert_{\frac{1}{2} - H + \eta;[0,t_0]} \right) ;
\end{align}
\item [2)]	if $H\in (1/2,1)$:	\begin{align}\label{OP963a1}
\begin{split}
	&\sup_{s\in [0,t_0]}\left\vert \mathcal{L}^{\sigma,l,z,t_0}_s-\frac{\frac{3}{2}-H}{\rho_H\Gamma(\frac{5}{2}-H)} \sigma^{-1} F(0) s^{\frac{1}{2}-H}\right\vert \\&\quad\leq C_{\theta,\kappa} \left[1+t_0^{\eta}\left(\left\Vert l\right\Vert_{H-\frac{1}{2}+\eta;[0,{t_0}]}+\Vert X^{{z,t_0}}\Vert_{\eta;[0,{t_0}]}\right)\right].
\end{split}
\end{align}
\end{itemize}

\begin{comment}

\begin{align*}
\left\Vert\mathcal{L}^{\sigma,l,z,t_0} \right\Vert_{\infty;[0,t_0]}\leq\begin{cases}
	C_\theta {t_0}^{\frac{1}{2}-H}\left(1+\Vert l\Vert_{\infty;[0,t_0]}+t_0^{\eta}\Vert X^{{z,t_0}} \Vert_{\frac{1}{2}-H+\eta;[0,{t_0}]}\right) \ \ \ \ \ \,&0<H<\frac{1}{2},\\
	C_\theta \left(t_{0}^{H-\frac{1}{2}}+t_0^{H-\frac{1}{2}+\eta}\left(\left\Vert l\right\Vert_{H-\frac{1}{2}+\eta;[0,{t_0}]}+\Vert X^{{z,t_0}}\Vert_{\eta;[0,{t_0}]}\right)\right) \ \ \ \ \ \,&\frac{1}{2}<H<1. 
\end{cases}
\end{align*}
and for the case that $H>\frac{1}{2}$
\begin{align*}
\sup_{s\in [0,t_0]}\left\vert \mathcal{L}^{\sigma,l,z,t_0}_s-\frac{\frac{3}{2}-H}{\rho_H\Gamma(\frac{5}{2}-H)} \sigma^{-1} F(0) s^{\frac{1}{2}-H}\right\vert \leq C_\theta \left(t_0^{H-\frac{1}{2}+\eta}\left(\left\Vert l\right\Vert_{H-\frac{1}{2}+\eta;[0,{t_0}]}+\Vert X^{{z,t_0}}\Vert_{\eta;[0,{t_0}]}\right)\right)
\end{align*}
\end{comment}
\end{lemma}
\proof
First, note that \( \sigma \in T_{\theta,\kappa} \), which implies
\begin{align}\label{85asas}
\|\sigma\| \, \|\sigma^{-1}\| \leq \theta
\quad \text{and} \quad 
\|\sigma\| \geq \kappa.
\end{align}
We now focus on the case where \( H \in \left(0, \tfrac{1}{2}\right) \). From  Lemma~\ref{AS85sdf}, the fact that \( X^{z,t_0}_{0} = 0 \), together with Lemma~\ref{Straightforward} and Lemma~\ref{ASUj144as0}, we obtain
\begin{align*}
&\left\Vert\mathcal{L}^{\sigma,l,z,t_0} \right\Vert_{\infty;[0,t_0]}\lesssim
t_0^{\frac{1}{2}-H}\left\Vert F^{\sigma}\left(l +\frac{\rho_H}{\|\sigma\|} \sigma  \left(\mathcal{J}_{+}^{ H-\frac{1}{2}} X^{z,t_0}\right)\right)\right\Vert_{\infty;[0,t_0]}\\&\lesssim t_{0}^{\frac{1}{2}-H}\left(1+\Vert l\Vert_{\infty;[0,t_0]}+\Vert \mathcal{J}_{+}^{ H-\frac{1}{2}} X^{z,t_0}\Vert_{\infty;[0,{t_0}]}\right)\\&\lesssim t_{0}^{\frac{1}{2}-H}\left(1+\Vert l\Vert_{\infty;[0,t_0]}+t_0^{\eta}\Vert X^{{z,t_0}} \Vert_{\frac{1}{2}-H+\eta;[0,{t_0}]}\right).
\end{align*}
We now consider the case $H\in(1/2,1)$ which requires a more careful analysis. Given that
\begin{align*}
l(0) = \mathcal{J}_{+}^{ H - \frac{1}{2}} X^{z,t_0}_0 = 0 
\quad \text{and} \quad 
F^{\sigma}(0) = \frac{F(0)}{\left\Vert \sigma \right\Vert},
\end{align*}
and by applying \eqref{ASUj144as} with \( \alpha = H - \frac{1}{2} \), we obtain
\begin{align}\label{ILOAs369a}
\begin{split}
\mathcal{L}^{\sigma,l,z,t_0}_s 
&= \frac{\frac{3}{2} - H}{\rho_H \Gamma\left(\frac{5}{2} - H\right)} \sigma^{-1} F(0) s^{\frac{1}{2} - H} \\
&\quad + \frac{\left\Vert \sigma \right\Vert \sigma^{-1}}{\rho_H}
\left( \tilde{\mathcal{J}}_{+}^{\frac{1}{2} - H} 
F^{\sigma} \left( l + \frac{\rho_H}{\left\Vert \sigma \right\Vert} \sigma \mathcal{J}_{+}^{ H - \frac{1}{2}} X^{z,t_0} \right) \right)(s),
\quad s \in [0,t_0].
\end{split}
\end{align}
Thanks to \eqref{85asas} and Lemma~\ref{ASUj144as0}, it follows that
\begin{align}\label{OL56as}
\begin{split}
&\left\Vert\frac{\left\Vert \sigma \right\Vert \sigma^{-1}}{\rho_H}
\left( \tilde{\mathcal{J}}_{+}^{\frac{1}{2} - H} 
F^{\sigma} \left( l + \frac{\rho_H}{\left\Vert \sigma \right\Vert} \sigma \mathcal{J}_{+}^{ H - \frac{1}{2}} X^{z,t_0} \right) \right)\right\Vert_{\infty;[0,t_0]}\\&\myquad[1]\lesssim t_{0}^\eta\left\Vert F^{\sigma}\left(l + \frac{\rho_H}{\|\sigma\|} \sigma \mathcal{J}_{+}^{ H-\frac{1}{2}} X^{z,t_0}\right)\right\Vert_{C^{0,H-\frac{1}{2}+\eta}([0,{t_0}])}\\ &\myquad[2]\lesssim 1+t_{0}^\eta\left\Vert l\right\Vert_{C^{0,H-\frac{1}{2}+\eta}([0,{t_0}])}+t_{0}^\eta\left\Vert \mathcal{J}_{+}^{ H-\frac{1}{2}} X^{z,t_0}\right\Vert_{C^{0,H-\frac{1}{2}+\eta}([0,{t_0}])} .
\end{split}
\end{align}
\begin{comment}
content...

\begin{align}\label{OL56as}
\begin{split}
	\left\Vert\mathcal{L}^{\sigma,l,z,t_0} \right\Vert_{\infty;[0,t_0]}&\lesssim
	t_{0}^{H-\frac{1}{2}}+\left\Vert F^{\sigma}\left(l + \frac{\rho_H}{\|\sigma\|} \sigma \mathcal{J}_{+}^{ H-\frac{1}{2}} X^{z,t_0}\right)\right\Vert_{C^{0,H-\frac{1}{2}+\eta}([0,{t_0}])}\\
	&\lesssim t_{0}^{H-\frac{1}{2}}+1+\left\Vert l\right\Vert_{C^{0,H-\frac{1}{2}+\eta}([0,{t_0}])}+\left\Vert \mathcal{J}_{+}^{ H-\frac{1}{2}} X^{z,t_0}\right\Vert_{C^{0,H-\frac{1}{2}+\eta}([0,{t_0}])} .
\end{split}
%\\	
%	&+t_0^{\eta+H-\frac{1}{2}}\left( 1+\left\Vert F^{\sigma}\left(l + \frac{\rho_H}{\|\sigma\|} \sigma \mathcal{J}_{+}^{ H-\frac{1}{2}} X^{z,t_0}\right)\right\Vert_{H-\frac{1}{2}+\eta;[0,{t_0}]}\right) 
\end{align}
\end{comment}
Recall that \( l(0) = X^{z,t_0}_{0} = 0 \), so Corollary~\ref{IL9as} yields
\begin{align*}
\left\Vert l\right\Vert_{C^{0,H-\frac{1}{2}+\eta}([0,{t_0}])}&\lesssim \Vert l\Vert_{H-\frac{1}{2}+\eta;[0,t_0]},\\
\left\Vert \mathcal{J}_{+}^{ H-\frac{1}{2}} X^{z,t_0}\right\Vert_{C^{0,H-\frac{1}{2}+\eta}([0,{t_0}])} &=\Vert \mathcal{J}_{+}^{ H-\frac{1}{2}}X^{z,t_0}\Vert_{\infty;[0,t_0]}+ \|\mathcal{J}_{+}^{ H-\frac{1}{2}}X^{z,t_0}\|_{H-\frac{1}{2}+\eta;[0,t_0]} \\& \lesssim \Vert X^{z,t_0}\Vert_{\eta;[0,t_0]}.
\end{align*}
\begin{comment}

Also,  \textcolor{blue}{A $\mathcal{J}^{H-1/2}_+$ improves the H\"older $\eta$-H\"older regularity of $X$ by $H-1/2$ we should get 
\begin{align*}
	\left\Vert \mathcal{J}_{+}^{ H-\frac{1}{2}} X^{z,t_0}\right\Vert_{C^{0,H-\frac{1}{2}+\eta}([0,{t_0}])} \leq \|X^{z,t_0}\|_{C^{0,\eta}([0,t_0])} \leq t^{\eta}_0 \|X^{z,t_0}\|_{\eta;[0,t_0]}.
\end{align*}
}
\end{comment}
Therefore,
\begin{align}\label{56IIOPAs}
\begin{split}
&\left\Vert\frac{\left\Vert \sigma \right\Vert \sigma^{-1}}{\rho_H}
\left( \tilde{\mathcal{J}}_{+}^{\frac{1}{2} - H} 
F^{\sigma} \left( l + \frac{\rho_H}{\left\Vert \sigma \right\Vert} \sigma \mathcal{J}_{+}^{ H - \frac{1}{2}} X^{z,t_0} \right) \right)\right\Vert_{\infty;[0,t_0]}\\&\myquad[1]\lesssim 1+t_0^{\eta}\left(\left\Vert l\right\Vert_{H-\frac{1}{2}+\eta;[0,{t_0}]}+\Vert X^{{z,t_0}}\Vert_{\eta;[0,{t_0}]}\right),
\end{split}
\end{align}
which proves \eqref{OP963a1}. %\todo{by C.4 we get $\|X^{z,t_0}\|_\eta$ and not $t^{H-1/2+\eta}_0 \|X^{z,t_0}\|_\eta $,\textcolor{red}{Thanks, I think what written is true since $\Vert\Vert_{C^{0,\eta}}$ is stronger than the Holder norm } this is true and was applied correctly for $l$, but we get here $\Vert X\Vert_{C^{0,\eta}}$  we get $t^\eta_0$,\textcolor{red}{Thanks, and sorry for this overlook. i will correct it}  }
\begin{comment}
\begin{align*}
\left\Vert\mathcal{L}^{\sigma,l,z,t_0} \right\Vert_{\infty;[0,t_0]}&\lesssim t_0^{H-\frac{1}{2}+\eta}\left( 1+\left\Vert F^{\sigma}\left(l + \frac{\rho_H}{\|\sigma\|} \sigma \mathcal{J}_{+}^{ H-\frac{1}{2}} X^{z,t_0}\right)\right\Vert_{H-\frac{1}{2}+\eta;[0,{t_0}]}\right) \\
&\lesssim t_0^{H-\frac{1}{2}+\eta}\left(1+\left\Vert l\right\Vert_{H-\frac{1}{2}+\eta;[0,{t_0}]}+\Vert \mathcal{J}_{+}^{ H-\frac{1}{2}} X^{z,t_0}\Vert_{H-\frac{1}{2}+\eta;[0,{t_0}]}\right)\\&\lesssim t_0^{H-\frac{1}{2}+\eta}\left(1+\left\Vert l\right\Vert_{H-\frac{1}{2}+\eta;[0,{t_0}]}+\Vert X^{{z,t_0}}\Vert_{\eta;[0,{t_0}]}\right)
\end{align*}
\end{comment}
\qed
This leads to the following result.
\begin{corollary}\label{YHAS96assd}
Let us consider the setting of Lemma~\ref{IOO1}. Then there exist positive constants 
\( C_{H,\theta,\kappa}^1 \) and \( C_{H,\theta,\kappa}^2 \) such that:	
\begin{itemize}
\item [1)] For \( H \in (0, \frac{1}{2}) \):
\begin{align*}
\left\Vert\mathcal{L}^{\sigma,l,z,t_0} \right\Vert_{\infty;[0,t_0]}
&\leq  C_{H,\theta,\kappa}^1 {t_0}^{\frac{1}{2}-H}\big(1+\Vert l\Vert_{\infty;[0,t_0]}+t^{\eta}_0\Vert W \Vert_{\frac{1}{2}-H+\eta;[0,{t_0}]}\big)\\&+ C_{H,\theta,\kappa}^2\left(t_{0}^{\frac{1}{2}-H}\vert z\vert+t_{0}^{\frac{1}{2}} \Upsilon(t_0,W)\right).
\end{align*}
\item [2)] For \( H \in (\frac{1}{2}, 1) \):
\begin{align}\label{OL52as}
\begin{split}
	&\sup_{s\in [0,t_0]}\left\vert \mathcal{L}^{\sigma,l,z,t_0}_s-\frac{\frac{3}{2}-H}{\rho_H\Gamma(\frac{5}{2}-H)} \sigma^{-1} F(0) s^{\frac{1}{2}-H}\right\vert\\&\quad\leq C_{H,\theta,\kappa}^1 \left[1+{t_0}^{\eta}\left(\left\Vert l\right\Vert_{H-\frac{1}{2}+\eta;[0,{t_0}]}+\Vert W\Vert_{\eta;[0,{t_0}]}\right)\right]+C_{H,\theta,\kappa}^2\left(t_{0}^{\frac{1}{2}-H}\vert z\vert+t_{0}^{\frac{1}{2}}\Upsilon(t_0,W)\right).
\end{split}
\end{align}
\end{itemize}
\end{corollary}
\proof
The proof follows from Lemma~\ref{IOO} and Lemma~\ref{IOO1} by substituting  into \eqref{sudisl5s} 
\( \gamma = \tfrac{1}{2} - H + \eta \) for \( 0 < H < \tfrac{1}{2} \), 
and \( \gamma = \eta \) when \( \tfrac{1}{2} < H < 1 \).
\qed
\begin{remark}
Note that there is a cancellation in \eqref{OL52as}, which allows us to eliminate the singularity of \( \mathcal{L}^{\sigma,l,z,t_0}_s \) in zero by subtracting \( \frac{\frac{3}{2} - H}{\rho_H \Gamma\left(\frac{5}{2} - H\right)} \sigma^{-1} F(0)s^{\frac{1}{2}-H} \).
\end{remark}
We can now  estimate  \( G_{t_0}^{\sigma}(l; y) \) as follows.
\begin{proposition}\label{YHHNA7as}
\begin{comment}
Let \( \sigma \in T_{\theta,\kappa} \), $t_0\in(0,1]$, \( l \in C^{H-}_{\mathrm{loc}}(\mathbb{R}^+, \mathbb{R}^d) \), \( y \in \mathbb{R}^d \),
$z = \|\sigma\| \, \sigma^{-1}(y - l(t_0))$ and $0<\eta<\min\{H,\frac{1}{2}\}$. 
We further define
\[
Q(x) := \max\left\{ x^{1/2},\, x^{1/4} \right\}, \quad \text{for } x > 0.
\]
\end{comment}
Let us consider the setting of Lemma~\ref{IOO1} and 
\[
Q(a) := \max\left\{ a^{1/2},\, a^{1/4} \right\}, \quad \text{for } a > 0.
\]
Then there exist constants \(\alpha_1, \alpha_2 \in (0,1)\) depending on $\eta$ and $H$ and \(K^{1}_{H,\theta,\kappa},\, K^{2}_{H,\theta,\kappa},\, K^{3}_{H,\theta,\kappa} > 0\) such that for sufficiently small \(t_0 > 0\), the following statements hold:
\begin{itemize}
\item [1)] For \(H \in \left(0, \frac{1}{2}\right)\):	
\begin{align}\label{UN100}
\begin{split}
	G_{t_0}^{\sigma}(l;y)&\leq K^{2}_{H,\theta,\kappa}\exp\left(K^{1}_{H,\theta,\kappa}t_{0}^{\alpha_1}\left(\Vert l\Vert_{\infty;[0,t_0]}^2+\vert z\vert^2\right)\right)\times\Gamma_{1}(t_0,\theta),
\end{split}
\end{align}
where
\begin{align}\label{85aw12aa}
\Gamma_{1}(t_0,\theta):=Q\left(\mathbb{E}\left[\exp\left(K^{1}_{H,\theta,\kappa}t_{0}^{\alpha_2}\left(\Vert W \Vert_{\frac{1}{2}-H+\eta;[0,{t_0}]}^2+(\Upsilon(t_0,W))^2\right)\right)\right]\right).
\end{align}
\item [2)] For \(H \in \left(\frac{1}{2}, 1\right)\):	
\begin{align}\label{UN101}
\begin{split}
	&G_{t_0}^{\sigma}(l;y)\\&  \leq K^{2}_{H,\theta,\kappa}\exp(K^{3}_{H,\theta,\kappa}t_{0}^{1-2H})\exp\left(K^{1}_{H,\theta,\kappa}\left(t_{0}^{\alpha_1}\left\Vert l\right\Vert_{H-\frac{1}{2}+\eta;[0,{t_0}]}^2+t_{0}^{1-2H}\vert z\vert^2\right)\right)\times \Gamma_{2}(t_0,\theta),
\end{split}
\end{align}
where
\begin{align}\label{85aw12aa3}
\Gamma_{2}(t_0,\theta):=Q\left(\mathbb{E}\left[\exp\left(K^{1}_{H,\theta,\kappa}t_{0}^{\alpha_2}\left(\Vert W \Vert_{\eta;[0,{t_0}]}^2+(\Upsilon(t_0,W))^2\right)\right)\right]\right).
\end{align}
\end{itemize}
\end{proposition}
\proof
First, observe that by Lemma~\ref{FerniqueAA} and Corollary~\ref{YHAS96assd}, it is possible to choose \( t_0 \) sufficiently small so that condition~\eqref{AS3987as} is satisfied. Therefore for \( 0 < H < \frac{1}{2} \), the claim follows directly from Remark \ref{RAEDASAs}, Lemma~\ref{IOO}, and Corollary~\ref{YHAS96assd}.
Now consider the case \( \frac{1}{2} < H < 1 \), which requires a more careful analysis. Thanks to Lemma~\ref{IOO} and Corollary~\ref{YHAS96assd}, we have
\begin{align}\label{F14}
\begin{split}
&\int_{0}^{t_0} \left| \mathcal{L}^{\sigma,l,z,t_0}_s \right|^2\mathrm{d}s\lesssim t_{0}^{2-2H}+t_0\sup_{s\in [0,t_0]}\left\vert \mathcal{L}^{\sigma,l,z,t_0}_s-\frac{\frac{3}{2}-H}{\rho_H\Gamma(\frac{5}{2}-H)} \sigma^{-1} F(0) s^{\frac{1}{2}-H}\right\vert^2\\
&\lesssim 1+{t_0}^{1+2\eta}\left(\left\Vert l\right\Vert_{H-\frac{1}{2}+\eta;[0,{t_0}]}+\Vert W\Vert_{\eta;[0,{t_0}]}\right)^2+t_{0}\left(t_{0}^{\frac{1}{2}-H}\vert z\vert+t_{0}^{\frac{1}{2}}\Upsilon(t_0,W)\right)^2.
\end{split}
\end{align}
Similarly, we obtain
\begin{align}\label{F15}
\begin{split}
&\int_{0}^{t_0} \left|\left\langle \mathcal{L}^{\sigma,l,z,t_0}_s, K_{s}(z,t_0,W) \right\rangle\right| \, \mathrm{d}s
\lesssim \int_{0}^{t_0} s^{\frac{1}{2}-H} \left| K_s(z,t_0,W) \right| \, \mathrm{d}s \\&\quad+ \sup_{s \in [0,t_0]} \left| \mathcal{L}^{\sigma,l,z,t_0}_s - \frac{\frac{3}{2}-H}{\rho_H \Gamma\left(\frac{5}{2}-H\right)} \sigma^{-1} F(0) s^{\frac{1}{2}-H} \right| 
\int_{0}^{t_0} \left| K_s(z,t_0,W) \right| \, \mathrm{d}s\\&\myquad[2]\lesssim t_{0}^{1-2H}\vert z\vert+t_{0}^{1-H} \Upsilon(t_0,W)\\&\myquad[3]+ \left(1+t_0^{\eta}\left(\left\Vert l\right\Vert_{H-\frac{1}{2}+\eta;[0,{t_0}]}+\Vert X^{{z,t_0}}\Vert_{\eta;[0,{t_0}]}\right)\right) \left(t_{0}^{\frac{1}{2}-H}\vert z\vert+t_{0}^{\frac{1}{2}} \Upsilon(t_0,W)\right).
\end{split}
\end{align}
We now estimate the terms on the right-hand side of~\eqref{F15}. %To do so, it suffices to use the following simple inequality:
%\[
%AB \leq \frac{1}{2} \left( A^2 + B^2 \right), \quad \text{for all } A, B \geq 0.
%\]
To do so, we use Young's inequality which implies that
\begin{align*}
t_0^{1 - 2H} |z| &\leq \frac{t_0^{1 - 2H} |z|^2 + t_0^{1 - 2H}}{2},
\end{align*}
and
\begin{align*}
&\left(1 + t_0^{\eta} \left( \left\Vert l \right\Vert_{H - \frac{1}{2} + \eta;\, [0, t_0]} + \left\Vert X^{z, t_0} \right\Vert_{\eta;\, [0, t_0]} \right) \right) t_0^{\frac{1}{2} - H} |z| \\
&\quad \leq \frac{ \left(1 + t_0^{\eta} \left( \left\Vert l \right\Vert_{H - \frac{1}{2} + \eta;\, [0, t_0]} + \left\Vert X^{z, t_0} \right\Vert_{\eta;\, [0, t_0]} \right) \right)^2 + t_0^{1 - 2H} |z|^2 }{2}.
\end{align*}
Similarly, the remaining terms can be estimated in the same way. Combining these estimates together with~\eqref{F14} and Lemma~\ref{YAHSSSSQQ}, yields~\eqref{UN101}. Finally, note that choosing \( t_0 \) sufficiently small, Lemma~\ref{FerniqueAA} and Lemma~\ref{IOO} ensure that \( \Gamma_{1}(t_0, \theta) \) and \( \Gamma_{2}(t_0, \theta) \) are finite.
\qed
\begin{remark}\label{Fernique}
Note that the terms
\begin{align*}
\exp(K^{3}_{H,\theta,\kappa} t_{0}^{1-2H}) \quad \text{and} \quad \exp\left(K^{1}_{H,\theta,\kappa} \, t_{0}^{1-2H} |z|^2\right)
\end{align*}
which appear in~\eqref{UN101} become large as \( t_0 \to 0 \) when \( H > \frac{1}{2} \). 
The first term is a constant depending only on \( t_0 \) and not on $z$. The second term requires a more careful analysis due to its quadratic dependence on \( |z| \). Fortunately, this term can be effectively controlled in the final estimate of the stationary density given in \eqref{UAJSMs}.
\end{remark}
%\todo{what we recall and the identities for $Z$ are not content of the lemma,\textcolor{red}{I wanted to emphasize the dependence on \(\omega^-\) and present a compact form to simplify calculations. This way, readers can grasp the rest of the argument by looking at this lemma without needing to revisit the previous definition.
%If this is still unclear or does not make sense, let me know how it can be improved.} we do not recall a paragraph of notations in a lemma before stating the result,\textcolor{red}{Now is reformulated base on your comment}}
Combining the previous results, we obtain the following representation of the density $\tilde{p}^\sigma_\infty$.  \begin{comment}
First of all, for $x,y\in \R^d$ we recall that
\[
l^{x, \sigma, \omega^-} = x + \frac{\sigma}{\|\sigma\|} \, \mathcal{P}(\omega^-)
\]
and set 
\begin{align}\label{ZSZZAsa}
\begin{split}
Z^{x,y,\sigma,\omega^-} 
&:= \|\sigma\| \, \sigma^{-1} \big( y - l^{x, \sigma, \omega^-}(t_0) \big) \\
&= \|\sigma\| \, \sigma^{-1} (y - x) - \mathcal{P}(\omega^-)(t_0).
\end{split}
\end{align}
\end{comment}
\begin{lemma}\label{SIMLIF}
Consider the setting of Proposition~\ref{REPREDDA}. Then, the density  \( \tilde{p}^{\sigma}_{\infty}(y) \) is given by
\begin{align}\label{NM12a}
\tilde{p}^{\sigma}_{\infty}(y) =  \mathcal{M}(t_0, H, \sigma) \int_{\mathbb{R}^d \times \mathcal{B}} \exp\left( -\frac{ \left| Z^{x,y,\sigma,\omega^-} \right|^2 }{ \rho_H^2 t_0^{2H} } \right) G_{t_0}^{\sigma}\left(l^{x, \sigma, \omega^-}; y\right) \, \tilde{\mu}^{\sigma}\left( \mathrm{d}(x, \omega^-) \right),
\end{align}
where
\begin{align*}%\label{CONZSADS}
\mathcal{M}(t_0, H, \sigma) := \frac{1}{\left( \sqrt{2\pi} \, \rho_H \, t_0^H \right)^d \, \det\left( \|\sigma\|^{-1} \sigma \right)},
\end{align*}
and
\begin{align*}
\begin{split}
Z^{x,y,\sigma,\omega^-} &:= \|\sigma\| \, \sigma^{-1} \left( y - l^{x, \sigma, \omega^-}(t_0) \right) \\
&= \|\sigma\| \, \sigma^{-1} (y - x) - \mathcal{P}(\omega^-)(t_0).
\end{split}
\end{align*}
Moreover, there exists a constant \( \mathcal{R}(t_0,\kappa,H) > 0 \) such that
\begin{align}\label{NM12a2}
\sup_{\sigma \in T_{\theta,\kappa}} \mathcal{M}(t_0, H, \sigma) \leq \mathcal{R}(t_0, H, \theta).
\end{align}
\end{lemma}
\proof
It follows from Proposition~\ref{law1} that
\begin{align}\label{OL6as}
\tilde{p}^{\sigma}_{t_0}(l^{x,\sigma,\omega^-}; y) 
= \mathcal{M}(t_0, H, \sigma) \exp\left( -\frac{ \left| Z^{x,y,\sigma,\omega^-} \right|^2 }{ \rho_H^2 t_0^{2H} } \right) G_{t_0}^{\sigma}(l^{x, \sigma, \omega^-}; y).
\end{align}
Using~\eqref{sdse52s} entails~\eqref{NM12a}. Additionally, based on the definition of the set \( T_{\theta,\kappa} \), one can easily derive~\eqref{NM12a2}.
\qed
The following elementary lemma is needed in the proof of Proposition \ref{INFORM}.
\begin{lemma}\label{uja63sfdaa}
Let $\xi\in \R^d$ and \( |\xi| \leq M \). Then, for every \( \epsilon > 0 \), there exists a constant \( \epsilon^M > 0 \) such that for every \( \zeta \in \mathbb{R}^d \),
\begin{align*}
-|\xi + \zeta|^2 \leq \epsilon|\zeta|^2 - \epsilon^M |\xi|^2.
\end{align*}
\proof
The proof is straightforward, so we only outline the main idea. We proceed by contradiction. Assume that for some fixed \( \epsilon > 0 \), and for every \( n \in \mathbb{N} \), there exists a pair \( (\xi^n, \zeta^n) \) such that \( |\xi^n| \leq M \) and 
\begin{align*}
-|\xi^n + \zeta^n|^2 > \epsilon |\zeta^n|^2 - \frac{1}{n} |\xi^n|^2.
\end{align*}
Then, thanks to the uniform boundedness of \( (\xi^{n})_{n \geq 1} \), we can conclude that the sequence is precompact and obtain a contradiction.
\qed
\end{lemma}
%\section{{Proof of Theorem~\ref{ATTRRP}}}\label{ATTRATT}
\section{Fractional calculus}\label{Fractional}
In this section, we collect technical results related to fractional calculus which are used in Appendix \ref{AAZAZA}.
\begin{definition}
Let \( 0 < \alpha < 1 \). The \emph{right-sided Riemann-Liouville fractional integral} is defined by
\begin{align}\label{fractional integral}
\begin{split}
\mathcal{J}^{\alpha}_{+} : L^{1}([0,T], \mathbb{R}^d) &\longrightarrow L^{1}([0,T], \mathbb{R}^d), \\
(\mathcal{J}^{\alpha}_{+} f)(t) &:= \frac{1}{\Gamma(\alpha)} \int_{0}^{t} f(s) (t-s)^{\alpha - 1} \, \mathrm{d}s.
\end{split}
\end{align}
Moreover, it holds that
\begin{align}\label{YAHSN85asa}
\|\mathcal{J}^{\alpha}_{+} f \|_{L^\infty([0,T],\mathbb{R}^d)} \leq \frac{T^{\alpha}}{\Gamma(\alpha)} \| f \|_{L^\infty([0,T],\mathbb{R}^d)}.
\end{align}
%where the uniform norm \( \|\cdot\|_{\infty;[0,T]} \) is understood in the almost everywhere sense. everybody knows what's  esssup
\end{definition}
It is well known that \( \mathcal{J}^\alpha_+ \) increases the H\"older regularity of a function. More precisely, we have the following statement.
\begin{lemma}\label{AS85sdf}
Let \( \beta \in [0, 1] \) and \( \alpha \in (0, 1) \) such that \( \alpha + \beta \neq 1 \). Then, the operator
\begin{align*}
\tilde{\mathcal{J}}^{\alpha}_{+} : C^{0,\beta}([0,T],\mathbb{R}^d) 
&\longrightarrow C^{\lfloor \alpha+\beta \rfloor,\, \alpha+\beta - \lfloor \alpha+\beta \rfloor}([0,T],\mathbb{R}^d), \\
(\tilde{\mathcal{J}}^{\alpha}_{+}f)(t) 
&:= (\mathcal{J}^{\alpha}_{+} f)(t) - \frac{f(0)}{\Gamma(1+\alpha)} t^\alpha,
\end{align*}
is continuous. In particular, there exists a constant \( C_{\alpha, \beta} > 0 \) such that for \( t \in [0, T] \) and every \( f \in C^{0,\beta}([0,T],\mathbb{R}^d) \), we have  
\[
\big| (\tilde{\mathcal{J}}^{\alpha}_{+}f)(t) \big| \leq C_{\alpha, \beta} \, t^{\alpha+\beta} \| f \|_{C^{0,\beta}([0,T])}.
\]
\end{lemma}
\proof
We refer to \cite[Theorem 3.1 and Corollary~1 in Chapter~1]{SKM93}.
\qed
This entails the following corollary.
\begin{corollary}\label{IL9as}
Let \( \alpha, \beta \in (0, 1) \) and \( \alpha + \beta < 1 \). Then there exists a constant \( C_{\alpha, \beta} > 0 \) such that for all \( f \in C^{\beta}([0,T],\mathbb{R}^d) \) with \( f(0) = 0 \), we have 
\begin{align*}
\| \mathcal{J}^{\alpha}_{+} f \|_{\alpha + \beta;[0,T]} 
\leq C_{\alpha, \beta} \| f \|_{\beta;[0,T]}.
\end{align*}
\end{corollary}
\begin{comment}
conte\begin{remark}
In case that $\alpha+\beta =1$, then teh follwinh holds
\begin{align*}
\sup_{\substack{s,t \in [0,T]\\s \neq t}}
\frac{\lvert \mathcal{J}^{\alpha}_{+} f(t) - \mathcal{J}^{\alpha}_{+} f(s)\rvert}{\lvert t - s\rvert^{\alpha}},
\end{align*}

\end{remark}nt...
\end{comment}

The right-sided Riemann-Liouville fractional derivative can be defined as follows

\begin{definition}
Let \( \alpha \in (0,1) \) and define the domain
\[
D(\mathcal{J}^{-\alpha}_{+}) := \big\{ f \in L^1([0,T], \mathbb{R}^d) : \mathcal{J}_{+}^{1-\alpha} f \in AC([0,T], \mathbb{R}^d) \big\},
\]
where \( AC([0,T], \mathbb{R}^d) \) denotes the space of absolutely continuous functions on \([0,T]\).
For any \( f \in D(\mathcal{J}^{-\alpha}_{+}) \), the \emph{right-sided Riemann-Liouville fractional derivative} of order \( \alpha \) is defined for almost every \( t \in [0,T] \) by
\[
(\mathcal{J}^{-\alpha}_{+} f)(t) := \frac{\mathrm{d}}{\mathrm{d}t} \mathcal{J}_{+}^{1-\alpha} f(t) = \frac{1}{\Gamma(1-\alpha)} \frac{\mathrm{d}}{\mathrm{d}t} \int_0^t f(s) (t-s)^{-\alpha} \, \mathrm{d}s.
\]
\end{definition}
From Theorem \ref{AS85sdf} we derive the following result.
\begin{lemma}\label{ASUj144as0}
For every \( \beta \in (\alpha, 1] \), we have the inclusion
\[
C^{0,\beta}([0,T], \mathbb{R}^d) \subset D(\mathcal{J}^{-\alpha}_{+}).
\]
Moreover, there exists a \((\beta - \alpha)\)-H\"older continuous function \(\tilde{\mathcal{J}}^{-\alpha}_{+} f\) and a constant \({C}_{\alpha,\beta} > 0\) such that for every \( t \in [0,T] \),
\begin{align}\label{ASUj144as}
\begin{split}
&(\mathcal{J}^{-\alpha}_{+} f)(t) 
= \tilde{\mathcal{J}}^{-\alpha}_{+} f(t) 
+ \frac{f(0)(1-\alpha)}{\Gamma(2-\alpha)} t^{-\alpha}, \\
&\left\| \tilde{\mathcal{J}}^{-\alpha}_{+} f \right\|_{C^{0,\beta-\alpha}([0,T])} 
\leq {C}_{\alpha,\beta} \| f \|_{C^{0,\beta}([0,T])}.
\end{split}
\end{align}
%	\textcolor{red}{
Furthermore, the following estimate holds:
\begin{align}\label{ASUj144as3}
\left\| \tilde{\mathcal{J}}^{-\alpha}_{+} f \right\|_{\infty;[0,T]} \leq  \tilde{C}_{\alpha,\beta} T^{\beta-\alpha} \| f \|_{C^{0,\beta}([0,T])},
\end{align}
for some constant \( \tilde{C}_{\alpha,\beta} > 0 \).
%}
\end{lemma}
\proof
Since \( \beta + 1 - \alpha > 1 \), Theorem~\ref{AS85sdf} implies that
\begin{align*}
\mathcal{J}^{1-\alpha}_{+} f \in C^{1,\,\beta-\alpha}([0,T], \mathbb{R}^d),
\end{align*}
and consequently, \( f \in D(\mathcal{J}^{-\alpha}_{+}) \). Again, from Theorem~\ref{AS85sdf}, we have the decomposition
\begin{align}\label{ILasse}
(\mathcal{J}^{1-\alpha}_{+} f)(t) = \tilde{\mathcal{J}}^{1-\alpha}_{+} f(t) + \frac{f(0)}{\Gamma(2-\alpha)} t^{1-\alpha}, 
\end{align}
where the function \(\tilde{\mathcal{J}}^{1-\alpha}_{+} f\) is such that its derivative
\[
\tilde{\mathcal{J}}^{-\alpha}_{+} f(t) := \frac{\mathrm{d}}{\mathrm{d}t} \tilde{\mathcal{J}}^{1-\alpha}_{+} f(t)
\]
is H\"older continuous of order \( \beta - \alpha \), and
\begin{align*}
|(\tilde{\mathcal{J}}^{1-\alpha}_{+}f) (t)| \leq C_{1-\alpha, \beta} \, t^{1-\alpha + \beta}\| f \|_{C^{0,\beta}([0,T])}.
\end{align*}
This implies that \( \frac{\mathrm{d}}{\mathrm{d}t} \big( \tilde{\mathcal{J}}^{1-\alpha}_{+} f \big)(0) = 0 \).
Differentiating the decomposition in~\eqref{ILasse} yields
\[
(\mathcal{J}^{-\alpha}_{+} f)(t) = \tilde{\mathcal{J}}^{-\alpha}_{+} f(t) + \frac{f(0)(1-\alpha)}{\Gamma(2-\alpha)} t^{-\alpha}.
\]
The second estimate in~\eqref{ASUj144as} follows directly from the regularity result in Theorem~\ref{AS85sdf} applied to \( \mathcal{J}^{1-\alpha}_{+} \). Moreover, the bound in~\eqref{ASUj144as3} follows directly from~\eqref{ASUj144as}, using the fact that \( \frac{\mathrm{d}}{\mathrm{d}t} \big( \tilde{\mathcal{J}}^{1-\alpha}_{+} f \big)(0) = 0 \).
\qed
\bibliographystyle{alpha}
\bibliography{refs}

\end{document}